%% file: paper.tex
\documentclass[12pt,a4paper]{amsart}
\renewcommand{\baselinestretch}{1.1}
\usepackage{amsmath}
\usepackage{amssymb}
\usepackage{latexsym}
\usepackage{xypic}
\usepackage{amscd,amsthm}

\usepackage{amsmath}
\usepackage{color}
\usepackage{bbold}
\usepackage{amsthm}
\usepackage{amssymb}
\usepackage{amsfonts}
\usepackage{epsfig} 
\usepackage{multirow}
\usepackage{rotate} 
\usepackage{graphicx}
\usepackage{fancyhdr}
\usepackage{wrapfig}

\usepackage{textcomp}
\usepackage{graphicx}
\usepackage{psfrag}
\usepackage{epic}
\usepackage{eepic}
\usepackage[matrix,arrow,curve]{xy}
\input{epsf}

\newtheorem{prop}{Proposition}[section]

\newtheorem{theo}[prop]{Theorem}
\newtheorem{conj}[prop]{Conjecture}
\newtheorem{rem}[prop]{Remark}

\newtheorem{coro}[prop]{Corollary}
\newtheorem{lem}[prop]{Lemma}
\newtheorem{exam}[prop]{Example}

\theoremstyle{definition}
\newtheorem{dfn}[prop]{Definition}
\theoremstyle{remark}

\theoremstyle{remark}

 \DeclareMathOperator{\inter}{\mathrm{int}}

 \DeclareMathOperator{\conv}{\mathrm{conv}}

\DeclareMathOperator{\nv}{{{\mathrm Vol}_{\mathbb N}}}

\DeclareMathOperator{\N}{{\mathbb N}}

\DeclareMathOperator{\A}{{\mathbb A}}
\DeclareMathOperator{\R}{{\mathbb R}}
\DeclareMathOperator{\Z}{{\mathbb Z}}
\DeclareMathOperator{\PP}{{\mathbb P}}

\DeclareMathOperator{\Hom}{{\mathrm Hom}}

\newtheorem*{theorem*}{Theorem}
\newtheorem*{corollary*}{Corollary}

\begin{document}
\title[Gorenstein Toric Del Pezzo Varieties]{Classification 
of Gorenstein Toric Del Pezzo Varieties 
in arbitrary dimension}

\author[Victor Batyrev]{Victor Batyrev}
\address{Mathematisches Institut, Universit\"at T\"ubingen, 
Auf der Morgenstelle 10, 72076 T\"ubingen, Germany}
\email{ victor.batyrev@uni-tuebingen.de}

\author[Dorothee Juny]{Dorothee Juny}
\address{Mathematisches Institut, Universit\"at T\"ubingen, 
Auf der Morgenstelle 10, 72076 T\"ubingen, Germany}
\email{dorothee@juny.de}

\begin{abstract}
A $n$-dimensional Gorenstein toric Fano variety $X$ is called 
Del Pezzo variety if the anticanonical class $-K_X$ is 
a $(n-1)$-multiple of a Cartier divisor. Our purpose 
is to give a complete biregular classfication of Gorenstein toric Del Pezzo 
varieties in arbitrary dimension $n \geq 2$. We show that  up to 
isomorphism  there exist 
 exactly 37 Gorenstein toric Del Pezzo
 varieties of dimension $n$  
which are not cones over $(n-1)$-dimensional 
Gorenstein toric Del Pezzo varieties. Our results are closely related 
to the classification of all Minkowski sum decompositions of reflexive 
polygons due to Emiris and Tsigaridas and to 
the classification up to deformation 
of $n$-dimensional  almost Del Pezzo manifolds 
obtained by Jahnke and Peternell.
\vspace{-0.5ex}
\end{abstract}

\dedicatory{Dedicated to the memory of Professor V. A Iskovskih}

\maketitle

\section*{Introduction}

A projective algebraic  variety $X$ is called {\sl Gorenstein Fano variety} 
if the anticanonical divisor $-K_X$ is an  ample 
Cartier divisor. A $n$-dimensional  Gorenstein  Fano variety is called  
{\sl Gorenstein Del Pezzo variety}
if 
$$-K_X =
(n - 1)L$$ for some ample  Cartier divisor $L$ on $X$. 
Smooth Del Pezzo $n$-folds 
have been classified up to deformation by Iskovskih and Fujita 
\cite{Isk77,Isk78,Isk80,Fuj80,Fuj81,Fuj84}. 
A classification of Gorenstein Del Pezzo surfaces is known due to Demazure  
\cite{Dem80}. 

All Gorenstein Del Pezzo varieties 
 of dimension $n \geq 4$ have been classified up to deformation by Fujita 
\cite{Fuj90a,Fuj90b}. 
A smooth $n$-fold $X$ is called an {\sl almost 
Del Pezzo manifold} if there exists a $K_X$-trivial 
birational morphism 
$\varphi\,: \, X \to X'$ where $X'$ is a Gorenstein  Del Pezzo variety. 
It is well-known that if $n =2$ 
then the minimal resolution $\widehat{X}$ of 
a  Gorenstein Del Pezzo surface $X$ 
is always an almost Del Pezzo surface 
\cite{DuV34}. Such a surface $\widehat{X}$ 
can be  obtained from $\PP^2$ or $\PP^1 \times \PP^1$ by blowing up 
points which are not necessary in general position. 
Recently all almost Del Pezzo manifolds of arbitrary dimension $n$ were 
classified up to deformation 
by Jahnke and Peternell \cite{JP08}. 

In this paper we consider $n$-dimensional 
Gorenstein Fano varieties $X$ which are 
{\sl toric varieties}, i.e.,  $X$ contains  a
$n$-dimensional algebraic torus  $T$ as an open dense subset  
such that  the group 
multiplication $T \times T \to T $
extends to  a regular 
action $T \times X \rightarrow X$ 
\cite{Oda88,Ful93}. Let $M \cong \Z^n$ be 
the lattice of characters 
of $T$ and $M_{\R}:= M \otimes \R$ the corresponding real vector space. 
A natural $T$-action on  the space  
$H^0(X, {\mathcal O} (-K_X))$ splits into a finite 
direct sum of $1$-dimensional weight subspaces  
parametrized by a finite subset of characters 
$\{ \chi_i \} \subset M$. 
In this situation, the convex hull $\Delta$ of the set $\{ \chi_i \}$  
is a special $n$-dimensional polytope in $M_{\R}$ which 
is called {\sl reflexive 
polytope} \cite{Bat94}. The Gorenstein toric Fano variety $X$ is uniquely 
determined 
up to an isomorphism of the $n$-dimensional lattice $M$ 
by the corresponding reflexive polytope 
$\Delta$. 
It is known that for  any given dimension $n$ 
there exist only finitely many $n$-dimensional 
reflexive polytopes up to a lattice isomorphism. 
This implies 
that there exist only finitely many Gorenstein toric Fano varieties of
dimension $n$ up to biregular isomorphism (see \cite{Bat82}).
If $N:= {\rm Hom}(M, \Z)$ 
is the dual lattice and $N_{\R} = {\rm Hom}(M_{\R}, \R)$ 
the dual real vector space, then 
for every reflexive polytope $\Delta \subset M_{\R}$ there exists
a dual reflexive polytope $\Delta^* \subset N_{\R}$. Using an  
isomorphism 
$N \cong M$, we obtain that up to isomorphism all 
reflexive polytopes 
satisfy a nice combinatorial duality which plays important role in  
Mirror Symmetry  \cite{Bat94} (see also \cite{CK99}). 

If $X$ is a $n$-dimensional toric variety such that 
$-K_X = rL$ for an ample divisor 
$L$ then, up to a translation 
by an element of $M$, 
the reflexive polytope $\Delta$ is isomorphic  
to a $r$-multiple of another lattice polytope $P$, 
i.e., $\Delta \cong r P$ where all vertices of $P$ are contained in $M$. 
Such polytopes $P$ are called  
{\sl Gorenstein polytopes of index $r$} \cite{BR07,BN08}. For example, 
the standard $n$-dimensional basic lattice  simplex is a Gorenstein 
polytope of index $n+1$ and the standard $n$-dimensional lattice 
unit cube is a Gorenstein polytope of index $2$.  
The number  $d:=  n+1 -r$  is called the {\sl  degree of} 
a $n$-dimensional Gorenstein polytope $P$ (see also \cite{Bat06}).
The name ``degree'' of a $n$-dimensional  lattice polytope 
$P$ is motivated by the following fact: 
if 
\[ E_P (t) := \sum_{k =0}^{\infty} |kP \cap M| t^k \]
is the  Ehrhart series of $P$, then 
$h^*_P(t):=(1-t)^{n+1} E_P(t)$ is a polynomial of degree $d$.

The purpose of our paper  is to give a complete classification 
 up to isomorphism of all Gorenstein toric Del Pezzo 
varieties in arbitrary dimension $n \geq 2$. This is equivalent to 
a classification of all $n$-dimensional Gorenstein polytopes of index 
$r =n-1$, or, equivalently, of degree $2$. 
In the case $n =2$, Gorenstein polygons of index $1$ are exactly 
reflexive polygons and their  classification is well-known (it was 
obtained independently in  
\cite{Bat85,Rab89,Koe91}).  
There exist exactly $16$ reflexive 
polygons (see Figure 1). 

 \begin{figure}[h] \label{16polygons}
      \centering {\includegraphics[{height=10cm}]{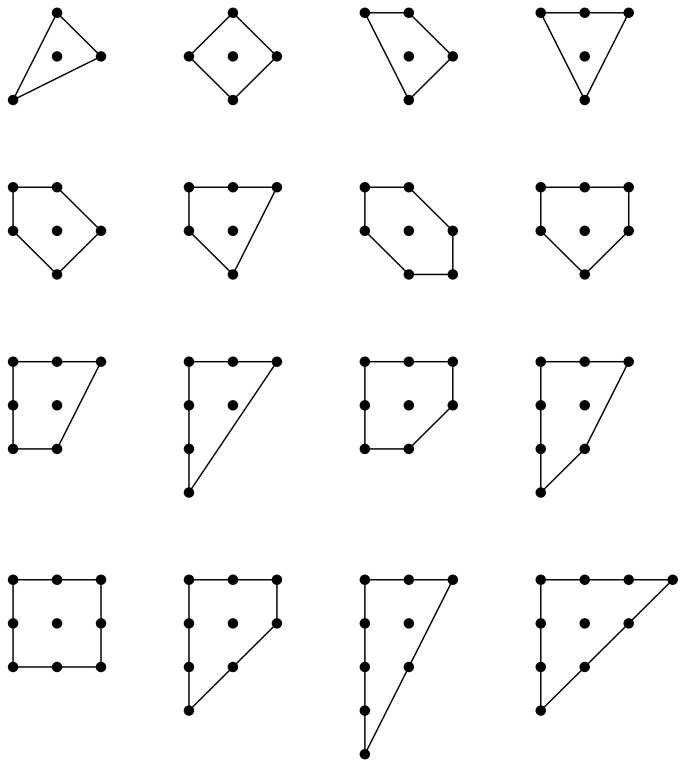}}
      \caption{All $16$ reflexive polygons, up to
        isomorphism}\label{LatticePolygons}
    \end{figure}

It is important to note  that if $X$ is a $n$-dimensional Gorenstein 
Del Pezzo variety then the  cone $X'$ 
over $X$ is a $(n+1)$-dimensional Gorenstein Del Pezzo 
variety. If in addition $X$ is toric, then $X'$ is also toric 
and the $(n+1)$-dimensional 
 Gorenstein polytope $P'$ corresponding to $X'$ is a lattice 
pyramid 
$\Pi(P)$ over the $n$-dimensional Gorenstein polytope $P$ corresponding 
 to $X$. This shows that it is  
enough to classify only those Gorenstein polytopes of index $ r = n-1$ 
which are not pyramids over low dimensional ones. 


Using a computer program, $3$- and
$4$-dimensional reflexive polytopes have been classified by Kreuzer
and Skarke \cite{KS98,KS00}. 
Their computer calculation resulted in 4319
polytopes in dimension $3$ and $473\,800\,776$ polytopes in
dimension $4$. Using the lists of $3$- and 
$4$-dimensional polytopes, Kreuzer has found $31$  
$3$-dimensional Gorenstein polytopes of index $2$ and 
$36$ $4$-dimensional Gorenstein polytopes of index  $3$.
In this paper, we give an independent proof of this fact  
without using computer classifications.

Our complete classification 
of $n$-dimensional Gorenstein polytopes of degree $2$ (i.e., of 
index $r = n-1$) gives the following more precise result: 

\begin{table}[h]
  \centering
  \begin{tabular}[h]{c|c|c}
  Dimension $n$ & Number of polytopes & Number of non-pyramids\\ \hline
  $2$       & $16$                & $16$     \\\hline
  $3$       & $31$                & $15$\\\hline
  $4$       & $36$                & $5$\\\hline
  $5$       & $37$                & $1$\\\hline
  $\geq 6$  & $37$                & -
\end{tabular}
  \label{tab:classresults}
\end{table}

We remark that the classification of $n$-dimensional Gorenstein polytopes
of index $r = n-2$ is a much harder problem. This problem in case $n =3$ 
is equivalent to the classification of all $3$-dimensional reflexive 
polytopes \cite{KS98}. Using the computer classification 
of $4$-dimensional reflexive polytopes \cite{KS00}, Kreuzer has found 
exactly 
 $5363$ reflexive polytopes $\Delta$ of dimension $4$ which are isomorphic 
to $2P$ for some lattice polytope $P$. This shows that   there exist 
exactly $5363$ Gorenstein polytopes $P$  
of dimension $4$ and  index  $2$. In particular, there exist exactly 
$1044 = 5363- 4319$ Gorenstein $4$-dimensional polytopes of 
index $2$ which are not pyramids over $3$-dimensional reflexive polytopes. 
However, all these  facts seem to be very 
difficult to check without computer.

It follows from a  recent result of Haase, Nill and Payne  
\cite[Cor. 3.2]{HNP08} 
together with the main result in \cite{Bat06}
that 
every  $n$-dimensional 
Gorenstein polytope $P$ of degree $d$ is a pyramid if 
\[ n \geq  4d { 2d + V-1 \choose 2d }      ,\]
where 
\[ V = (2d-1)^{2d-1} \cdot \left( (2d-1)! \right)^{2d} \cdot 14^{(2d-1)^2 
\cdot 2^{2d}}. \]   
In \cite[Prop. 1.5]{Nil08} Nill has shown 
that the same statement holds true already 
for 
\[  n \geq (V-1)(2d+1). \]
However, these inequalities are still too 
far from being sharp. 
It would be nice to know exactly  the maximal dimension $n$ 
(depending on $d$) such 
that every $n$-dimensional 
Gorenstein polytope of degree $d$ is a pyramid. 
In this connection  we suggest the following: 

\begin{conj} \label{conj-d}
Every $3d$-dimensional Gorenstein polytope $P$ of degree $d$ is 
a pyramid. Moreover, there exists up to an isomorphism a 
unique $(3d-1)$-dimensional Gorenstein polytope $\Theta_d$ of degree $d$ 
which is not a pyramid. 
\end{conj} 

Conjecture  \ref{conj-d} is easy to check for $d =1$. Our classification 
verifies Conjecture   \ref{conj-d} in the case  $d =2$, but  it
 remains  open for  $d \geq 3$. A precise description of   
the special Gorenstein  polytope $\Theta_d$ for arbitrary degree $d$ 
is  given in  \ref{special-d}.

\medskip
\bigskip

Let us summarize the   ideas used in our   classification 
of $n$-dimensional Gorenstein polytopes of degree $2$. 

First of all we use  a natural partial order on the set 
of $n$-dimensional Gorenstein polytopes of degree $2$: for two lattice  
 polytopes $P_1$ and $P_2$, we write 
$P_1  \preccurlyeq   
P_2$ if $P_1$ is isomorphic to a lattice subpolytope of $P_2$. 
We say that a $n$-dimensional Gorenstein polytope $P$ of degree $2$ is 
{\sl minimal} if there does not exist  a   
$n$-dimensional 
Gorenstein subpolytope $P' 
 \preccurlyeq P$ of degree $2$ such that $P' \neq P$. 
For  $n$-dimensional Gorenstein polytopes of fixed degree $d$ there 
exists a combinatorial duality 
 which generalize
the polar duality for $n$-dimensional reflexive polytopes \cite{BN08}. 
Using this duality, we obtain that if $P$ is a minimal  
$n$-dimensional Gorenstein polytope of degree $2$, then the dual 
Gorenstein polytope $P^*$ is a {\sl maximal} one, i.e., $P^*$ is not 
contained in a strictly larger $n$-dimensional Gorenstein polytope 
of degree $2$. Since every $n$-dimensional Gorenstein polytope is 
contained in a  maximal one, in the  classification of  
all $n$-dimensional Gorenstein polytopes of degree $2$ the first step 
is to find  all minimal ones among them.  After that, using the 
duality $P \leftrightarrow P^*$, 
we immediately obtain all maximal $n$-dimensional 
Gorenstein polytopes 
of degree $2$. Finally, all remaining Gorenstein polytopes can be found 
as lattice subpolytopes of the maximal ones. 

Let $P$ be a $n$-dimensional lattice polytope. The integer $$\nv(P):= 
n! ({\rm volume} \;{\rm of} \; P)$$ 
we call the {\sl lattice normalized volume of} $P$.  
In the classification of minimal Gorenstein polytopes of degree $2$ the 
following statement  is very 
important: if $P$ is minimal, then 
$\nv(P) \leq 4$. 
The proof of this technical  statement consists of several steps and 
uses  the complete classification of
$n$-dimensional lattice polytopes whose $h^*$-polynomial 
is of degree $\leq 1$
\cite{BN07}.
As we already remarked, the classification of all $n$-dimensional Gorenstein 
poytopes $P$ of degree $2$ is
equivalent to the  classification of  all $n$-dimensional lattice
polytopes $P$ having  the  Ehrhart series 
\[ E_P (t) = \frac{1 + a_Pt + t^2}{(1-t)^{n+1}}, \quad a_P \in \Z_{\geq 0},\] 
where  $a_P = {\rm Vol}_{\N}(P) -2$. If $X_P$ is the corresponding 
$n$-dimensional Gorenstein toric Del Pezzo variety and $-K_{X_P} = (n-1)L$
for some Cartier divisor $L$, then the integer 
$a_P+2 =  {\rm Vol}_{\N}(P)$ equals the intersection number  $L^n$ and
the  Ehrhart series  $E_P (t)$ is exactly the Hilbert-Poincar{\'e} 
series of the homogeneous coordinate ring of $X_P$:
$$S_P := \bigoplus_{ k =0}^{\infty} H^0(X_P, {\mathcal O}(kL)).$$ 
The graded Gorenstein ring $S_P$ can be seen as  a semigroup algebra of the 
set of lattice points in the $(n+1)$-dimensional cone $C_P$ over the 
Gorenstein polytope $P$. If $P$ is a Gorenstein polytope such that 
 $\nv(P) \leq 4$ (i.e. $a_P \leq 2$), then one can apply  
algebraic ideas from \cite{Bat06} and classify all possible 
Gorenstein rings $S_P$ (and the corresponding Gorenstein polytopes $P$) 
by an  enumeration of all binomial relations 
among the minimal generators of $S_P$. 

Let us say some words about the way of 
presentation  of the results of our 
classification. It is not difficult to draw pictures of 
all $3$-dimensional Gorenstein polytopes of degree $2$. However, the same way 
of presentation can not be applied in dimension $n \geq 4$. 
It was shown in \cite[Th. 2.6]{BN08} that if a $d$-dimensional 
reflexive polytope  $\Delta$ is a Minkowski sum of $r$ lattice polytopes 
$\Delta_1, \ldots, \Delta_r$:
\[ \Delta = \Delta_1 + \cdots + \Delta_r, \]
then the $(d+ r-1)$-dimensional Cayley polytope $P:= \Delta_1 * \cdots * 
\Delta_r$ is a Gorenstein polytope of degree $d$ (or, of index $r$). 
In particular, a Minkowski sum decomposition $\Delta = \Delta_1 + 
\cdots + \Delta_r$ of a $2$-dimensional reflexive 
polytope $\Delta$ defines a $(r+1)$-dimensional Gorenstein 
polytope  $P= \Delta_1 * \cdots * 
\Delta_r$ of degree $2$. 
The classification of all  Minkowski sum decompositions 
of reflexive polygons was recently obtained by Emiris and Tsigaridas 
\cite{ET06}. 
Our classification shows that all Gorenstein polytopes of degree $2$ and 
of dimension $ n \geq 4$ are 
Cayley polytopes corresponding to  different  Minkowski sum decompositions 
of $2$-dimensional reflexive polytopes.  The same statement holds true 
for almost all (except one) 
$3$-dimensional Gorenstein polytopes of degree $2$.

\bigskip

The paper is organized as follows: 

In Section  \ref{basics} we review 
basic definitions and properties of Gorenstein  polytopes. 
The combinatorial duality of Gorenstein polytopes $P \leftrightarrow P^*$ 
and their constructions 
as Cayley polytopes corresponding to Minkowski summands of reflexive polygones 
will be of our primary interest.

In Section 2 we investigate the graded Gorenstein $K$-algebra  
$S_P$ associated with a 
Gorenstein polytope $P$ of degree $d$. Let $A_P:= S_P/\langle y \rangle$
be the Artinian 
Gorenstein $K$-algebra  obtained from $S_P$ 
as quotient modulo an  ideal $\langle y \rangle $ generated by a 
regular sequence 
of homogeneous elements $y_1, \ldots, y_n$ of degree $1$. 
If the degree of $P$ is $2$ then $1 + a_Pt + t^2$ is 
the  
Hilbert-Poincar\'e polynomial of $A_P$ and the   
graded Gorenstein $K$-algebras $A_P$ can be classified up to 
isomorphism (see \ref{Gor1a1}). 
This classification allows to determine all possible binomial relations 
in the $K$-algebra $S_P$. If $a_P \leq 2$, then we find  a 
complete list of Gorenstein polytopes $P$ which are not pyramids 
(see \ref{class-aleq2}).

In Section 3 we describe general combinatorial properties 
of Gorenstein polytopes of degree $2$. In particular, we 
prove  some useful  equalities for $3$-dimensional  Gorenstein 
polytopes of degree $2$. 
It is well known  that if $\Delta$ and $\Delta^*$ are dual to each other 
reflexive polygons then 
\[ \nv( \Delta) +  \nv(\Delta^*)  =12. \]
This remarkable 
property of reflexive polygons was recently 
investigated by many mathematicians 
\cite{PR00,HS02,RST05}. We show that there exists a similar relation 
between the number 
12  and dual to each other $1$-dimensional faces $E \subset P$, 
$E^* \subset P^*$ of the 
 $3$-dimensional Gorenstein polytopes $P$, $P^*$ of degree  $2$ 
\[\sum_{E \subset  P, \dim E=1} \nv(E) \cdot
  \nv(E^{\ast}) = 12. \]
(see \ref{12-3dim}). 

In Section 4 we prove that if a Gorenstein polytope $P$ of degree $2$ is
minimal then  $\nv(P) \leq 4$. As it was explained above, together with 
the classification in Section 2 this result easy implies a complete 
 classification of Gorenstein polytopes of degree $P$. Here we present 
the whole list of $d$-dimensional Gorenstein polytopes of
degree  $2$ which are not pyramids (including some of their
combinatorial properties).

In Section \ref{almostDP} we compare 
our combinatorial method with 
the birational method of Jahnke and Peternell used in their 
classification of almost Del Pezzo manifolds  \cite{JP08}. 
\bigskip

{\bf Acknowlegements}. We are very  grateful to J\"urgen Hausen, Priska 
Jahnke and Benjamin Nill for their interest in our work and helpful 
comments.

\section{General properties of  Gorenstein polytopes}
\label{basics}

By a {\sl lattice} $M$ of dimension $n$ we mean a free abelian group
of rank $n$, $M \cong \mathbb{Z}^n$. We also consider the  
{\sl dual lattice} $N =
M^{\ast} = \Hom (M, \mathbb{Z}) \cong \mathbb{Z}^n$ together with the  
canonical pairing \[ \langle \:\: , \:\: \rangle
: M \times N \rightarrow \mathbb{Z}. \] Further, we define 
$M_{\mathbb{Q}} := M \otimes_{\mathbb{Z}} \mathbb{Q}$, $N_{\mathbb{Q}} := N 
\otimes_{\mathbb{Z}} \mathbb{Q}$ together  
with the $\mathbb{Q}$-bilinear pairing
$\langle \:\: , \:\: \rangle : M_{\mathbb{Q}} \times N_{\mathbb{Q}}
\rightarrow \mathbb{Q}$. 
The real vector spaces 
 $M_\mathbb{R}$ and  $N_\mathbb{R}$  are defined analogously. 
A {\sl polytope} $P$ in $M_{\mathbb{R}}$ is the convex hull
of a finite subset in $M_{\mathbb{R}}$. We say that $P$ is a 
{\sl lattice polytope} if all 
its vertices lie in $M$. Further, a {\sl face} of $P$ is a subset
$F \subseteq P \subset M_{\R}$ which minimizes on $P$ some linear function 
$f \in N_{\R}$,
and a {\sl facet} of $P$ is a codimension-$1$ face of $P$.
By $\partial P$ we denote the boundary of the polytope $P$ and
by $\inter(P )$ the interior
$P \backslash \partial P$ of $P$.

Let $P$ be a $n$-dimensional lattice polytope in $M_{\mathbb{R}}$.  
For any nonnegative integer $k$ the number $i(P,k) :=
|kP \cap M|$  is a polynomial of degree $n$ in $k$  \cite{Ehr67}. 
This polynomial is   called the
{\sl Ehrhart polynomial} of $P$. Its values in negative integers 
are defined by the formula 
\[ i(P, -k) = (-1)^{n} | \inter (kP) \cap M|.\]
Moreover, one has 
\[ E_{P}(t) := \sum_{k \geq 0} i (P,k) t^k = \frac{h_0^* + h_1^* t +
  \cdots + h_n^* t^n}{(1-t)^{n+1}} \] 
for some nonnegative integers $h_i^*$ \cite{Sta80}.
The Ehrhart polynomial  $i(P,k)$ can be
recovered from the integers $h_i^*$ by the formula: 
\[ i(P,k) = \sum_{j= 0}^n h_j^* {n+k-j \choose n}.\]

\begin{dfn}
The polynomial
\[ h^{\ast}_{P} (t) := (1-t)^{n+1} E_P(t) \sum_{k \geq 0} | kP \cap
 M| t^n \] is called the $h^{\ast}$-{\sl polynomial} of a $n$-dimensional 
lattice polytope $P$. Its degree is called 
the {\sl degree of } $P$:
$\deg P := \max \{ i : h^*_i \not= 0\},  \;\; 0 \leq  \deg P \leq \dim P$. 
The vector 
$ h^*(P) = (h_0^*, h_1^*, \dots, h_{\deg P}^*)$ is called the $h^*$-{\sl
    vector} of $P$.
\end{dfn}

\begin{rem}\label{monotonyhvector}
{\rm   It was proved by  Stanley in \cite{Sta80} that the numbers 
$h^*_i = h^*_i(P)$ are monotone functions of $P$: 
if $P$ and $P'$ are two lattice polytopes and
  $P' \subseteq P$, then $h^*_i (P') \leq h_i^*(P)$ for all $i$.}
\end{rem}

\begin{dfn}\cite[Def. 2.1]{BN07}
  We call a $n$-dimensional lattice polytope $P$ ($n \geq 1$) a {\sl
    Lawrence prism with heights} $\theta_1,\dots,\theta_n$ if there exists a 
lattice basis $e_1,  \dots, e_n$ of $M$ and non-negative integers
  $\theta_1,\dots,\theta_n$ such that
  \[ P = \conv \{ 0, \theta_1 e_n, e_1, e_1 + \theta_2 e_n,  
\dots, e_{n-1}, e_{n-1} +
  \theta_{n} e_n  \}.\]
  The lattice vector $e_n$ is called the {\sl direction of} $P$.
\end{dfn}

\begin{dfn}\cite[Def. 2.2]{BN07}
  We call a $n$-dimensional lattice polytope $P$ ($n \geq 2$) {\sl
    exceptional} if there exists a lattice basis $e_1, \dots, e_n$ of
  $M$ such that
  \[ P = \conv \{0, 2 e_1, 2 e_2, e_3, \dots,
  e_n \}. \]
\end{dfn}

The following statement was proved  in \cite{BN07}:  

\begin{theo} \label{deg1}
Every  $n$-dimensional
  lattice polytope $P$ of degree $\leq 1$ is either a basic simplex, 
or an exceptional simplex, a  or Lawrence prism such that 
$\theta_1 + \cdots + \theta_n \geq 2$. The corresponding $h^*$-polynomial 
is equal respectively to $1$,  $1 + 3t$,  
$1 + (\theta_1 + \cdots + \theta_n -1)t$. 
\end{theo}

\begin{dfn}
  For any $r$ convex polytopes $\Delta_1, \dots, \Delta_r$ in 
$M_{\R}$, their {\sl Minkowski sum} is defined as
  \[ \Delta := \Delta_1 + \cdots + \Delta_r = \{ x_1 + \cdots + x_r
  \mid x_i \in \Delta_i\}. \] 
The polytopes $\Delta_i$ $( i=1, \ldots, r)$ are called {\sl 
Minkowski  summands}
  of $\Delta$. Here we do not require all polytopes $\Delta_1,
  \dots, \Delta_r$ to have the maximal dimension $n = \dim M_{\R}$.
Define the lattice $\widetilde{M} := M \oplus \mathbb{Z}^r = M \oplus
\mathbb{Z} e_1 \oplus \cdots \oplus \mathbb{Z} e_r$, where $\{ e_1, \dots,
e_r\}$ form a lattice basis of $\mathbb{Z}^r$.
The {\sl Cayley polytope} $\Delta_1 * \cdots * \Delta_r$ associated with
  $\Delta_1, \dots, \Delta_r$ is the convex hull of $(\Delta_1, e_1),
  \dots, (\Delta_r, e_r)$ in $\widetilde{M_{\mathbb{R}}}$.
\end{dfn}

\begin{dfn}
  Let $\Delta \subset M_{\mathbb{R}}$ be a lattice polytope of dimension $n$. 
 We set $\Delta_0 := \Delta$ and choose  $r$ $0$-dimensional 
lattice polytopes  $\Delta_1,  \ldots, \Delta_r$ (points). 
Then the polytope $\Delta_0 * \Delta_1 * \cdots * 
\Delta_r$ is called {\sl
 $r$-fold pyramid over $\Delta$} and we denote it by $\Pi^r(\Delta)$. 
If $r =1$ we write simply $\Pi(\Delta)$ and call  $\Pi(\Delta)$ {\sl 
pyramid over $\Delta$}. 
\end{dfn}

\begin{rem} 
{\rm 
A Lawrence prism with  heights $\theta_1,\dots,\theta_n$ can be considered 
as Cayley polytope of $n$
$1$-dimensional polytopes  
$[0, \theta_1], \dots, [0, \theta_n] \subset \R$. 
A $n$-dimensional  exceptional
simplex is the $(n-2)$-fold pyramid over the 
lattice triangle ${\rm conv}\{0, 2e_1, 2e_2 \}$.}
\end{rem}

\begin{dfn} \label{dual-pol}
Consider a $d$-dimensional convex polytope $\Delta$ in $M_{\mathbb{R}}$
 that contains the zero point $0 \in M_{\mathbb{R}}$ in its interior.  
The polytope 
  \[ \Delta^* = \{ y \in N_{\mathbb{R}} \mid \langle x,y \rangle \geq
  -1 \quad \forall x \in \Delta\} \] is called {\sl dual polytope} of
  $\Delta$.

  A $d$-dimensional lattice polytope $\Delta$ in $M_{\mathbb{R}}$ is
  called {\sl reflexive} if $0 \in \Delta$ and the dual polytope
  $\Delta^*$ is again a
  lattice polytope. 
  If $\Delta$ contains a single interior lattice point $m$ such that
  $\Delta - m$ is reflexive, we say that $\Delta$ is {\sl reflexive
    with respect to $m$}.
\end{dfn}

\begin{rem} 
{\rm If $\Delta$ is reflexive polytope, then the dual 
polytope $\Delta^*$ is again reflexive and $(\Delta^*)^* = \Delta$. 
The duality $\Delta \leftrightarrow \Delta^*$ is very 
important for Mirror Symmetry \cite{Bat94,CK99}. }
\end{rem}

\begin{prop}\cite{Hib92} \label{ReflPolytope}
  Let $\Delta \subset M_{\mathbb{R}}$ be a lattice polytope. Then $\Delta$ is
  reflexive if and only if for every nonnegative integer $k$ 
the number of lattice points in $k\Delta$ equals
  the number of lattice points in the interior $\inter((k+1)\Delta)$ 
of $\Delta$.
\end{prop}

\begin{coro}
  A $n$-dimensional lattice polytope $\Delta$ is reflexive 
if and only if $$h_i^*(\Delta) =
  h_{n-i}^*(\Delta), \;\; 0 \leq i \leq {\frac{n}{2}}.$$
\end{coro}

\begin{dfn}
Let $r$ be a positive integer. 
A $n$-dimensional lattice polytope $P$ is called 
{\sl Gorenstein of index $r$} if  $ rP$ is a reflexive polytope with respect 
to some interior lattice point $m \in {\rm int}(rP)$. 
\end{dfn}

Analogous to  \ref{ReflPolytope}, one obtains:

\begin{prop}\label{Gor-Polytope}
  Let $P \subset M_{\mathbb{R}}$ be a $n$-dimensional 
lattice polytope. Then $P$ is
  a Gorenstein polytope of index $r$ 
 if and only if for every nonnegative integer $k$ 
the number of lattice points in $kP$ equals
  the number of lattice points in the interior $\inter((k+r)P)$ 
of $\Delta$. In this case, one has $\deg P = n +1 - r$. 
\end{prop}

\begin{coro} \label{h*dual}
  Let $P$ be a $n$-dimensional lattice polytope with $h^*$-polynomial
  $h^*_P(t) = \sum_i h^*_i t^i$ and let $d = \deg P$.  Then 
$P$ is a Gorenstein polytope of degree $d = n+1 -r$ if and only if 
$$h^*_i(P) =
  h^*_{d-i}(P),  \;\; 0 \leq i \leq \frac{d}{2}.$$
\end{coro}

\begin{dfn}  \cite{BB97}
{\rm 
Consider two $(n+1)$-dimensional lattices $\overline{M}:= M \oplus \Z$ and 
$\overline{N} := N \oplus \Z$ which are dual to each other with respect 
to the natural extension of the pairing  $\langle *, * \rangle \; :
{M} \times {N} \to \Z$ to $\overline{M} \times \overline{N} \to \Z$. 
For the latter we use the same notation  $\langle *, * \rangle$.  
Let   $\overline{M}_{\mathbb{R}}$ and 
$\overline{N}_{\mathbb{R}}$ be the corresponding real vector spaces.  

  A strongly convex $(n+1)$-dimensional 
cone $C \in \overline{M}_{\mathbb{R}}$ with the vertex 
$0 = -C \cap C$ is 
called {\sl Gorenstein cone} if
  there exists an element $n_{C} \in \overline{N}_{\mathbb{R}} $ 
such that $\langle e_i,
  n_{C} \rangle = 1 $ for all primitive lattice generators 
$e_i$ of $C$. Note that  $n_{C}$ is uniquely determined.

 If the cone $C$ is  Gorenstein, then the set
  \[ P_{C} = \{ x \in C \mid \langle x, n_{C}
  \rangle = 1\}
  \]
  is a convex $n$-dimensional 
lattice polyhedron which is called {\sl support} of
  $C$. 

On the other hand, if $P$ is a $n$-dimensional 
lattice polytope in $M_{\R}$, then 
\[ C_P:= \{ (\lambda x, \lambda)\; : \; \lambda \in \R_{\geq 0}, 
x \in P \}  \subset \overline{M}_{\mathbb{R}} \]  
is a Gorenstein cone with the support polytope $P$  
(here we have $n_{C_P} = (0,\cdots, 0,1)$).} 
\end{dfn}

\begin{dfn} \cite{BB97}
{\rm Define the {\sl dual cone} as 
\[ C^{\vee} : = \{  y \in \overline{N}_{\mathbb{R}}\; : \; 
\langle x, y \rangle \geq 0 \;\; \forall x \in C\}. \]
A $(n+1)$-dimensional 
Gorenstein cone $C$ is called {\sl reflexive} if its dual cone
$C^{\vee}$ is Gorenstein as well. In this case we denote by
$m_{C^{\vee}} \subset M$ the lattice point 
such that the
hyperplane $\langle m_{C^{\vee}}, y \rangle$  contains all lattice 
points generating  the dual cone
$C^{\vee}$.
For a reflexive Gorenstein cone, the uniquely defined positive integer
  \[ r_{C} = \langle m_{C^{\vee}}, n_{C} \rangle \] is
called  {\sl index} of $C$.}
\end{dfn}

\begin{rem} \label{refl-cone-dual}
{\rm It is very important for our purpose
that if $C$ is a  reflexive Gorenstein cone 
of index $r$ then   $C^{\vee}$ is also a reflexive Gorenstein cone 
of index $r$ . Moreover, a Gorenstein cone is reflexive of index $r$ 
if and only if its support is a Gorenstein polytope of index $r$
(see \cite[Prop. 1.11]{BN08}).}
\end{rem}

\begin{dfn}\label{DefGorPol}
Let $P \subset M_{\mathbb{R}}$ be a $n$-dimensional  
Gorenstein polytope of index $r$ and $C_P$ be  the corresponding 
$(n+1)$-dimensional 
reflexive Gorenstein cone of index $r$.  We define the {\sl dual 
$n$-dimensional Gorensten polytope} $P^*$ to be the support
of the dual reflexive Gorenstein cone $C^{\vee}_P$, i.e.,  
\[ C_{P^*} := C^{\vee}_P. \]
It follows immediately from the definition that $(P^*)^* =P$. 
\end{dfn}

\begin{rem} \label{dual-pyramid} 
{\rm It is easy to show that the duality for Gorenstein polytopes 
commutes with the pyramid construction, i.e., if $P$ and $P^*$ 
are dual to each other $n$-dimensional Gorenstein polytopes of index 
$r$, then $\Pi(P)$ and $\Pi(P^*)$ are dual to each other $(n+1)$-dimensional 
Gorenstein polytopes of index $r+1$.  In particular, 
a Gorenstein polytope $P$ of degree $d$ is a pyramid if and only if the dual
 Gorenstein polytope $P^*$ is a pyramid.
}
\end{rem}

\begin{theo}\cite[Th. 2.6]{BN08} \label{CayleyGorenstein}
 Let $\Delta_1, \dots, \Delta_r
  \subset M_{\mathbb{R}}$ be lattice polytopes. Assume that the 
 Minkowski sum $\Delta =  
  \Delta_1 + \cdots + \Delta_r$ is $n$-dimensional. Denote by 
  $\tilde{\Delta}$  the
  Cayley polytope $\Delta_1 \ast \cdots \ast \Delta_r$  
associated with $\Delta_1, \dots, \Delta_r$.  Then
  the following statements are equivalent:
  \begin{enumerate}
  \item $\tilde{\Delta}$ is a $(n+r-1)$-dimensional 
Gorenstein polytope of degree  $n$ (and index $r$),
  \item The polytope $\Delta$ is reflexive,
  \item The cone $C_{\tilde{\Delta}}$ associated with $\tilde{\Delta}$ is
    a reflexive Gorenstein cone of index $r$.
  \end{enumerate}
  \end{theo}

This theorem shows a way for constructing examples of Gorenstein polytopes
using Minkowski sum decompositions of reflexive polytopes. 

\begin{exam}
{\rm Let $\Delta$ be the reflexive cube $[-1,1]^n \subset \R^n$. 
Then $\Delta$ is a Minkowski sum of $2n$ $1$-dimensional 
lattice polytopes $\Delta_i^{\pm}:= 
[0, \pm e_i]$ ($ i =1, \ldots, n)$, where 
$e_1, \ldots, e_n$ the standard basis of $\R^n$. The corresponding 
$(3n-1)$-dimensional Gorenstein polytope 
\[ \Delta_1^+ * \cdots * \Delta^+_n * 
\Delta_1^- * \cdots * \Delta^-_n \]
 of degree $n$ will be denoted by $\Theta_n$. 
\label{special-d}
} 
\end{exam} 

\begin{rem} 
{\rm In \cite{ET06}, I. Emiris and E. Tsigaridas computed all 
possible Minkowski decompositions of all reflexive lattice polytopes 
in dimension $2$. Their results 
are listed up to isomorphism in Table  \ref{minkdec1} and Table 
\ref{minkdec2}. } 
\end{rem}

\begin{table}

\begin{tabular}[b]{|l|c|c|c|}\hline&&\\

\rotatebox{0}{
  {\includegraphics[{height=1.5cm}]{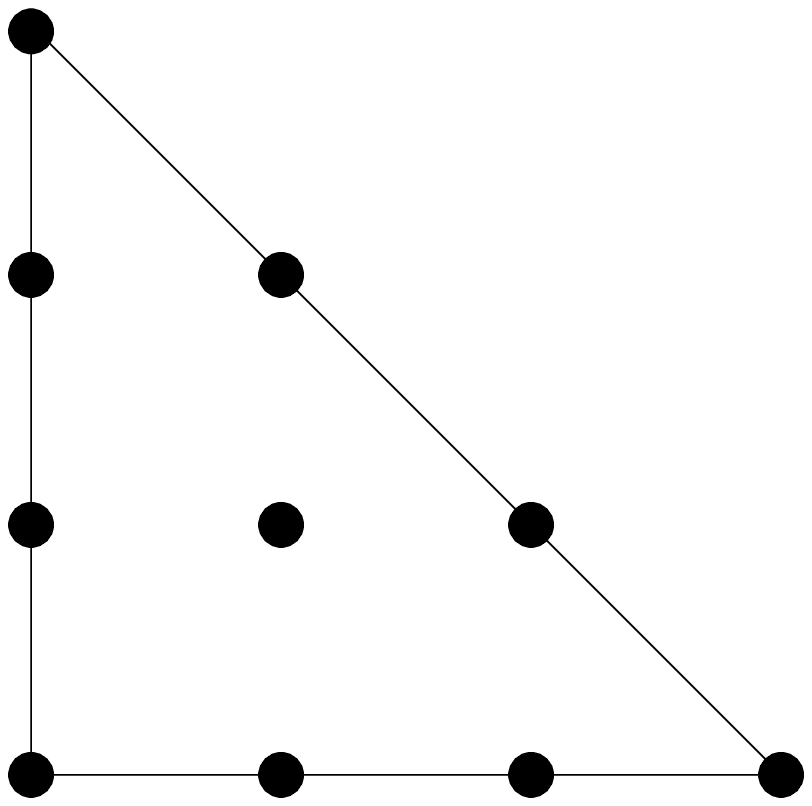}}} &
\raisebox{.5cm}{\rotatebox{0}{
    {\includegraphics[{width=0.5cm}]{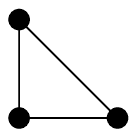}}}}
\raisebox{.8cm}{\; $+$\, }  
\raisebox{.3cm}{ \rotatebox{0}{
    {\includegraphics[{width=1cm}]{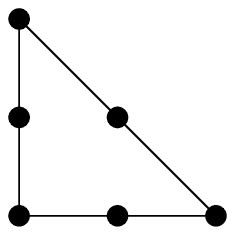}}} } & \hspace*{.5cm}
\raisebox{.5cm}{\rotatebox{0}{
    {\includegraphics[{height=0.5cm}]{side7.eps}}}}
\raisebox{.8cm}{\; $+$ }  
\raisebox{.5cm}{ \rotatebox{0}{
    {\includegraphics[{width=.5cm}]{side7.eps}}} } 
\raisebox{.8cm}{\; $+$ }  
\raisebox{.5cm}{ \rotatebox{0}{
    {\includegraphics[{height=.5cm}]{side7.eps}}} }\\ \hline && \\

\rotatebox{0}{
  {\includegraphics[{height=1cm}]{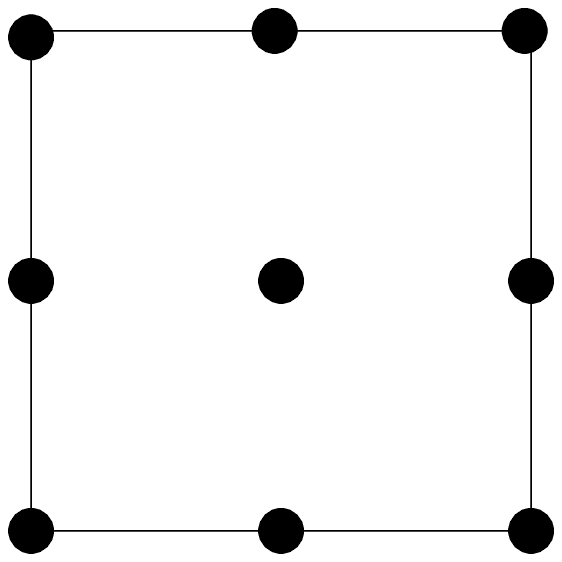}}} &
\raisebox{.3cm}{\rotatebox{0}{
    {\includegraphics[{height=1cm}]{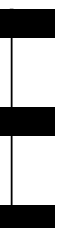}}}}\hspace*{.4cm}
\raisebox{.7cm}{\; $+$\, }  
\raisebox{.7cm}{ \rotatebox{0}{
    {\includegraphics[{width=1cm}]{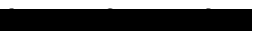}}} } & 
\raisebox{.5cm}{ \rotatebox{0}{
    {\includegraphics[{width=1cm}]{side5.eps}}}}
\raisebox{.5cm}{\; $+$\, }  
\raisebox{.3cm}{ \rotatebox{0}{
    {\includegraphics[{height=.5cm}]{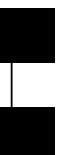}}}\qquad } \hspace*{-.5cm}
\raisebox{.5cm}{\; $+$\, }
\raisebox{.3cm}{ \rotatebox{0}{
    {\includegraphics[{height=.5cm}]{side2.eps}}}\qquad }
\\ \hline && \\
& \raisebox{.3cm}{\rotatebox{0}{
    {\includegraphics[{height=.5cm}]{side2.eps}}}}\hspace*{.4cm}
\raisebox{.5cm}{\; $+$\, }  
\raisebox{.3cm}{ \rotatebox{0}{
    {\includegraphics[{width=1cm}]{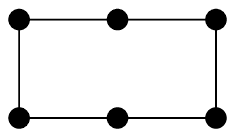}}} } & 
\raisebox{.5cm}{ \rotatebox{0}{
    {\includegraphics[{width=.5cm}]{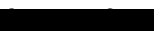}}}}
\raisebox{.5cm}{\; $+$\, }  
\raisebox{.3cm}{ \rotatebox{0}{
    {\includegraphics[{height=.5cm}]{side2.eps}}}\qquad } \hspace*{-.5cm}
\raisebox{.5cm}{\; $+$\, }
\raisebox{.3cm}{ \rotatebox{0}{
    {\includegraphics[{height=.5cm}]{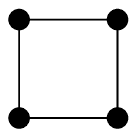}}}\qquad }\\ \hline && \\ 
& \raisebox{.3cm}{ \rotatebox{0}{
    {\includegraphics[{height=.5cm}]{side9.eps}}}}
\raisebox{.5cm}{\; $+$\, }  
\raisebox{.3cm}{ \rotatebox{0}{
    {\includegraphics[{width=.5cm}]{side9.eps}}}\qquad } & 
 \raisebox{.5cm}{ \rotatebox{0}{
    {\includegraphics[{width=.5cm}]{side1.eps}}}}
\raisebox{.5cm}{\; $+$\, }  
\raisebox{.5cm}{ \rotatebox{0}{
    {\includegraphics[{width=.5cm}]{side1.eps}}}\qquad } \hspace*{-.5cm}
\raisebox{.5cm}{\; $+$\, }
\raisebox{.3cm}{ \rotatebox{0}{
    {\includegraphics[{height=.5cm}]{side2.eps}}}\qquad } \hspace*{-.7cm}
\raisebox{.5cm}{\; $+$\, }
\raisebox{.3cm}{ \rotatebox{0}{
    {\includegraphics[{height=.5cm}]{side2.eps}}}\qquad }\\ \hline && \\

\begin{minipage}[t]{1cm}\vspace*{-.6cm}
{\rotatebox{270}{
  {\includegraphics[{height=1cm}]{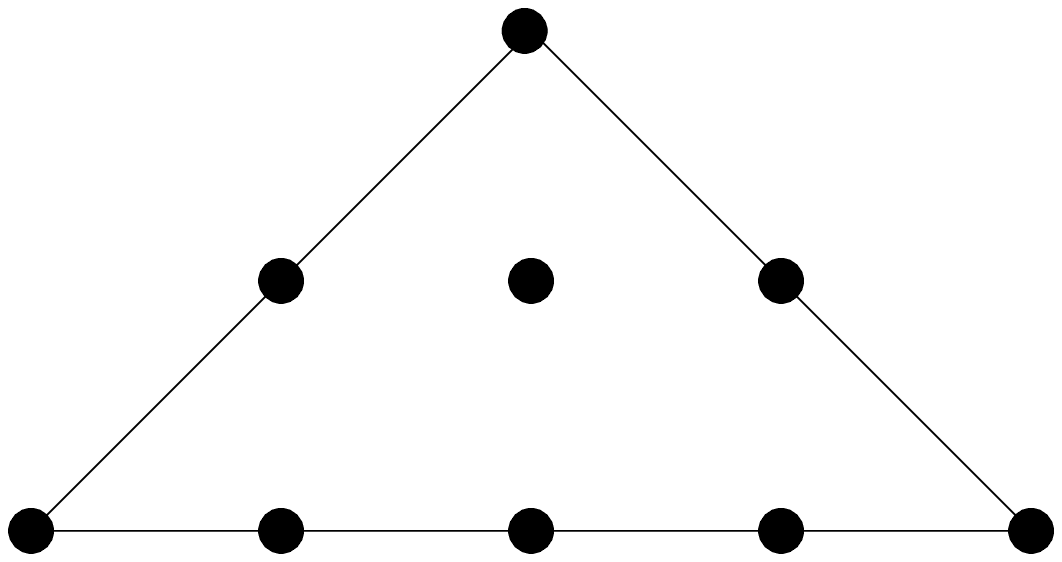}}}\vspace*{.2cm}\hspace*{-2cm}}
  \end{minipage} & \hspace*{-.7cm}
\raisebox{-.9cm}{ \rotatebox{0}{
    {\includegraphics[{width=0.5cm}]{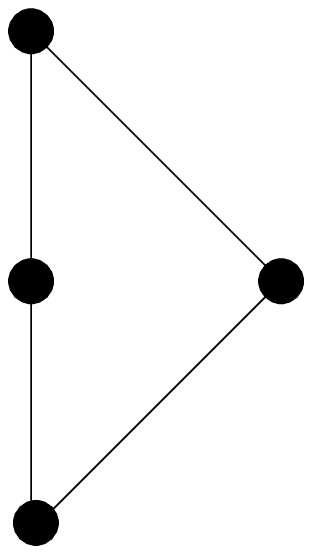}}}}
\raisebox{-0.5cm}{\; $+$\, }  
\raisebox{.1cm}{ \rotatebox{270}{
    {\includegraphics[{height=0.5cm}]{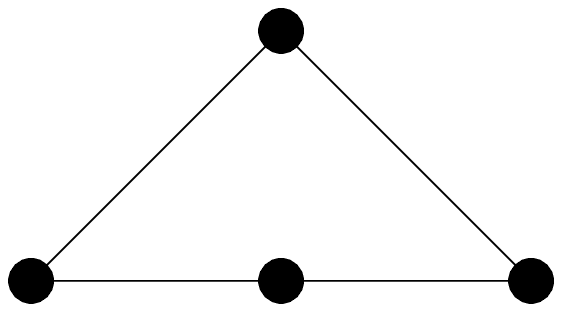}}} } 
 & \\ \hline  && \\ 

\rotatebox{0}{
  {\includegraphics[{height=1cm}]{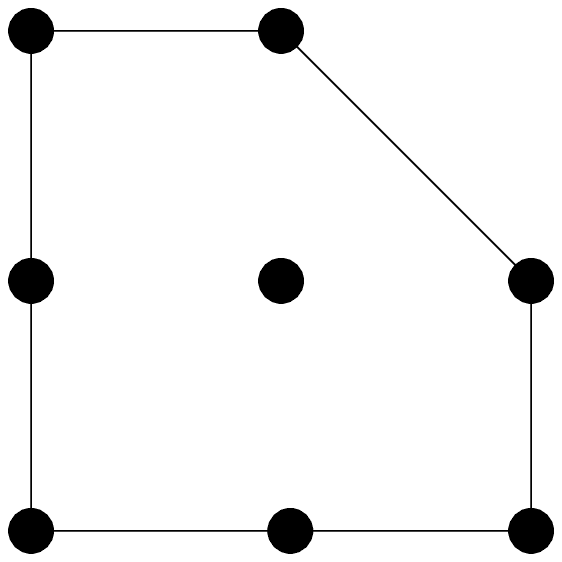}}} & \hspace*{.4cm}
\raisebox{.3cm}{\rotatebox{0}{
    {\includegraphics[{height=0.5cm}]{side2.eps}}}}
\raisebox{.5cm}{\; $+$\, }  
\raisebox{.3cm}{ \rotatebox{0}{
    {\includegraphics[{width=1cm}]{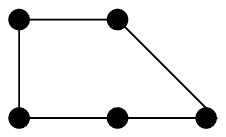}}} }  & \\ \hline && \\
 & \hspace*{-.7cm}
\raisebox{.3cm}{\rotatebox{0}{
    {\includegraphics[{width=0.5cm}]{side7.eps}}}}
\raisebox{.5cm}{\; $+$\, }  
\raisebox{.17cm}{ \rotatebox{90}{
    {\includegraphics[{height=0.5cm}]{side9.eps}}} } & 
\raisebox{.5cm}{\rotatebox{0}{
    {\includegraphics[{width=0.5cm}]{side1.eps}}}}
\raisebox{.5cm}{\; $+$\, }  
\raisebox{.3cm}{ \rotatebox{0}{
    {\includegraphics[{height=0.5cm}]{side2.eps}}} } 
\raisebox{.5cm}{\; $+$\, }  
\raisebox{.17cm}{ \rotatebox{90}{
    {\includegraphics[{width=0.5cm}]{side7.eps}}} }\\ \hline
\end{tabular}
\caption{Minkowski sum decompositions of reflexive polygons}
\label{minkdec1}
\end{table}

{\centering
\begin{table}

\begin{tabular}[t]{|l|c|c|c|}\hline&&\\
\rotatebox{90}{
  {\includegraphics[{height=1cm}]{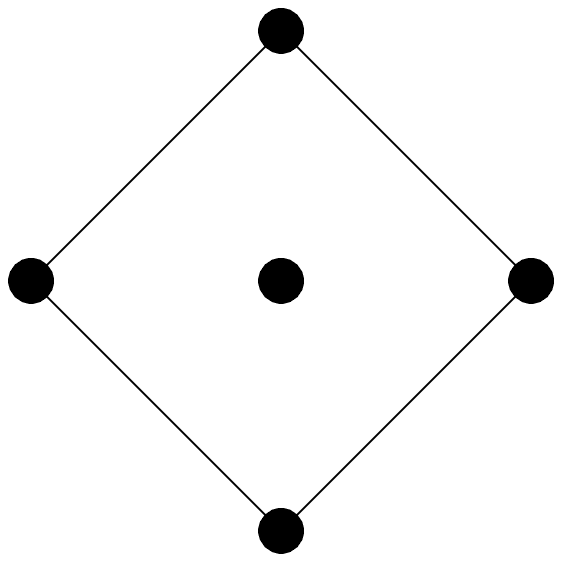}}} &
\raisebox{.2cm}{\rotatebox{90}{
    {\includegraphics[{width=0.5cm}]{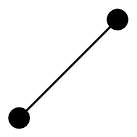}}}}
\raisebox{.5cm}{\; $+$\, }  
\raisebox{.2cm}{ \rotatebox{90}{
    {\includegraphics[{width=0.5cm}]{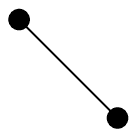}}} } & \\ \hline && \\

\rotatebox{0}{
  {\includegraphics[{height=1cm}]{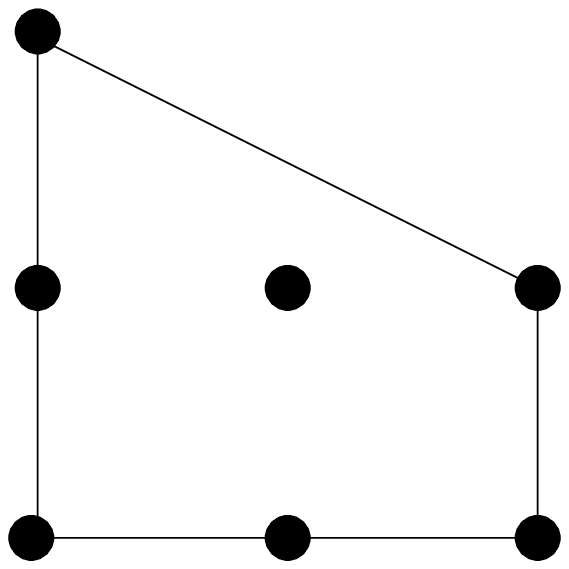}}} & \hspace*{.75cm}
\raisebox{.2cm}{\rotatebox{90}{
    {\includegraphics[{width=0.5cm}]{side1.eps}}}}
\raisebox{.5cm}{\; $+$\, }  
\raisebox{.3cm}{ \rotatebox{0}{
    {\includegraphics[{height=0.5cm}]{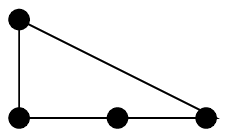}}} } & \\ \hline && \\

\rotatebox{0}{
  {\includegraphics[{height=1cm}]{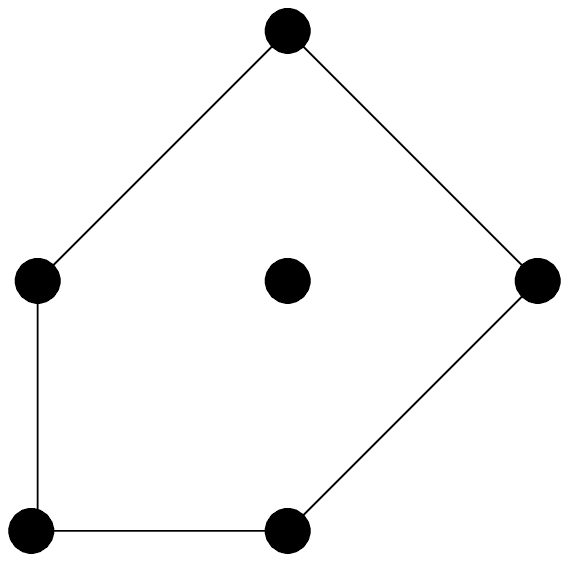}}} & \hspace*{.0cm}
\raisebox{.2cm}{\rotatebox{90}{
    {\includegraphics[{width=0.5cm}]{side3.eps}}}}
\raisebox{.5cm}{\; $+$\, }  
\raisebox{.3cm}{ \rotatebox{0}{
    {\includegraphics[{height=0.5cm}]{side7.eps}}} } & \\ \hline && \\

\rotatebox{0}{
  {\includegraphics[{height=1cm}]{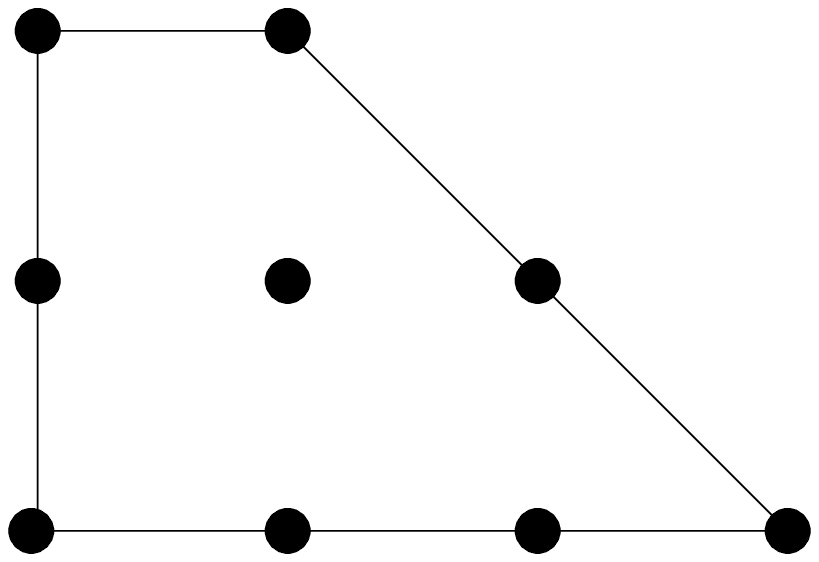}}} & \hspace*{.5cm}
\raisebox{.4cm}{\rotatebox{90}{
    {\includegraphics[{height=0.5cm}]{side2.eps}}}}
\raisebox{.5cm}{\; $+$\, }  
\raisebox{.0cm}{ \rotatebox{0}{
    {\includegraphics[{height=1cm}]{side13.eps}}} } & \\ \hline && \\
 & \hspace*{.2cm}
\raisebox{.3cm}{\rotatebox{0}{
    {\includegraphics[{width=0.5cm}]{side7.eps}}}}
\raisebox{.5cm}{\; $+$\, }  
\raisebox{.3cm}{ \rotatebox{0}{
    {\includegraphics[{height=0.5cm}]{side12.eps}}} } & \hspace*{.5cm}
\raisebox{.4cm}{\rotatebox{90}{
    {\includegraphics[{height=0.5cm}]{side2.eps}}}}
\raisebox{.5cm}{\; $+$ }  
\raisebox{.3cm}{ \rotatebox{0}{
    {\includegraphics[{width=.5cm}]{side7.eps}}} } 
\raisebox{.5cm}{\; $+$ }  
\raisebox{.3cm}{ \rotatebox{0}{
    {\includegraphics[{width=.5cm}]{side7.eps}}} }\\ \hline && \\

\rotatebox{0}{
  {\includegraphics[{height=1cm}]{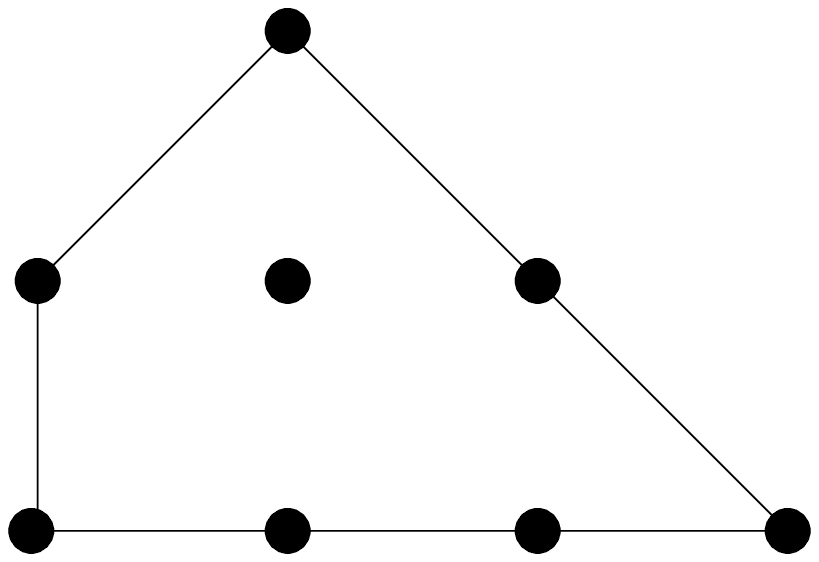}}} & \hspace*{.3cm}
\raisebox{.3cm}{\rotatebox{0}{
    {\includegraphics[{width=0.5cm}]{side7.eps}}}}
\raisebox{.5cm}{\; $+$\, }  
\raisebox{.3cm}{ \rotatebox{0}{
    {\includegraphics[{height=0.5cm}]{side16.eps}}} } &  \\ \hline && \\

\rotatebox{0}{
  {\includegraphics[{height=1cm}]{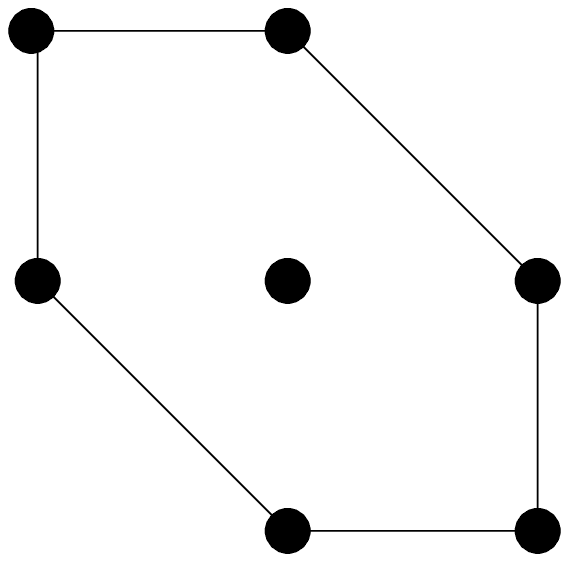}}} & \hspace*{-0.1cm}
\raisebox{.8cm}{\rotatebox{180}{
    {\includegraphics[{height=0.5cm}]{side7.eps}}}}
\raisebox{0.5cm}{\; $+$\, }  
\raisebox{.3cm}{ \rotatebox{0}{
    {\includegraphics[{height=0.5cm}]{side7.eps}}} } & \\ \hline && \\
 & \hspace*{-.1cm}
\raisebox{.3cm}{\rotatebox{0}{
    {\includegraphics[{width=0.5cm}]{side3.eps}}}}
\raisebox{.5cm}{\; $+$\, }  
\raisebox{.3cm}{ \rotatebox{0}{
    {\includegraphics[{height=0.5cm}]{side9.eps}}} } & \hspace*{.5cm}
\raisebox{.3cm}{\rotatebox{0}{
    {\includegraphics[{height=0.5cm}]{side3.eps}}}}
\raisebox{.5cm}{\; $+$ }  
\raisebox{.5cm}{ \rotatebox{0}{
    {\includegraphics[{width=.5cm}]{side1.eps}}} } 
\raisebox{.5cm}{\; $+$ }  
\raisebox{.3cm}{ \rotatebox{0}{
    {\includegraphics[{height=.5cm}]{side2.eps}}} }\\ \hline && \\

\rotatebox{0}{
  {\includegraphics[{height=1cm}]{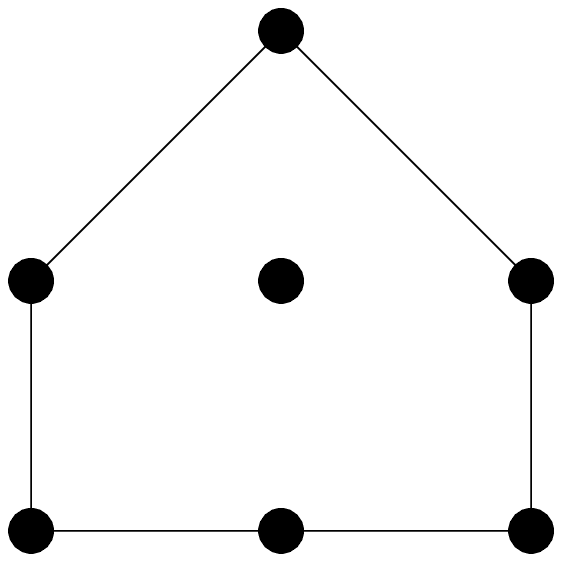}}} & \hspace*{.7cm}
\raisebox{.8cm}{\rotatebox{180}{
    {\includegraphics[{height=0.5cm}]{side2.eps}}}}
\raisebox{0.5cm}{\; $+$\, }  
\raisebox{.3cm}{ \rotatebox{0}{
    {\includegraphics[{height=0.5cm}]{side16.eps}}} } & \\ \hline && \\
 & \hspace*{-.1cm}
\raisebox{.3cm}{\rotatebox{0}{
    {\includegraphics[{width=0.5cm}]{side7.eps}}}}
\raisebox{.5cm}{\; $+$\, }  
\raisebox{.17cm}{ \rotatebox{90}{
    {\includegraphics[{height=0.5cm}]{side7.eps}}} } & \rule{6.5cm}{0pt}\\ 
\hline
\end{tabular} 
\caption{Minkowski sum decompositions of reflexive polygons (continued)}
\label{minkdec2}
\end{table}}

\begin{rem} \label{dual-faces}
{\rm Let $P \subset M_{\mathbb{R}}$ be a $n$-dimensional  
Gorenstein polytope of index $r$. By  \ref{refl-cone-dual}, 
the dual polytope $P^*$ of a
Gorenstein polytope $P$ of index $r$ is also a Gorenstein 
polytope of index $r$. Moreover, $P$ and $P^*$ are combinatorially 
dual to each other, i.e., there exists 
a natural bijection $F \leftrightarrow F^*$ 
between $k$-dimensional faces $F$ 
of $P$ and $(n-k-1)$-dimensional faces $F^*$ of $P^*$.  
} 
\end{rem}

\begin{rem} \label{sublattice}
{\rm If  $r>1$ then the reflexive polytopes $rP$
and $rP^*$ are not  dual to each other  
as reflexive polytopes (see \ref{dual-pol}). 

Denote by $M' \subset  M_{\mathbb{R}}$ the $n$-dimensional lattice 
generated by $M$ and $m/r \in  M_{\mathbb{R}}$, where $m \in rP$ is the 
unique interior point of the reflexive polytope $rP$.  
In order to get the dual Gorenstein polytope $P^*$ 
one has to dualize the reflexive polytope $rP$, i.e., 
to get the dual polytope  $(rP)^*$ and replace the dual lattice 
$N$ by the  sublattice $N' \subset N$, $|N/N'| =r$,  
 which is dual to $M'$ (see \cite[Prop. 1.15]{BN08}).
 } 
\end{rem}

\begin{rem} \label{duality-Gor}
{\rm
By   \cite[Prop. 1.16]{BN08}, 
one has 
\[ \partial (rP)^* \cap N =  \partial (rP)^* \cap N', \]
i.e., there exists a natural bijection between the set of $N'$-lattice 
points in $P^*$ and the set of boundary $N$-lattice points in the dual 
reflexive polytope $(rP)^*$. Moreover, every facet $\Gamma$ of
  $P^*$ considered as a lattice polytope 
with respect to  $N'$ is isomorphic  
to the facet $\Gamma$ of $\Delta^*$  considered  as a lattice 
polytope with respect to $N$.  
}
\end{rem}

\begin{prop}
  Let $P \subset M_{\mathbb{R}}$ be an $n$-dimensional Gorenstein
  polytope of degree $d$. Denote by $\Gamma_1, \ldots,\Gamma_k$ all
  codimension-$1$ faces of $P$. Then one has
\[ \sum_{i =1}^k \nv(\Gamma_i) = (n+1 -d) \nv(P). \]
\label{face-vol}
\end{prop}

\begin{proof}
The polytope $\Delta = (n+1-d)P$ is reflexive. Up to a shift by a 
lattice vector, we can assume that $0$ is  the unique
interior lattice point of $\Delta$. Then one can decompose
$\Delta$ into a union of $n$-dimensional pyramids with vertex $0$
over $(n-1)$-dimensional
faces $(n+1-d)\Gamma_1, \ldots, 
(n+1-d)\Gamma_k$ of $\Delta$. Since heights of all
these pyramids are equal to one, we obtain
\[   \sum_{i =1}^k \nv((n+1-d)\Gamma_i) = \nv (\Delta). \]
It remains to combine it with 
the equalities $\nv (\Delta) = (n+1-d)^n \nv(P)$ and
\[ \nv ((n+1-d)\Gamma_i) = (n+1-d)^{n-1} \nv(\Gamma_i),
 \;1 \leq i \leq k. \]
\end{proof}

\begin{dfn}
Let  $v_i \in P$ be  a vertex  of a  $n$-dimensional
Gorenstein polytope $P$ of degree $d$. We denote by $C(v_i)$ 
the $n$-dimensional
cone generated by vectors $v -v_i$ where $v$ runs over all lattice
points of $P$.
\end{dfn}

\begin{prop}
  Let $v_i \in P$ be an arbitrary vertex of a $n$-dimensional
  Gorenstein polytope $P$ of degree $d$ and let $\Gamma_i \subset P^*$
  be the dual facet in the dual polytope $P^*$.
  Then then cone $C(v_i)$ is dual to the cone $C_{\Gamma_i}$ over the
  lattice polytope $\Gamma_i$ $($see \ref{dual-v-g}$)$.
\label{Ci}
\end{prop}

\begin{exam} \label{dual-v-g}
  \begin{itemize}
  \item[]
  \item[] 
  \item[]\begin{figure}[h] \vspace*{-1cm}
    \centering    
    {\epsfig{file=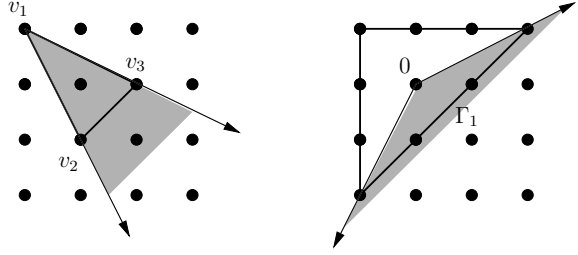, height=3.5cm}}
    \caption{Duality between $C(v_1)$ and $C_{\Gamma_1}$}
  \end{figure} 
  \end{itemize}  \vspace*{-.8cm}
\end{exam}

\begin{proof}
  Let $\Delta = (n+1-d)P \subset M_{\R}$ be the reflexive polytope
  corresponding to $P$. By \ref{sublattice}, $P^*$ is obtained 
from the dual reflexive
polytope $\Delta^* \subset N_{\R}$ by choosing a sublattice $N' \subset N$
of index $r = n+1-d$. By \ref{duality-Gor},  every facet $\Gamma$ of
  $P^*$ with respect to the lattice $N'$ is isomorphic 
  to the facet $\Gamma$ of $\Delta^*$ (with respect to the lattice $N$). On
  the other hand, the cone $C(v_i)$ is 
equal to the cone generated by vectors $v
  -(n+1-d)v_i$ where $v$ runs over all lattice points of $\Delta = (n+1-d)P$.
  The duality between reflexive polytopes 
$\Delta$ and $\Delta^*$ shows that $C(v_i)$ 
in dual to cone
  $C_{\Gamma_i}$ over $\Gamma_i$ is a facet of $\Delta$. 
\end{proof}

\section{The graded Gorenstein algebra $S_P$}
\label{SecGorenstein}

Let $K$ be an arbitrary field of characteristic $0$ and let $P \subset 
M_{\R}$ be an arbitrary $n$-dimensional lattice polytope. We consider 
the $(n+1)$-dimensional lattice $\overline{M} = M \oplus \Z$ and 
the $(n+1)$-dimensional cone $C_P \subset \overline{M}_{\R}$ over $P$. 
There exists a natural grading of lattice points $\overline{m} = (m, k) \in 
C_P \cap \overline{M}$ $(m \in M, k \in \Z_{\geq 0})$: 
\[ \deg \overline{m} := k, \]
so that $C_P \cap \overline{M}$ becomes a graded monoid.

\begin{dfn} 
{\rm 
We denote by $S_P := K[C_P \cap \overline{M}]$ 
the graded $K$-algebra of the semigroup 
$C_P \cap \overline{M}$. For each lattice point  $\overline{m} \in 
C_P \cap \overline{M}$ we denote by ${\bf x}^{\overline{m}}$ the corresponding 
momomial in  $S_P$. It is well-known that the $K$-algebra 
$S_P$ is always Cohen-Macaulay (see e.g. \cite{Sta78}). } 
\end{dfn}

\begin{dfn} 
{\rm 
A lattice point $m \in C_P \cap \overline{M}$ 
is called {\em irreducible} if it is not 
representable as a sum of two lattice points $\overline{m} = \overline{m}' +
\overline{m} ''$ $(\overline{m}', \overline{m}'' \in 
C_P \cap \overline{M})$.} 
\end{dfn}

Obviously every lattice point of degree $1$ in $C_P \cap \overline{M}$ 
is irreducible. 
Denote by $I_P(k)$ the number of irreducible lattice points of degree 
$k$ in $C_P \cap \overline{M}$. Since the semigroup $C_P \cap \overline{M}$ 
is finitely generated we obtain that  
\[ I(P,t):= \sum_{k \geq 0} I_P(k)t^k  \]
is a polynomial, i.e. $I_P(k) = 0$ for $k \gg 0$. The number  $I_P(k)$ 
 can be identified with the number of lattice points of degree $k$ in a 
minimal generating set of the semigroup $C_P \cap \overline{M}$.

\begin{dfn} 
{\rm 
Let $S = \bigoplus_{i \geq 0} S_i$ be a finitely generated 
graded $K$-algebra, 
$S_0 \cong K$. 
For any $k, l\in \N$ we denote by 
$S_kS_l$ the $K$-subspace in $S_{k+l}$ generated by products 
$xy$, where $x \in S_k$, $y \in S_l$.  
We define the numbers $g_k(S)$ as follows: 
$g_0(S) = 1$,  
\[ g_k(S):= \dim_K S_k/(\sum_{j=1}^{k-1} S_j S_{k-j}), \; \forall k >0. \]
It is clear that $g_1(S) =  \dim_K S_1$ and the sum 
\[ G_S(t):= \sum_{k \geq 0} g_k(S)t^k  \]
is a polynomial, i.e. $ g_k(S) = 0$ for $k \gg 0$.  
The number  $g_k(S)$ $(k>0)$ can be identified with the number 
of elements of degree $k$ in the minimal generating system of the 
$K$-algebra $S$. Here the minimality means that any proper subset of this 
system does not generate $S$.} 
\end{dfn} 

\begin{prop} 
Let   $S= S_P$ be  the graded algebra of a lattice polytope $P$. Then the 
polynomial $G_S(t)$ is equal to  $I(P,t)$.
\end{prop} 

\noindent
{\em Proof.} Since $S_P= K[ C_P \cap \overline{M}]$ 
is generated by monomials ${\bf x}^{\overline{m}}$ corresponding to 
lattice points $\overline{m} \in  C_P \cap \overline{M}$, a minimal generating 
system $G= \{\gamma_1, \ldots, \gamma_l \}$ 
of the $K$-algebra $S_P$ can be chosen as a finite subset 
of these monomials. It remains to show that ${\bf x}^{\overline{m}} \in G$
if and only if $\overline{m} \in  C_P \cap \overline{M}$ is irreducible. 
Indeed, if  ${\bf x}^{\overline{m}} \in G$ and $ 
\overline{m} = \overline{m}' +
\overline{m} ''$ 
$\overline{m}',\overline{m} '' \in  C_P \cap \overline{M})$ and 
$k = \deg \overline{m}'$, $l = \deg \overline{m}''$, 
then ${\bf x}^{\overline{m}} = {\bf x}^{\overline{m}'} {\bf x}^{\overline{m}''} \in 
S_k S_l$. This contradicts 
the minimality of $G$. On the other hand, if $\overline{m} \in  C_P \cap 
\overline{M}$ is an irreducible lattice point of degree $k$, then   
${\bf x}^{\overline{m}}$ can not be a polynomial expression 
of elements of degree $<k$. 
Hence ${\bf x}^{\overline{m}}$ must appear in $G$. 
\hfill $\Box$

\begin{prop} \label{generatA}
Let $P$ be a $n$-dimensional lattice polytope. Denote by 
$A_P$ an  Artinian graded $K$-algebra obtained by quotient of  
$S_P$ modulo a regular sequence of $n+1$ elements $y_0, y_1, \ldots, y_n$ 
of degree $1$. Then 
\[ G_{A_P}(t) = I_P(t) - (n+1)t. \] 
\end{prop} 

\noindent
{\em Proof.} Obviously, we have  \[ g_1(A_P) = \dim_K (A_P)_1 = 
| P \cap M| - (n+1) = I_P(1) - (n+1). \]
Denote by  $\psi$ the epimorphism $S_P \to A_P$. If $G \subset S_P$ 
is minimal generating system of monomials, then $\psi(G)$ 
generate $A_P$ as $K$-algebra. 
This implies $g_k(A_P) \leq I_P(k)$ $\forall k\geq 2$. In fact, we 
have the equality 
 $g_k(A_P) = I_P(k)$ $\forall k\geq 2$, because    
the kernel of the epimorphism 
\[ \psi\, : \, (S_P)_k \to (A_P)_k \]
is contained in $(S_P)_1 (S_P)_{k-1}$. 
\hfill $\Box$  


Let  $P$ be an arbitary  
$n$-dimensional lattice polytope and let $I(P,1) = \sum_{k \geq 0} 
I_P(k)$ be  
the number of irreducible lattice points $C_P \cap \overline{M}$. 
Then we can describe the graded  algebra  
$S_P$ as a quotient of a polynomial ring $R:= K[x_1, \ldots, x_{I(P,1)}]$ 
such that exactly $I_P(k)$ of variables 
$x_i$ have weight $k$ and  $J:= {\rm Ker}( R \to S_P)$ 
is an ideal generated by binomials 
\[ B= x_{i_1}^{a_1} \cdots x_{i_s}^{a_s} - x_{j_1}^{b_1} \cdots x_{j_r}^{b_r}. \]  
We use the following two statements: 

\begin{prop}\cite[Prop. 3.5]{Bat06}
Let $\{ B_1, \ldots, B_p\}$ be a set of binomials which minimally  
generate the ideal $J$. Then $P$ is a pyramid over a $(n-1)$-dimensional 
lattice polytope $P'$, i.e., $P = \Pi(P')$, if and only if the exists 
a variable $x_i \in \{x_1, \ldots, x_{I(P,1)} \}$ which does not appear 
in any of binomials  $ B_1, \ldots, B_p$.  
\label{pyramid1}
\end{prop} 

\begin{prop}\cite[Prop. 3.6]{Bat06} 
Let $\overline{R}:= R/\langle y \rangle$ be a polynomial 
ring in  $I(P,1) - (n+1)$ variables obtained from the polynomial ring 
$R:= K[x_1, \ldots, x_{I(P,1)}]$ by quotient modulo a regular sequence 
$y_0, y_1, \ldots, y_n$ of elements of degree $1$. We denote 
by $\overline{J}$ the ideal $J/\langle y \rangle J$ in $\overline{R}$. 
Then the numbers and the degrees of minimal generators 
of ideals $J \subset R$ and $\overline{J} \subset \overline{R}$ 
are the same.   
\label{degrees}
\end{prop} 

Now we apply the above statements to a Gorenstein polytope $P$. 
In this case, the  graded $K$-algebras $S_P$ and $A_P$ are known to be 
Gorenstein. 
If $A := A_0
\oplus \cdots \oplus A_l$ is 
an Artinian  graded Gorenstein  $K$-algebra 
($A_0 = K$, $A_l \not= 0$), then the zero ideal in $A$ 
is irreducible and  the multiplication in $A$ 
defines a perfect pairing
 \[ A_i \times A_{l-i} \rightarrow A_l \cong  K.\] 
For our purposes it will be sufficient  to use the following more special 
fact: 

\begin{prop} \cite{Hun07}
  Let $A := A_0 \oplus A_1 \oplus
  A_2$ be an Artinian graded $K$-algebra 
 where $A_0 \cong A_2 \cong  K$ is  a field, and let $\langle \cdot, \cdot
  \rangle$ be a symmetric bilinear form on $A_1$
  given by the multiplication 
  \[ \langle \cdot , \cdot \rangle : A_1 \times A_1 \rightarrow A_2 \cong 
  K \qquad (x,y) \mapsto \langle x, y  \rangle:=  xy. \]  Then
  $A$ is a Gorenstein ring if and only if $\langle \cdot , \cdot
  \rangle$ is a non-degenerate symmetric bilinear form on
  $A_1$. 
\end{prop}

This characterisation allows to describe all 
relations in $A$ if   $\dim_K (A_1) \leq 2$:

\begin{prop} Let $A= A_0 \oplus A_1 \oplus
  A_2$ be an Artinian  
  graded Gorenstein $K$-algebra  and $a = \dim_K (A_1)  \in \{ 0, 1, 2 \}$. 
Then we have: 

{\rm (i)} if $a =0$, then $A \cong K[X] / (X^2)$,  $\deg (X) = 2$;

{\rm (ii)} if $a = 1$, then $A \cong K[X] / (X^3)$, $\deg(X) = 1$;

{\rm (ii)} if $a  = 2$, then $A \cong K[X, Y] / (XY, X^2 - \lambda Y^2)$, 
 $\deg(X) = \deg(Y) = 1$, $\lambda \neq 0$.
\label{Gor1a1}
\end{prop} 

\noindent
{\em Proof.}  
(i) Let $x \in A_2$ be a nonzero element. Consider the $K$-algebra 
homomorphism $\psi\, : \, K[X] \to A$ defined by the condition $\psi(X) = x$. 
Since $A_2 A_2 \subset A_4 =0$ we have $x^2 =0$, i.e., $X^2 \in {\rm 
Ker}\, \psi$. So $\psi$  induces an isomorphism 
$A \cong K[X] / (X^2)$, because both rings $A$ and  $K[X] / (X^2)$ 
are $2$-dimensional $K$-vector spaces. 

(ii)  Let $x \in A_1$ be a nonzero element. Since the bilinear map 
$A_1 \times A_1 \to A_2$ is nondegenerate, we have $x^2 \neq 0$, i.e. 
$\psi$ is surjective. 
However, $x^3 =0$, because $A_3 = 0$. Consider the $K$-algebra 
homomorphism $\psi\, : \, K[X] \to A$ defined by the condition 
$\psi(X) = x$.  Then $X^3 \in {\rm Ker}\, \psi$.
 So $\psi$  induces an isomorphism $R \cong K[X] / (X^2)$, because  
both rings $R$ and  $K[X] / (X^2)$ are $3$-dimensional $K$-vector spaces.

(iii) Since the bilinear map $A_1 \times A_1 \to A_2 \cong K$ 
is nondegenerate there exist $K$-linearly independent elements 
$x, y \in A_1$  such that $xy =0$ and $x^2 \neq 0$, $y^2 \neq 0$. 
Consider the $K$-algebra 
homomorphism $\psi\, : \, K[X,Y] \to A$ defined by the conditions 
$\psi(X) = x$, $\psi(Y) =y$. Then  $XY, X^2 - \lambda Y^2 
\in {\rm Ker}\, \psi$ for some  $\lambda \in K \setminus \{0 \}$.
Comparing the dimensions of $K[X, Y] / (XY, X^2 - \lambda Y^2)$ and 
$A$ we get an isomorphism  $A \cong K[X, Y] / (XY, X^2 - \lambda Y^2)$.
\hfill $\Box$

\begin{theo} \label{class-aleq2}
Let $P$ be a $n$-Gorenstein polytope of degree $2$ such that its $h^*$-vector 
is $(1,a,1)$ where $a \in \{0,1,2\}$, i.e., $\nv(P) = a+2 \leq 4$. 
Assume that $P$ is not a pyramid over 
a $(n-1)$-dimensional Gorenstein polytope. We write the graded Gorenstein ring 
$S_P$ as quotient of a polynomial ring $K[x_1,\ldots, x_l]$ modulo an 
ideal $J$  generated by binomials. 
Then $n \leq 5$ and there exist up to isomorphim exactly the following  $15$ 
possibilities for $P$ 

{\rm (1)} $\dim P = 2$:  $\Delta_i$, $1 \leq i \leq 4$; 

{\rm (2)}  $\dim P = 3$: $P_i$, $1 \leq i \leq 7$;

{\rm (3)}  $\dim P = 4$: $Q_i$, $1 \leq i \leq 3$;  

{\rm (4)}  $\dim P = 5$: $R_1$

The corresponding binomial relations in $S_P$ are presented in the table 
below

\begin{table}[h]
  \renewcommand{\baselinestretch}{1.5}\normalsize
  \centering
  \begin{tabular}[h]{cc}
    \begin{tabular}[h]{c|l}
$\Delta_1$ & $x_1 x_2 x_3 = x_4^3$   \\ \hline
  $\Delta_2$ & $x_1 x_2 = x_5^2$, $x_3 x_4 = x_5^2$ \\  \hline 
$\Delta_3$ & $x_1 x_2 = x_5^2$, $x_3 x_4 = x_1 x_5$ \\ \hline
 $\Delta_4$ & $x_1 x_2 = x_5^2$, $x_3 x_4 = x_1^2$  \\ \hline
    $P_1$ & $x_1 x_2 x_3 x_4 = x_5^2$   \\ \hline  
    $P_2$ & $x_1 x_2 x_3 = x_4 x_5^2$  \\ \hline   
    $P_3$ & $x_1 x_2 = x_3^2$, $x_4 x_5 = x_6^2$    
          \end{tabular} &
{\begin{tabular}[h]{c|l}
 $P_4$ & $x_1 x_2 = x_3^2$, $x_1 x_4 = x_5 x_6^2$ \\ \hline
 $P_5$ & $x_1 x_2 = x_3^2$, $x_3 x_4 = x_5 x_6$ \\ \hline
    $P_6$ & $x_1 x_2 = x_3 x_4 $, $x_1 x_5 = x_2 x_6$ \\ \hline
    $P_7$ & $x_1 x_2 = x_5 x_6$, $x_1 x_2 = x_3 x_4$  \\ \hline  
    $Q_1$ & $x_1 x_2 x_3 = x_4 x_5 x_6$\\ \hline
    $Q_2$ & $x_1 x_2 = x_3 x_4$, $x_5 x_6 = x_7 x_1$\\ \hline
    $Q_3$ & $x_1 x_2 = x_3 x_4$, $x_5 x_6 = x_7^2$\\ \hline
    $R_1$ & $x_1 x_2 = x_3 x_4$, $x_5 x_6 = x_7 x_8$ 
  \end{tabular}}
\end{tabular}
  \label{tab:realtions}
\end{table}

\end{theo}

\noindent
{\em Proof.} 
Let us consider all three values of  $a$ separately. 

Case 1: $a =0$. By \ref{generatA}, the graded $K$-algebra 
$S_P$ is minimally generated by $n+1$ elements of degree $1$ and 
by one element of degree $2$. 
First we note that $n \geq 3$, because for any $2$-dimensional 
lattice polytope $\Delta$ the graded $K$-algebra $S_\Delta$ is generated by elements 
of degree $1$. So $S_P$ has $n+2 \geq 5$ generators. Since $\dim (S_P)_1 =1$ 
the polytope $P$ contains exactly $n+1$ lattice points. These lattice 
points must be vertices, because $\dim P = n$. Thus, $P$ 
is a $n$-dimensional lattice simplex. Let $x_1, \ldots, x_{n+1}$ be variables 
in the polynomial ring $R= K[x_1, \ldots, x_{n+1}, x_{n+2}]$ 
whose images in $S_P$ are monomials ${\bf x}^{\overline{m}}$ corresponding 
to vertices $v_1, \ldots, v_{n+1}$ of $P$. Since $\deg x_{n+2} = 2$, 
the image of $x_{n+2}$ in $S_P$ is a monomial corresponding to an irreducible  
lattice point $v_{n+1} \in 2P \cap M$ which is not an {\em integral}  
nonnegative linear 
combination of $v_1, \ldots, v_{n+1}$. Since $P$ is a convex 
hull of $v_1, \ldots, v_{n+1}$, there exist nonnegative  
{\em rational} numbers $c_1, \ldots, c_{n+1}$ such that 
 \[ \sum_{i =1}^{n+1} c_i (v_i,1) = (v_{n+2}, 2). \]
 By \ref{degrees}, there is only one 
binomial relation among $n+2$ generators of $S_P$ and this relation 
has degree $4$. This implies, $2c_i \in \Z$ $(i =1, \ldots, n+1)$ and 
the single binomial relation in $S_P$ has form
$\prod_{i =1}^{n+1}x_i^{2c_i} = x_{n+2}^2$. 
Since $P$ is not a pyramid, by \ref{pyramid1}, $2c_i \geq 1$ 
$(i =1, \ldots , n+1)$. 
On the other hand, $\sum_{i=1}^{n+1} 2c_i = 4 = \deg x_{n+2}^2$. This 
implies $n+1 \leq 4$. Thus, $3$ is the only possible value of $n$ and 
$2c_1 = 2c_2 = 2c_3 = 2c_4 =1$. So we come to the binomial relation
$ x_1x_2 x_3 x_4 - x_5^2 =0$. 
This relation defines a $3$-dimensional Gorenstein simplex which we denote by 
$P_1$.

Case 2: $a =1$.  By \ref{generatA} and  \ref{degrees}, the graded $K$-algebra 
$S_P$ is minimally generated by $n+2$ elements of degree $1$ 
satisfying a single cubic binomial relation. Such a binomial relation 
includes at most $6$ variables from $\{ x_1, \ldots, x_{n+2} \} \subset R$. 
By \ref{pyramid1}, this implies 
$n+2 \leq 6$, i.e., $n \leq 4$. Moreover, 
if $n =4$, then up to renumeration of variables this binomial must be 
$x_1x_2x_3 - x_4 x_5 x_6$ (polytope $Q_1$). If $n =2$, then $P$ is a reflexive 
triangle and the binomial has form $x_1x_2x_3 - x_4^3$ (polytope $\Delta_1$). 
By \ref{pyramid1},
if $n = 3$ then at least one monomial in the binomial 
is product of $3$ different 
variables. Without loss of generality, we can assume that this monomial 
is $x_1x_2x_3$. Then the second monomial must have form $x_i^2 x_j$. Up to 
renumeration of variables we obtain the binomial relation 
$x_1x_2x_3 - x_4x_5^2$
which defines a $3$-dimensional Gorenstein polytope $P_2$. 

Case 3:  $a =2$. By \ref{generatA} and  \ref{degrees}, the graded $K$-algebra 
$S_P$ is minimally generated by $n+3$ elements of degree $1$ 
satisfying two quadratic  binomial relations. These two quadratic relations 
include at most $8$ variables from $\{ x_1, \ldots, x_{n+3} \} \subset R$. 
By \ref{pyramid1}, this implies $n+3 \leq 8$, i.e., $n \leq 5$. The quadratic 
monomials appearing in the binomials are either $x_i^2$, or $x_ix_j$. It is not 
difficult to enumerate all possibilities and get the corresponding Gorenstein 
polytopes 
$$\Delta_2,\Delta_3, \Delta_4, P_3, P_4, P_5, P_6, P_7, Q_2, Q_3, R_1.$$
\hfill $\Box$

\begin{tabular}[h]{ccc}
  \begin{minipage}[h]{4cm}
    \input{P1Mod.pstex_t}
  \end{minipage}&

  \begin{minipage}[h]{4cm}
    \input{P5Mod.pstex_t}
  \end{minipage} &

  \begin{minipage}[h]{4cm}
    \input{P2Mod.pstex_t}
  \end{minipage}\\

  \begin{minipage}[h]{4cm}
    \centering {\bf $P_1$}
  \end{minipage} & 

  \begin{minipage}[h]{4cm}
    \centering {\bf $P_2$}
  \end{minipage} &

  \begin{minipage}[h]{4cm}
    \centering {\bf $P_3$}
  \end{minipage}\\ 

  \begin{minipage}[h]{4cm}
    \centering $x_1 x_2 x_3 x_4 = x_5^2$
  \end{minipage} & 

  \begin{minipage}[h]{4cm}
    \centering $x_1 x_2 x_3 = x_4 x_5^2$
  \end{minipage} &

  \begin{minipage}[h]{3cm}
    \centering $x_1 x_2 = x_3^2$, $x_4 x_5 = x_6^2$
  \end{minipage}\\ \\ 

\end{tabular} \vspace*{-1em}


\begin{tabular}[h]{ccc}
  \begin{minipage}[h]{4cm}
    \input{P4Mod.pstex_t}
  \end{minipage} &

  \begin{minipage}[h]{4cm} 
    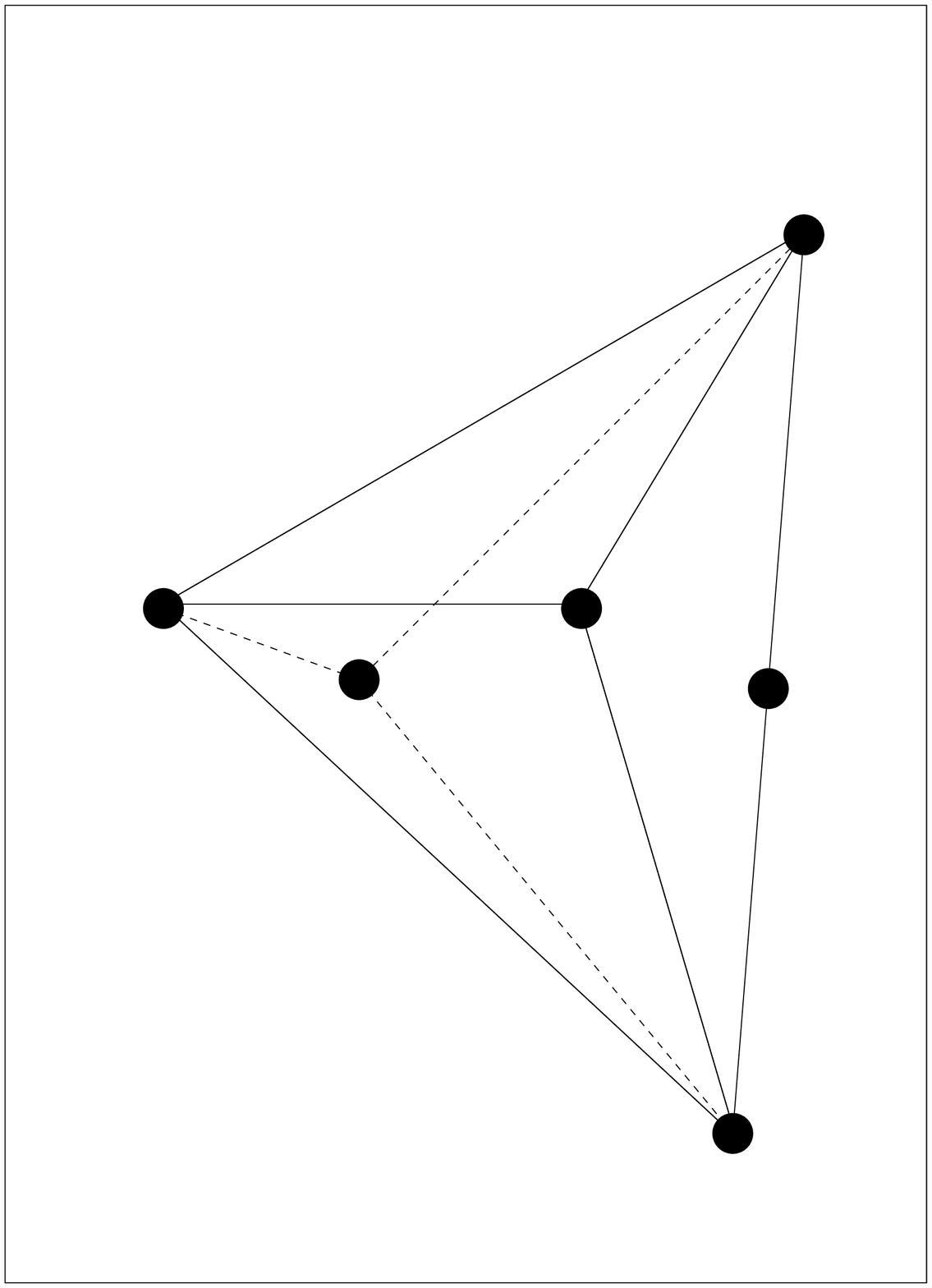
  \end{minipage} &

  \begin{minipage}[h]{4cm}
    \input{P10Mod.pstex_t}
  \end{minipage} \\ \\

  \begin{minipage}[h]{4cm}
    \centering {\bf $P_4$}
  \end{minipage} & 

  \begin{minipage}[h]{4cm}
    \centering {\bf $P_5$}
  \end{minipage} &

  \begin{minipage}[h]{4cm}
    \centering {\bf $P_6$}
  \end{minipage}\\ 

  \begin{minipage}[h]{3cm}
    \centering $x_1 x_2 = x_3^2$, $x_1 x_4 = x_5 x_6$
  \end{minipage} & 

  \begin{minipage}[h]{3cm}
    \centering $x_1 x_2 = x_3^2$, $x_3 x_4 = x_5 x_6$
  \end{minipage} &

  \begin{minipage}[h]{4cm}
    \centering $x_1 x_2 = x_3 x_4$, $x_1 x_5 = x_2 x_6$
 \end{minipage} \\ \\ 
\end{tabular} \vspace*{-1em}


\begin{center}

\begin{tabular}[h]{c}
 \begin{minipage}[h]{4cm}
    \input{P11Mod.pstex_t}
  \end{minipage}\\ \\
   \begin{minipage}[h]{4cm}\vspace*{-2em}
    \centering $P_7$
 \end{minipage} \\
 \begin{minipage}[h]{4.5cm} \vspace*{-1em}
    \centering $x_1 x_2 = x_5 x_6$,\ $x_1 x_2 = x_3 x_4$ \vspace*{1em}
 \end{minipage}
\end{tabular} \vspace*{-1em}

\end{center}

\section{Gorenstein polytopes of degree $2$}
\label{Properties}

Let $P\subset M_{\mathbb{R}}$ be a 
$n$-dimensional Gorenstein polytope of degree $2$ with $h^*$-polyomial
$h^*_P(t) = 1 + at + t^2$, $a \geq 0$. Then we have 
 \begin{eqnarray*}
    i(P,k) & = & {n+k \choose n} + {n+k-1 \choose n} a + {n+k-2 \choose n}\\
        & = & \frac{1}{k!} (a + 2)k^n +\ {\text \rm terms\ of\ lower\ order},
  \end{eqnarray*}
\[ | P \cap M| = i(P,1) = { n+1 \choose n}
  + {n \choose n} a + {n-1 \choose n} = n+a+1. \]
In particular, $$\nv(P) = a+2 = |P \cap M| - n +1.$$ 

\begin{prop}\label{volume}
  Let $P\subset M_{\mathbb{R}}$ be an $n$-dimensional Gorenstein
  polytope of degree $2$. Then for any lattice subpolytope $P' \subset
  P$ one has $\nv(P') \leq |P' \cap M| - n +1$. Moreover, one has:
\begin{enumerate}
\item $\nv(P') =  |P' \cap M | - n +1$ if and only if
$P'$ a Gorenstein polytope of degree $2$;
\item $\nv(P') =  |P' \cap  M| - n$
 if and only if $\deg P' \leq 1$.
\end{enumerate}
\label{latt}
\end{prop}

\begin{proof}
  We use the formula for the $h^*$-polynomial of $P$
  \[ h^*_P(t) = 1 + (|P\cap M| - n-1)t + t^2 \]
  and the equality $\nv (P) = h^*_P(1) = |P\cap M| - n+1$.  
By \ref{monotonyhvector}, we obtain that either $\deg P' =2$
  and $P'$ is a Gorenstein polytope (see \ref{h*dual}), 
or $\deg P' \leq 1$. In the first case,
  we obtain $\nv(P') = |P' \cap M| - n +1$ (as for $P$ above).  In the
  second case, the $h^*$-polynomial of $P'$ is linear and we obtain $\nv (P')
  = h^*_{P'}(1) = |P' \cap M| - n$.
\end{proof}

Let us  make some general remarks about pyramids:

\begin{prop}
A $n$-dimensional  Gorenstein polytope $P$ of degree $2$ is a
pyramid $\Pi(P')$ 
over an $(n-1)$-dimensional Gorenstein polytope $P'$ of degree $2$
if and only if
there exists a proper face $Q \subset P$ such that $\deg Q = 2$.
\label{pyr1}
\end{prop}

\begin{proof}
  One direction of this statement (i.e., if $P$ is a pyramid) is obvious.
  
  Now we assume that a proper face $Q \subset P$ such that $\deg Q = 2$.  Then
  $Q$ is contained in a facet $\Gamma \subset P$. By  
\ref{monotonyhvector}, $\deg Q \leq \deg \Gamma
  \leq \deg P$ and  we conclude that $\deg \Gamma = 2$. Thus, without loss of
  generality we can assume that $Q$ has codimension $1$ (i.e. $Q = \Gamma$).
  Moreover, since $h^*$-polynomial is monotone, the leading coefficient of
  $h^*$-polynomial of $\Gamma$ must be $1$. Thus, $\Gamma$ is a Gorenstein
  polytope of degree $2$ (see \ref{h*dual}). 
Let $p$ be the unique interior lattice point of
  $(n-2)\Gamma$. First we show that there exists only one lattice point $q$ of
  $P$ which is not contained in $\Gamma$. If there were two different lattice
  points $q_1, q_2 \in P \setminus \Gamma$, then we would have two different
  lattice points $q_1 + p$ and $q_2 + p$ in the interior of $(n-1)P$
  (contradiction). Therefore, $q$ is a vertex of $P$ which is  
combinatorially equivalent to a
  pyramid over $\Gamma$. It remains to show that the integral height $h$ of
  the vertex $q$ in this pyramid is $1$.  One has $\nv(P) = h \nv(\Gamma)$.
  One the other hand, by \ref{latt},  we have
  \[ \nv(P) = |P \cap M| -n +1, \;\; \nv(\Gamma) = |\Gamma \cap M| -n. \]
  We have already shown that $ |P \cap M| = |\Gamma \cap M| +1$.
  Therefore, $\nv(P) = \nv(\Gamma)$ and $h =1$.
\end{proof}

\begin{coro}
  A Gorenstein polytope $P$ of degree $2$ is not a pyramid if and only if all
  proper faces of $P$ and $P^*$ have degree $\leq 1$.
\label{pyr3}
\end{coro}

\begin{proof} It follows immediately from \ref{dual-pyramid} and 
\ref{pyr1}. 

\end{proof}

From now until the end of this section we assume that 
$P \subset M_{\mathbb{R}}$ is  a $3$-dimensional
Gorenstein polytope of degree $2$.

\begin{prop}  If $P$ is a $3$-dimensional Gorenstein 
 polytope of degree $2$, then
  \[ \sum_{E \subset P, \dim E=1} \nv(E) \cdot \nv(E^{\ast}) 
= 12.\]
\label{12-3dim}
\end{prop}
\begin{proof}
Let $\Delta$ be  a reflexive
polytope of dimension $3$.  Then
\[ \sum_{F \subset \Delta, \dim F=1} \nv(F) \cdot 
\nv(F^{\ast}) = 24,\]
where $F^* \subset \Delta$ denote the dual to $F$  $1$-dimensional face 
of the dual reflexive polytope $\Delta^*$. This equality follows 
from the formula for the Euler number of Calabi-Yau hypersurfaces in toric 
varieties \cite[Cor. 7.10]{BD96} and the fact that the Euler number 
of $K3$ surfaces is $24$ (see also \cite{Haa05}). 

We apply this formula to the Gorenstein polytope 
$\Delta = 2P$. Then $F=2E$ and hence 
$\nv(F) = 2 \nv(E)$. On the other hand, 
by \ref{duality-Gor}, the dual $1$-dimensional faces $F^*$ and 
$E^*$ are isomorphic as lattice polytopes, i.e., $\nv(F^*) = \nv(E^*)$. 
So we get 
\[ 24 =  \sum_{F \subset \Delta, \dim F=1} \nv(F) \cdot 
\nv(F^{\ast}) =  \sum_{E \subset P \dim E=1}  2 \nv(E) \cdot 
\nv(E^{\ast}). \]
Dividing  by $2$, we obtain  the required equality. 
\end{proof}

\begin{exam}
{\rm  Let $\Delta$ be a $2$-dimensional reflexive polytope and 
$\Delta^*$ its dual. By \ref{dual-pyramid}, 
$P:= \Pi(\Delta)$ and $P^* = \Pi(\Delta^*)$ 
are two dual to each other $3$-dimensional 
Gorenstein polytopes of degree $2$. A $1$-dimensional face $E \subset P$ 
is either an edge connecting the vertex of the pyramid $\Pi(\Delta)$ 
to a vertex $v_i \in 
\Delta$ ($\nv(E)=1$), or a $1$-dimensional face of $\Gamma_j \subset 
\Delta$. Then the dual $1$-dimensional face $E^* \subset P^*$ is respectively 
either $1$-dimensional face $v_i^* \subset \Delta^*$, or an edge connecting
the vertex of the pyramid $\Pi(\Delta^*)$  with dual 
vertex $\Gamma_j^* \in \Delta^*$  ($\nv(E^*)=1$). 
Therefore  
\[ \sum_{E \subset P, \dim E=1} \nv(E) \cdot \nv(E^{\ast}) 
= \sum_{i} 1 \cdot \nv(v_i^*) + \sum_j \nv(\Gamma_j) \cdot 1. \]
By \ref{face-vol}, we have 
\[ \nv(\Delta) =  \sum_j \nv(\Gamma_j), \;\; \nv(\Delta^*) = 
\sum_{i} \nv(v_i^*). \]
So for $3$-dimensional  Gorenstein polytopes 
 $P:= \Pi(\Delta)$ and $P^* = \Pi(\Delta^*)$  the equality in  
\ref{12-3dim} is equivalent to the well-known identity for reflexive polygons:
\[ \nv(\Delta) + \nv(\Delta^*) = 12. \]
}
\end{exam}

\begin{exam}
{\rm  Another example is shown in Figure 3 below. The Gorenstein 
simplex  $P$ has $6$ edges
  $E_1, \dots, E_6$ $(\nv(E_i) =1)$. Its dual polytope
  $P^*$ has also $6$ edges $E_1^*, \ldots, E_6^*$  $(\nv(E_i^*) =2)$. 
Therefore, we have 
\[\sum_{E \subset P, \dim E =1} \nv(E) \cdot
  \nv(E^{\ast}) = \sum_{i=1}^6 \nv(E_i) \cdot
  \nv(E^{\ast}_i) = 6 \cdot (1 \cdot 2) = 12. \]
  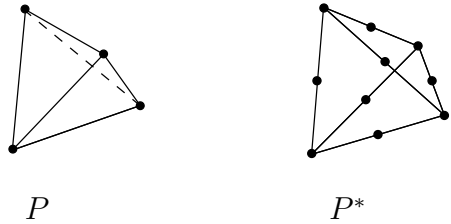
\begin{figure}[h] \label{ExDualPolytopesDim3}
    \centering
    \input{ExDualPolytopesDim3.pstex_t}
    \caption{Dual Gorenstein polytopes in dimension $3$}
  \end{figure} }
\end{exam}

For any face $F \subset P$ we denote by ${\rm Int}(E)$ the number 
of lattice points in the relative interior of $F$.

\begin{prop} \label{propint} Let  $P$ be  a $3$-dimensional Gorenstein
  polytope of degree $2$ which is not a
  pyramid. Then
\[ |P \cap M| + |P^* \cap N | + \sum_{E \subset P} {\rm Int}(E)
\cdot {\rm Int}(E^*) = 14. \]
\end{prop}
\begin{proof}

If $P$ is not a pyramid, then all lattice points of $P$ are contained
in edges (see \ref{pyr3}), i.e.,
\[ |P \cap M | = v + \sum_{E \subset P} {\rm Int}(E), \]
where $v$ denotes the number of vertices of $P$. Denote by
$e$ and $f$ the numbers of edges and faces of $P$ respectively.
By Euler  formula, $e = v+f -2$. Applying \ref{12-3dim} and $\nv(E) = 
{\rm Int}(E)+1$, one has
  \begin{eqnarray*}
    12 & = & \sum_{i=1}^e ({\rm Int}(E_i) +1) ({\rm Int}(E^{\ast}_i) + 1)\\
    & = & e + \sum_{i=1}^e ({\rm Int}(E_i) + {\rm Int}(E^{\ast}_i) +
 {\rm Int}(E_i) \cdot 
   {\rm Int} (E^{\ast}_i) ) \\
    & = &  v + f - 2 + 
\sum_{i=1}^e {\rm Int}(E_i) +  \sum_{i=1}^e {\rm Int}(E^{\ast}_i) +
     \sum_{i=1}^e {\rm Int}(E_i)\cdot  {\rm Int}(E^{\ast}_i).
  \end{eqnarray*}
Using  $|P \cap M| = \sum_{i=1}^e {\rm Int}(E_i) + v$ and $|P^* \cap N | = 
 \sum_{i=1}^e {\rm Int}(E^{\ast}_i) + f$, we obtain the statement. 
\end{proof}\vspace*{1ex}

\begin{exam}
{\rm   The dual to each other polytopes $P$ and $P^*$ 
in Example  \ref{ExDualPolytopesDim3}
 contain  respectively  $4$ and $10$ lattice points. Since  all
  lattice points of $P$ are vertices, so we obtain  
$\sum_{E \subset P}
  {\rm Int}(E) \cdot {\rm Int}(E^*) = 0$, i.e.,  
 \[|P \cap M| + |P^*
  \cap N | + \sum_{E \subset P}{\rm Int} (E) \cdot {\rm Int}(E^*) =
   4 + 10 + 0 = 14.\]}
\end{exam}

\begin{prop}
  If $P$ is a $3$-dimensional Gorenstein polytope 
of degree $2$ which is not a pyramid, then
\[ \nv(P) + \nv(P^*) \leq 10. \]
\label{num10}
\end{prop}
\begin{proof}
  Recall that $\nv(P) = |P \cap M|-2$ and $\nv(P^*) = |P^* \cap
  M|-2$. By \ref{propint} and by  the obvious inequality
  \[  \sum_{E \subset P}{\rm Int}(E) \cdot {\rm Int}(E^*) \geq 0, \]
  we obtain the statement.
\end{proof}

\section{Minimal Gorenstein polytopes and the classification}
\label{VolMinPoly}

\begin{dfn} Let  $P \subset M_{\mathbb{R}}$ be an arbitrary 
$n$-dimensional Gorenstein polytope of
degree $2$ with vertices $\{ v_1, \dots, v_k\}$. 
We call $P$ {\sl minimal} if for all $j \in \{ 1, \dots, k\}$ the
 convex hull of the set   
$\{v_1, \dots, v_k\} \backslash
  \{v_j\}$ is not a Gorenstein polytope of degree $2$. 
  Analogously, we call $P$ {\sl maximal}, if  for any lattice point 
 $v_{k+1} \in M
  \backslash \{ v_1, \dots, v_k\}$ 
the convex hull of the lattice points  $\{v_1,
  \dots, v_k, v_{k+1}\}$ is not Gorenstein polytope of degree $2$.
\end{dfn}

\begin{rem} 
{\rm If a  $n$-dimensional Gorenstein polytope $P$ of degree $2$ is 
isomorphic to a subpolytope of  
another  $n$-dimensional Gorenstein polytope $Q$ of degree $2$, then 
we write $P \preccurlyeq Q$. If $P \subseteq Q$, then    
we immediately obtain that 
both reflexive polytopes $(n-1)P$ and 
$(n-1)Q$ contain 
a {\it common} unique interior lattice point $m \in M$. 
By  \ref{sublattice}, we obtain 
that the dual 
Gorenstein polytope $Q^*$ is naturally contained in the dual polytope 
$P^*$, i.e.  $Q^* \preccurlyeq P^*$.  Therefore the duality  $P 
\leftrightarrow P^*$ 
establishes a bijection between the set of minimal and the set 
of maximal Gorenstein polytopes:
\begin{eqnarray*}
  \{ \text{minimal Gorenstein polytopes} \} &\leftrightarrow& 
\{ \text{maximal
  Gorenstein polytopes} \}.
\end{eqnarray*}
}
\end{rem} 

\begin{prop}
  Let $v_i \in P$ be an arbitrary vertex of a $n$-dimensional
  Gorenstein polytope $P$ of degree $2$ and let $\Gamma_i^* \subset P^*$
  be the facet of the dual polytope $P^*$ which is dual to $v_i$. Then
  \[ P_i:= {\rm conv} (\{ P \cap M \} \setminus \{v_i \}) \]
  is a Gorenstein polytope of degree $2$ if and only if the facet 
$\Gamma_i^*$ is a
  standard $(n-1)$-dimensional basic 
simplex.  In particular, $P$ is minimal if and
  only if $\nv(\Gamma^*) \geq 2$ for all facets $\Gamma^* \subset P^*$.
\label{minim}
\end{prop}
 
\begin{proof}
  By  \ref{volume}, we obtain  that
  $P_i$ is Gorenstein of degree $2$ if and only if $$\nv(P_i) = |P_i
  \cap M| - n + 1 = |P \cap M| - n - 2 = \nv(P)
  -1.$$  The latter  holds true if and only if $P_i$ is obtained 
from $P$ by cutting out a $n$-dimensional lattice polytope $S_i$ with  
$\nv(S_i) =1$  (i.e., $S_i$ is a basic simplex). In this case, the 
simplex $S_i$ is 
a convex hull of $v_i \in P$ and a simplicial facet $\Gamma_i \subset P_i$, 
where $\nv(\Gamma_i)=1$, i.e., $\Gamma_i$  is a
  standard $(n-1)$-dimensional basic simplex.
\end{proof}

\begin{coro} \label{sum-vol}
If $P$ and $Q$ two $n$-dimensional Gorenstein polytopes 
of degree $2$ such that  $Q \preccurlyeq P$. Then 
 
\[ \nv(P) + \nv(P^*) =  \nv(Q) + \nv(Q^*). \]
\end{coro}

\begin{proof}
We use the induction on $k = |P \cap M| - |Q \cap M|$. 
If $k =0$ then $P = Q$ and there is nothing to prove. 
If $k \geq 1$ then $P  \neq Q$ and 
there exists a vertex $v_i \in P$ which is not contained in 
$Q$. Let  $P_i:= {\rm conv} (\{ P \cap M \} \setminus \{v_i \})$. 
Then $Q \preccurlyeq P_i$. Moreover, the proof of \ref{minim} shows that 
$P_i$ is obtained 
from $P$ by cutting out a $n$-dimensional lattice basic simplex  
$S_i = {\rm conv}(\Gamma_i, v_i)$. 
Let $v_i^*$ be the dual vertex of $P_i$ corresponding to the facet 
$\Gamma_i \subset P_i$. Then 
 \[ P^*:= {\rm conv} (\{ P^*_i \cap N \} \setminus \{v_i^* \}). \]
As in the proof of 
\ref{minim},  we obtain 
\[ \nv(P^*) =  \nv(P^*_i) -1. \]
Together with  $\nv(P_i) =  \nv(P) -1$ this implies  
 \[ \nv(P) + \nv(P^*) =  \nv(P_i) + \nv(P^*_i). \]
Since  $Q \preccurlyeq P_i$ and $|P_i \cap M| - |Q \cap M| < 
|P \cap M| - |Q \cap M|$ we can apply the induction hypothesis 
and get 
\[ \nv(P_i) + \nv(P^*_i)  =  \nv(Q) + \nv(Q^*). \]
This proves the statement. 
\end{proof}

The main purpose of this section is to prove the following result:

\begin{theo}
Let $P$ be a minimal $n$-dimensional Gorenstein polytope of degree $2$.
Then $\nv(P) \leq 4$.
\label{theo-vol4}
\end{theo}

\begin{rem}
{\rm
If $P$ is a pyramid  $\Pi(P')$
over a $(n-1)$-dimensional Gorenstein polytope $P'$ of degree 2, then
$\nv(P) = \nv(P')$ and  minimality of $P$ is equivalent
to minimality of $P'$.
Therefore it is sufficient to prove \ref{theo-vol4}
for Gorenstein polytopes of degree $2$ which are not pyramids.}
\end{rem}

\begin{lem}
Let $P$ be a $n$-dimensional  Gorenstein polytope of degree $2$ and 
$P^*$ its dual. Assume that $$\nv(P^*) \leq 4.$$ Then either 
$P$ is not minimal, or
$\nv(P) \leq 4$.
\label{min5}
\end{lem}

\begin{proof} 
We can apply the classification of 
all $n$-dimensional Gorenstein polytopes $P^*$ of degree $2$ satisfying the 
condition $\nv(P^*) \leq 4$ (see \ref{class-aleq2}). 

If $\nv(P^*) \in \{ 2,3 \}$ a direct check via this classification
shows that
all facets of $P^*$ are standard simplices. Thus, by
\ref{minim},  we obtain that $P = (P^*)^*$ is not not minimal.

 If $\nv(P^*) =4$, then $P^*$ may contain facets which are not
standard simplices, but if $\nv(P)  \geq 5$ 
there always exists at least one facet
which is a standard simplex. By \ref{minim}, this implies that $P$ 
is not minimal. 
\end{proof}

By Lemma \ref{min5}, the statement of   Theorem \ref{theo-vol4} follows from:

\begin{theo}
Let $P$ be a minimal $n$-dimensional Gorenstein polytope of degree $2$.
Then at least one of two inequalities  $\nv(P) \leq 4$ and
$\nv(P^*) \leq 4$ holds.
\label{vol4'}
\end{theo}

We prove Theorem \ref{vol4'} in several steps:

Let $P$ be a minimal $n$-dimensional Gorenstein polytope of degree
$2$. We fix a vertex $v_i \in P$ and the facet $\Gamma_i^* \subset P^*$
which is dual to $v_i$. Without loss of generality we assume that $v_i
=0$, i.e., $C(v_i) \supset P$ is a cone with vertex $0$.

Since $P^*$ is not a pyramid, by Corollary  \ref{pyr3}, 
we obtain that $\Gamma_i^*$ is a $(n-1)$-dimensional
Gorenstein polytope of degree $\leq 1$. By \ref{minim}, 
the facet $\Gamma^*_i$ can not have  
degree $0$, i.e., $\Gamma^*_i$ is not a $(n-1)$-dimensional 
 basic lattice simplex. 
Therefore, we have $\deg \Gamma_i^* =1$. 
By the complete
classification of lattice polytopes of degree $1$ (see \ref{deg1}), 
$\Gamma_i^*$ is
either an $(n-1)$-dimensional exceptional simplex, or a Lawrence prism
with heights $h_1, \ldots, h_{n-1}$ such that $\nv(\Gamma_i^*) =
\sum_{j=1}^{n-1} h_j \geq 2$.

\begin{lem}
If $\Gamma_i^*$ is  a $(n-1)$-dimensional
exceptional simplex, then $\nv(P) = 2$.
\end{lem}

\begin{proof}
Using \ref{Ci}, one obtains that the cone $C(v_i) \supset P$
is isomorphic to the simplicial cone generated by
$2e_1 -e_2-e_3, e_2, e_3, \ldots, e_n$, where $e_1, \ldots, e_n$ is a
basis of $M$. Then $n+1$ lattice points $0, 2e_1 -e_2-e_3, e_2, \ldots,
e_n$ are contained in $P$ and their convex hull $P'_i$ is already a simplex 
whose  $h^*$-polynomial is $1 + t^2$. By minimality, we have
$P = P'_i$. Therefore $\nv(P) = \nv(P'_i) =2$.
\end{proof}

\begin{lem} 
If $\Gamma_i^*$ is a  $(n-1)$-dimensional
Lawrence prism with heights $h_1, \ldots, h_{n-1}$ such that
$\nv(\Gamma^*_i) = \sum_i h_i \geq 3$, then $\nv(P) \in \{ 3,4 \}$.
\end{lem}

\begin{proof}
Using \ref{Ci}, one obtains that the cone $C(v_i) \supset P$
is isomorphic to the cone generated by $n+1$ vectors
\[  e_1, \ldots, e_n, -e_n + \sum_{i=1}^{n-1} h_i e_i, \]
where
$e_1, \ldots, e_n$ is a basis of $M$.
We denote by $P'_i$ the convex hull of $0$ and $n+1$ lattice points
$e_1, \ldots, e_n, -e_n + \sum_{i=1}^{n-1} h_i e_i$.  Then a direct
calculation shows that $\nv(P'_i) = \sum_{i =1}^{n-1} h_i$.
Since $P'_i$ lies between two parallel hyperplanes $z_n = \pm 1$ and
the intersections of $P'_i$ with these hypeplanes are exactly two vertices
$p:= e_n$ and  $q:= -e_n + \sum_{i=1}^{n-1} h_i e_i$, all remaining
lattice points of $P'_i$ are contained in the hyperplane $\{z_n =0\}$.
We want to show that the set $M \cap P' \cap \{ z_n =0 \}$ is either
$\{ 0, e_1, \ldots e_{n-1} \}$, or it possibly has one more lattice point
$(p + q)/2$ (if  $(p+q)/2 \in M$). Indeed, by \ref{pyr3} all 
lattice points of $P$
are contained in $1$-dimensional faces of $P$. Since $P_i' \subset P$, 
the same is true also for $P_i'$, i.e. all lattice points of $P_i'$ are
contained in $1$-dimensional faces of $P_i'$. But $P_i'$ is a convex
hull of $n+2$ lattice points and all segments connecting two of
them (except maybe $[p,q]$) do not have interior lattice points, because 
the convex hull of  $\{ 0, e_1, \ldots e_{n-1} \}$ is a basic
$(n-1)$-dimensional simplex in the hyperplane $\{ z_n = 0\}$ and 
the $z_n$-coordinates of $p$ and $q$ are $1$ and $-1$ respectively.  
We note that  $\{ z_n =0 \} \cap [p,q] = (p + q)/2 =
(\sum_{i =1}^{n-1} h_i e_i)/2$. Thus, we obtain 
$|P'_i \cap M| \leq n+3$.  By \ref{latt}, the latter implies 
$\nv(P'_i) \leq | P'_i \cap M| - n +1 \leq 4$, i.e.
$ \sum_{i =1}^{n-1} h_i \in \{3,4\}$. If  $\sum_{i =1}^{n-1} h_i =3$, then 
$(p+q)/2 \not\in  M$ can not be a lattice point of $P'_i$ and 
$|P'_i \cap M| = n+2$.
Together with $\nv(P'_i) =3$ this implies that $P'_i$ is a Gorenstein 
polytope with $h^*$-polynomial $1 + t + t^2$.
If  $\sum_{i =1}^{n-1} h_i =4$, then by \ref{latt}, $P'_i$ must be
a Gorenstein polytope of degree $2$, i.e.,  $|P'_i \cap M| = n+3$, 
$(p+q)/2$ is a lattice point of $P'_i$, and $\nv(P_i') =4$. By minimality,
 we have
in both cases $P = P'_i$, i.e. $\nv(P) \in \{3,4 \}$.  
\end{proof}

The following statemnet 
 completes the proof of Theorem \ref{vol4'} and Theorem \ref{theo-vol4}:

\begin{lem}
Assume that  $P$ is a minimal Gorenstein
polytope of degree $2$ such that every codimension-$1$ face $\Gamma_i^* 
\subset P^*$  is a Lawrence prism with $\nv(\Gamma_i^*) =2$.
Then either $\nv(P) \leq 4$, or  $\nv(P^*) \leq 4$. 
\end{lem} 

\begin{proof} 
In this situation, all generators of the cone $C(v_i)$ are contained 
in an affine hyperplane $H_i$ which divides $P$ into two lattice
polytopes $P_i =  {\rm conv} (\{ P \cap M \} \setminus \{v_i \})$ 
and a Lawrence prism $P'_i$ containing the vertex $v_i$ 
such that $\nv(P_i') =2$. Moreover, 
$P_i'$ is a pyramid with vertex $v_i$ 
over the $(n-1)$-dimensional Laurence prism $\overline{P_i} = H_i \cap P$. 

By minimality,  $P_i$ must have degree $1$ and by \ref{deg1} 
we have the following two cases:

{\sc Case 1}:
The polytope  $P_i$ is an exceptional simplex, i.e., $\nv(P_i) =4$.  
Then its facet  $\overline{P_i}$ 
is also a simplex. Moreover, $P'_i$ must be also a simplex, since 
$P_i'$  is a pyramid over  
$\overline{P_i}$. Let $k$ be the number of vertices 
of $P$ (this is also the number of facets in $P^*$). 
The whole polytope $P$ is 
a union of simplices $P_i$ and $P_i'$.  Therefore  
$k \leq n+2$. By \ref{face-vol}, we have
\[ (n-1) \nv (P^*) = \sum_{ i =1}^{k} \nv(\Gamma_i^*) \leq  2(n+2). \]
If $n \geq  4$, then we obtain $\nv(P^*) \leq 2 + 6/(n-1) \leq 4$.
If $n =3$, then $\nv(P) = \nv(P_i') + \nv(P_i) = 2+4 =6$. By \ref{num10},
we have $\nv(P^*) \leq 4$.

{\sc Case 2}: 
The polytope  $P_i$ is a Lawrence prism. Since $P_i'$ and $\overline{P_i}$ 
are both Lawrence prisms and $\nv(P_i') = \nv(\overline{P_i}) =2$ the number
of vertices of $\overline{P_i}$ is at most $n+1$. On the other hand, 
 $\overline{P_i}$ is a facet of  $P_i$ which is not a basic
 $(n-1)$-dimensional simplex. Therefore  $P_i$ contains at most $2$ more 
vertices than in   $\overline{P_i}$. Together with $v_i \in P_i'$ we see 
that $P = P_i' \cup P_i$ contains  $k \leq n+4$ vertices.
By \ref{face-vol}, we have
\[ (n-1) \nv (P^*) = \sum_{ i =1}^{k} \nv(\Gamma_i^*) \leq  2(n+4). \]

If  $n \geq  5$, then we obtain $\nv(P^*) \leq 2 + 10/(n-1) < 5$, i.e.
$\nv(P^*) \leq 4$.

If $n =4$, then $ (n-1) \nv (P^*) = \sum_{ i =1}^{k} \nv(\Gamma_i^*)$
together with $\nv(\Gamma_i^*) =2$ implies that the number $k$ of vertices of
$P$ must be divisible by $3$. Since  $k \leq n+4 =8$, we obtain 
 $k =6$ and  $\nv (P^*) =4$.

If $n =3$, then it remains to exclude only the case
$\nv(P) = \nv(P^*)=5$ (all other cases follow from \ref{num10}).
By \ref{face-vol}, in the last case both polytopes $P$ and $P^*$ must have
$5$ vertices. There exist exactly two  different combinatorial  types of 
$3$-dimensional polytopes $P$ with $5$ vertices, but only for one of these 
types the dual polytope $P^*$ has also $5$ vertices. 
Therefore, both polytopes $P$ and $P^*$ must be
combinatorially equivalent to a pyramid over a lattice
$4$-gon $S$ such that $\nv(S) =2$. This implies that  $\nv(P) =
\nv(P^*) =5$  is divisible by $\nv(S) =2$. Contradiction.
\end{proof}

\begin{prop} 
Every $n$-dimensional Gorenstein polytope $P \neq 2S_3$ 
of degree $2$ is a 
Cayley polytope $P = \Delta_1* \cdots * \Delta_r$ where $\Delta_1, 
\ldots, \Delta_r$ are plane lattice polytopes such that the Minkowski 
sum $\Delta= \Delta_1 + \cdots + \Delta_r$ is a reflexive polygon. 
\label{cay}
\end{prop} 

\begin{proof} It is enough to assume that $P$ is not a pyramid, because 
a pyramid over a Cayley polytope is again a Cayley polytope.   
On the other hand, every $n$-dimensional 
subpolytope of a Cayley polytope $\Delta_1* \cdots * \Delta_r$ 
 is again a Cayley polytope $\Delta_1'* \cdots * \Delta_r'$, where 
$\Delta'_i$ is a lattice subpolytope of $\Delta_i$ $(i =1, \ldots, r)$. 
By  
\ref{theo-vol4}, if $P$ is a maximal Gorenstein polytope of degree 
$2$, then $\nv(P^*) \leq 4$, i.e., $P^*$ is one of polytopes 
classified in \ref{class-aleq2}. 
One can check the whole list of  these Gorenstein 
polytopes $P^*$ 
and obtain that among potential maximal polytopes  $P$ 
only  the $3$-dimensional Gorenstein polytope $2S_3$ can not be described as 
a Cayley polytope. Moreover, if $v_i$ is a vertex of $2S_3$ and 
$P_i$ is the convex hull of $(2S_3 \cap M) \setminus \{v_i\}$ 
then $P_i$ is again a Cayley polytope. This shows that any proper 
Gorenstein subpolytope of $2S_3$ is a also a Cayley polytope.    
\end{proof}

\begin{theo}
There exist exactly $37$ $n$-dimensional Gorenstein polytopes of degree 
$2$ which are not pyramids:\\
$(1)$ $16$ reflexive polygons; \\
$(2)$ $3$-dimensional Gorenstein polytopes $P_1, \ldots, P_{15}$; \\
$(3)$ $4$-dimensional  Gorenstein polytopes $Q_1, \ldots, Q_5$; \\
$(4)$  $5$-dimensional  Gorenstein  polytope $R_1$. 
\label{class-GorDP}
\end{theo} 

\begin{proof}
There are two ways to prove this statement. 

The first one is to use \ref{cay} and the classification of all 
possible Minkowski sum decompositions of reflexive polygons due to 
Emiris and Tsigaridas \cite{ET06}.  

The second one is to enumerate all Gorenstein subpolytopes of degree 
$2$ of maximal Gorenstein polytopes of degree $2$. The maximal Gorenstein 
polytopes  are contained 
in the list of all $P$ such that $a_{P^*}= \nv(P^*) -2 \leq 2$ 
(see \ref{class-aleq2}). 

Both ways do not need any computer and can be easily realized.    
\end{proof}

Below we give  a complete list of $n$-dimensional $(n \geq 3)$ Gorenstein
polytopes $P$ of degree  $2$ which are not pyramids. This list  includes
some of  combinatorial invariants of $P$ 
 such as its $f$-vector $f = (f_0, \dots,
f_n)$ (where $f_i$ denotes the number of $i$-dimensional faces of $P$), the 
dual partner $P^*$, the lattice normalized volume 
$\nv(P)$ as well as 
all possible descriptions of $P$  as a Cayley polytope.


\begin{center}

\begin{tabular}[h]{cc}
  \hspace*{0.1cm} \rotatebox{90}{{\epsfig{file=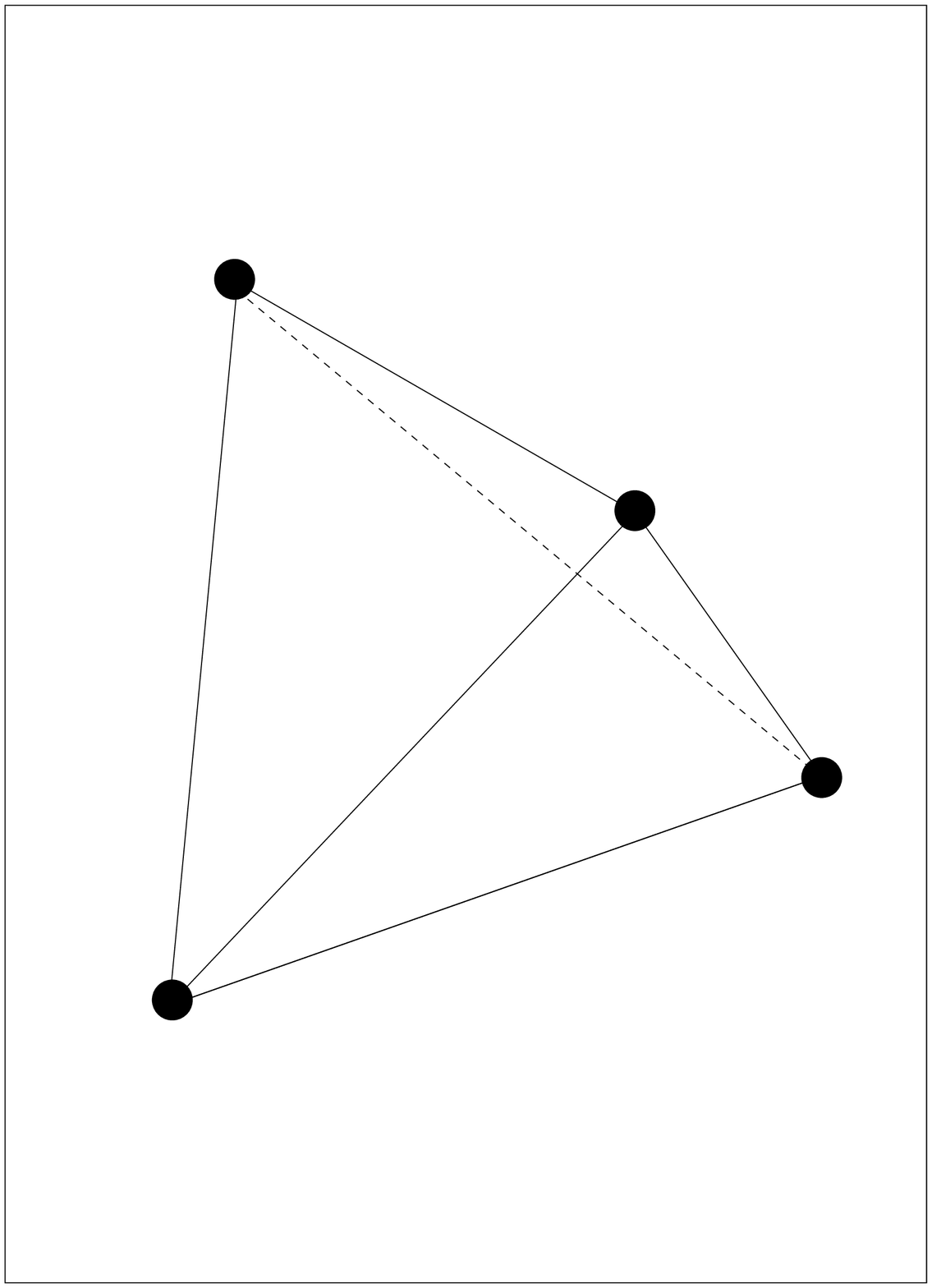, height=4cm}}}
  \hspace*{1.9cm}& 
  \rotatebox{90}{{\epsfig{file=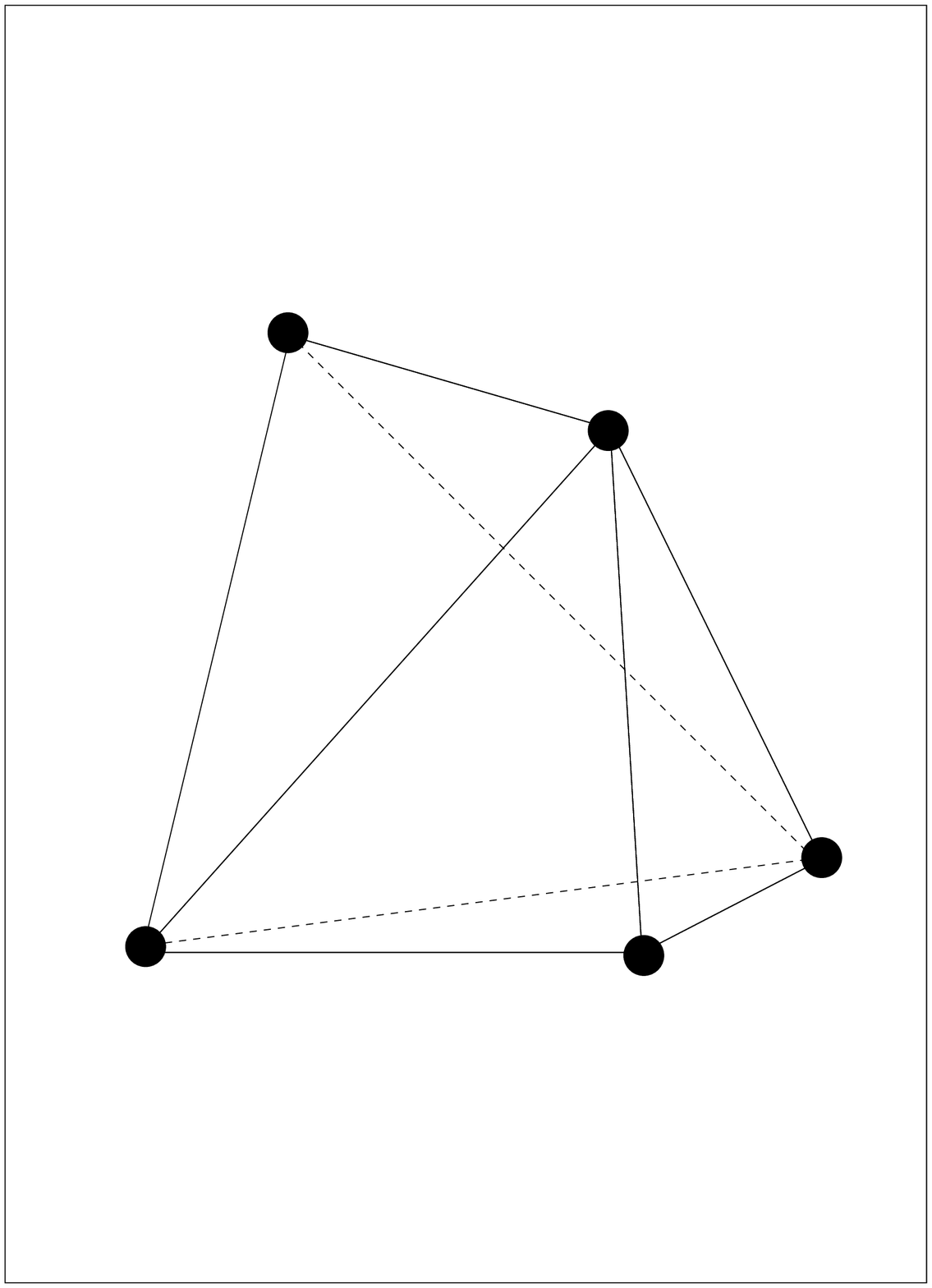, height=4cm}}}\\  
  \begin{minipage}[h]{4cm}
    \begin{tabular}[h]{cl}
     \raisebox{.5ex}[-1.5ex]{$P_1$}  &
      {  {\epsfig{file=side4.eps, height=.5cm}} 
      \raisebox{1ex}{$\ast$} 
      {\epsfig{file=side3.eps, height=.5cm}}}\\ \hline
    &$f = (1,4,6,4,1)$ \rule{0em}{1em}\\
    & $ P_{1}^{\ast} = P_{15} $\\
    & $\nv(P_1) = 2$
  \end{tabular}
  \end{minipage}\hspace*{2cm}
   &
  \begin{tabular}[h]{cl}
    \raisebox{.5ex}[-1.5ex]{$P_2$} &
    { {\epsfig{file=side7.eps, height=.5cm}} 
      \raisebox{1ex}{$\ast$} 
      {\epsfig{file=side4.eps, height=.5cm}}} \\ \hline
    &$f = (1,5,9,6,1)$ \rule{0em}{1em}\\
    & $ P_{2}^{\ast} = P_{14} $\\
    & $\nv(P_2) = 3$
  \end{tabular}
  \\ \vspace*{1em}
\end{tabular}  

\begin{tabular}[h]{cc}
  \hspace*{0.1cm} \rotatebox{90}{{\epsfig{file=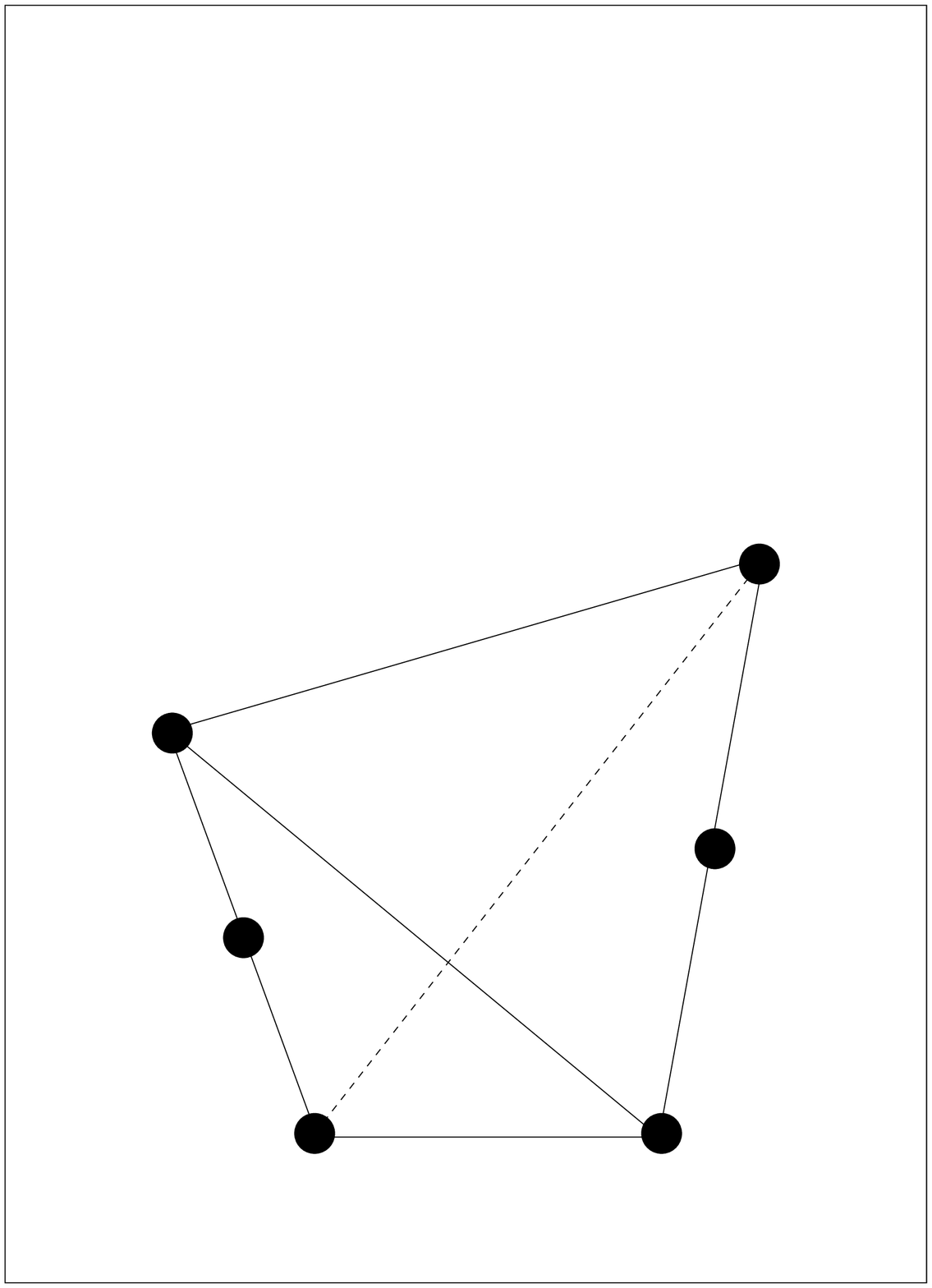, height=4cm}}}
  \hspace*{1.9cm}&
  \rotatebox{90}{{\epsfig{file=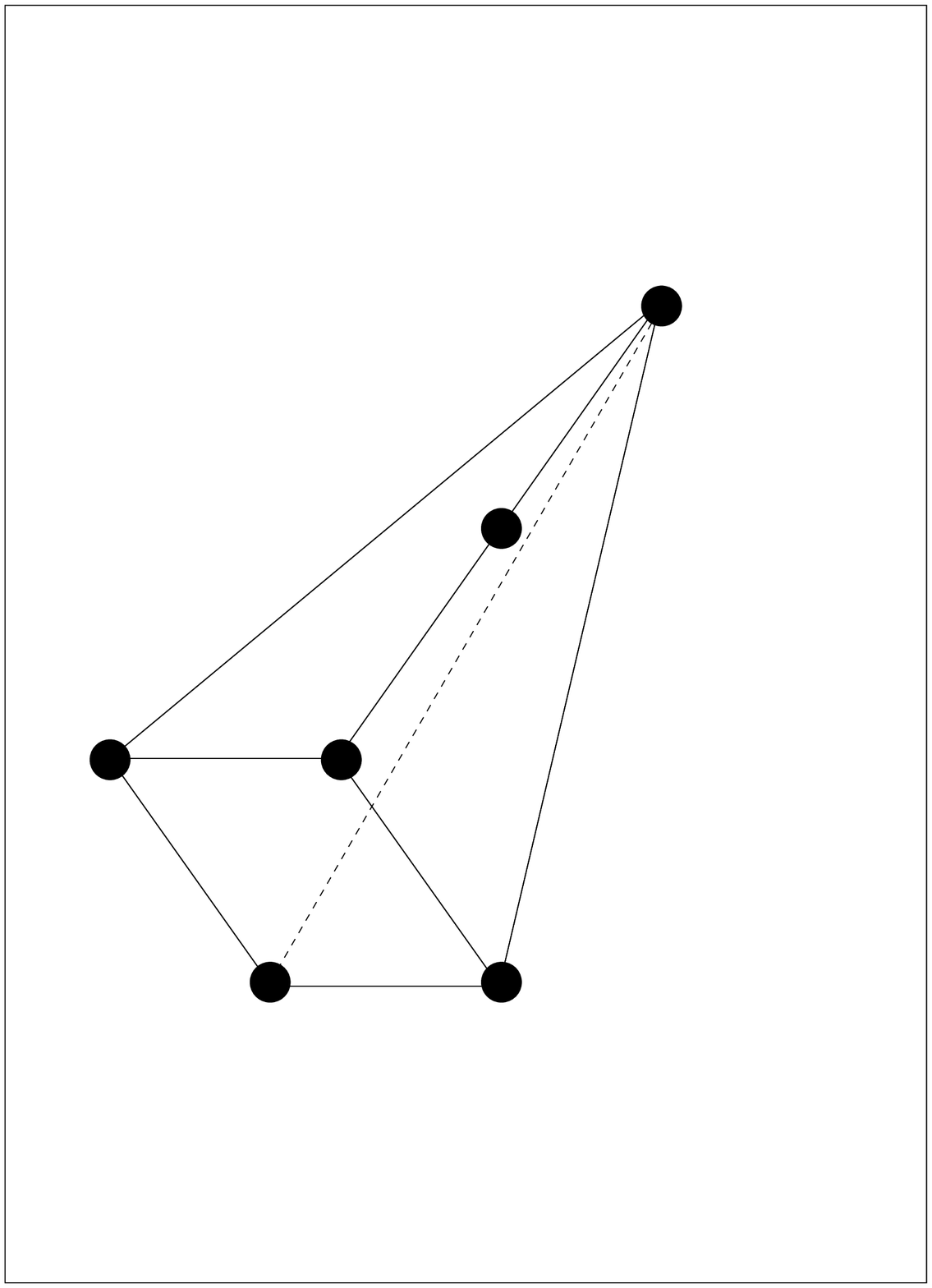, height=4cm}}} \\
  \begin{minipage}[h]{4cm}
    \begin{tabular}[h]{cl}
      \raisebox{.5ex}[-1.5ex]{$P_3$}&
      { {\epsfig{file=side6.eps, height=1cm}} 
      \raisebox{2.5ex}{$\ast$}
    {\epsfig{file=side5.eps, width=1cm}}} \\ \hline
    &$f = (1,4,6,4,1)$ \rule{0em}{1em}\\
    & $ P_{3}^{\ast} =  P_{3} $\\
    & $\nv(P_3) = 4$
  \end{tabular}
  \end{minipage}\hspace*{2cm}
   &
  \begin{tabular}[h]{cl}
    \raisebox{.5ex}[.5ex]{$P_4$}&
    {{\epsfig{file=side2.eps, height=.5cm}} 
      \raisebox{1ex}{$\ast$}
    {\epsfig{file=side10.eps, height=.5cm}} }\rule{0pt}{1cm} \\ \hline
    &$f = (1,5,8,5,1)$ \rule{0em}{1em}\\
    & $ P_{4}^{\ast} = P_{13} $\\
    &$\nv(P_4) = 4$
  \end{tabular}\\  \vspace*{1em}
\end{tabular}

\begin{tabular}[h]{cc}
  \hspace*{0.1cm} \rotatebox{90}{{\epsfig{file=P6.eps, height=4cm}}}
  \hspace*{1.9cm}&
  \rotatebox{90}{{\epsfig{file=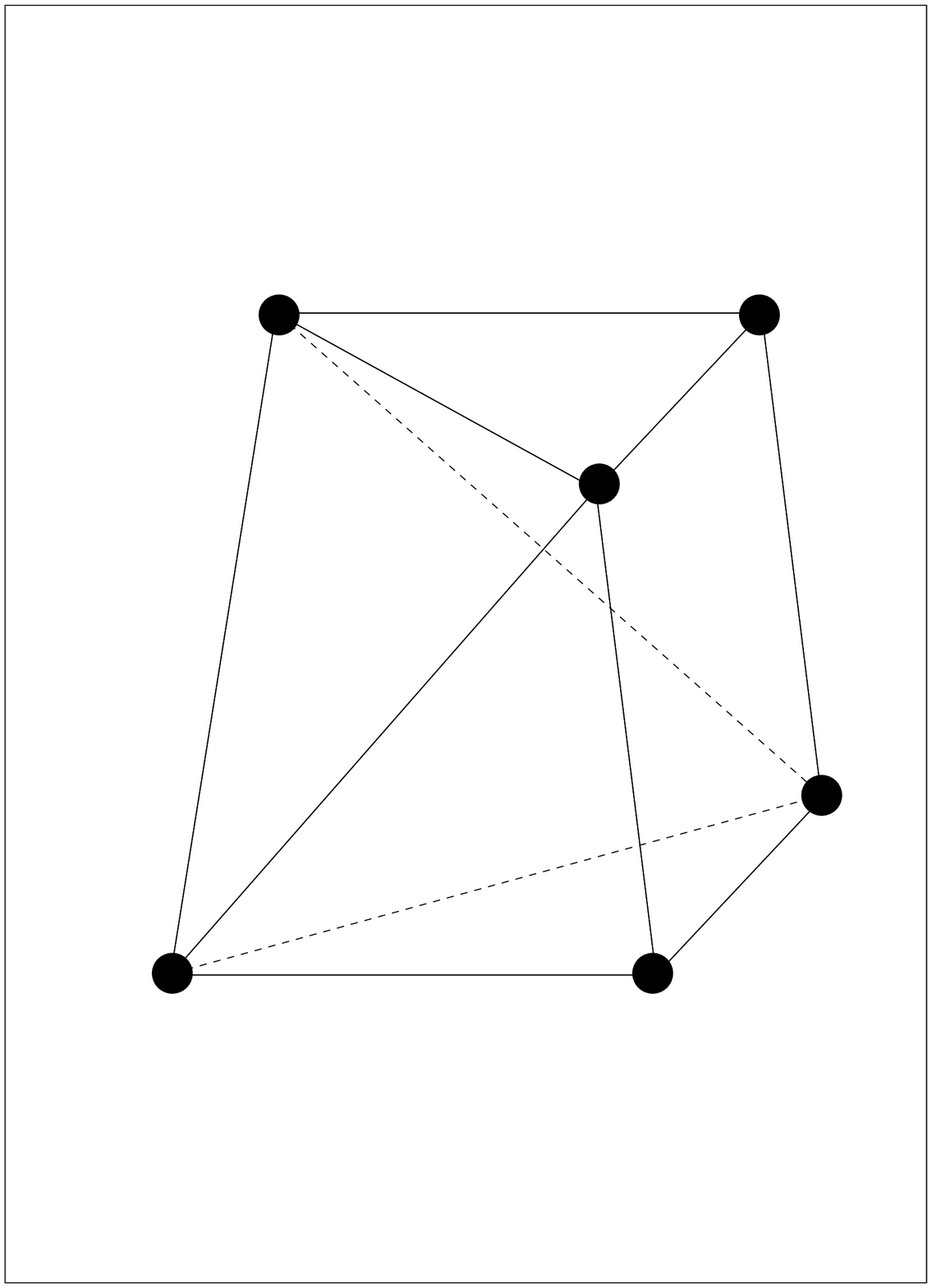, height=4cm}}} \\
  \begin{minipage}[h]{4cm}
    \begin{tabular}[h]{cl}
      \raisebox{.5ex}[-1.5ex]{$P_5$} &
      {{\epsfig{file=side16.eps, height=.5cm}}
      \raisebox{1ex}{$\ast$}
   {\epsfig{file=side2.eps, height=.5cm}} } \\ \hline
    &$f =(1,5,9,6,1) $ \rule{0em}{1em}\\
    & $ P_{5}^{\ast} = P_{10} $\\
    & $\nv(P_5) = 4$    
  \end{tabular}
  \end{minipage}\hspace*{2cm}
   &
  \begin{tabular}[h]{cl}
     &
    {{\epsfig{file=side9.eps,  height=.5cm}} 
      \raisebox{1ex}{$\ast$ }
      {\epsfig{file=side3.eps, height=.5cm}}} \\  
    \raisebox{3ex}[-1.5ex]{$P_6$} 
    & {\epsfig{file=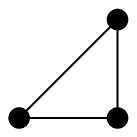, height=.5cm}} 
    \raisebox{1ex}{$\ast$ }
      {\epsfig{file=side7.eps, height=.5cm}}\\ \hline
      &$f = (1,6,11,7,1)$ \rule{0em}{1em}\\
      & $ P_{6}^{\ast} = P_{11} $\\
      & $\nv(P_6) = 4$
  \end{tabular}\\  \vspace*{1em}
\end{tabular}

\begin{tabular}[h]{cc}
  \hspace*{0.1cm} \rotatebox{90}{{\epsfig{file=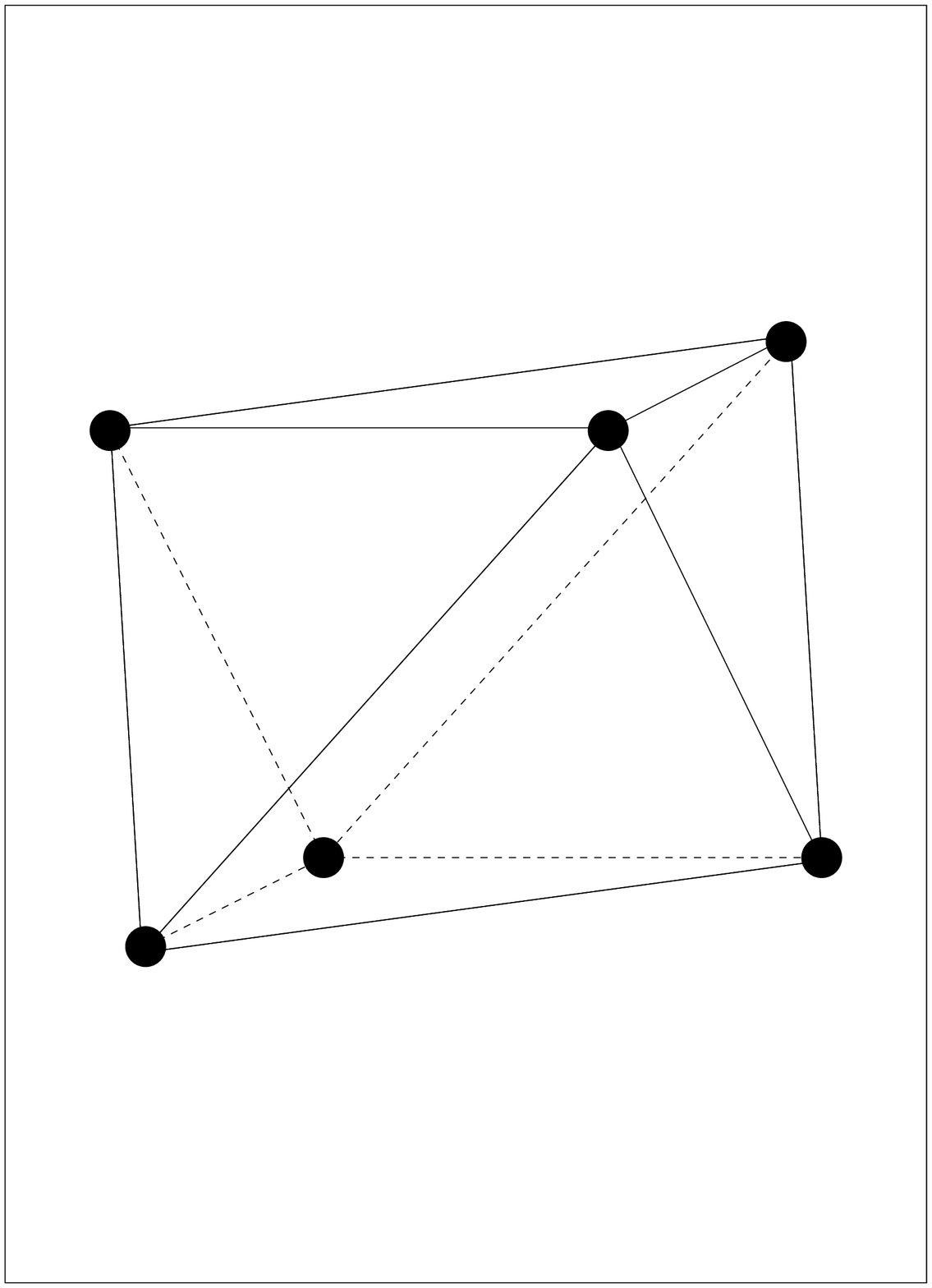, height=4cm}}}
  \hspace*{1.9cm}&
  \rotatebox{90}{{\epsfig{file=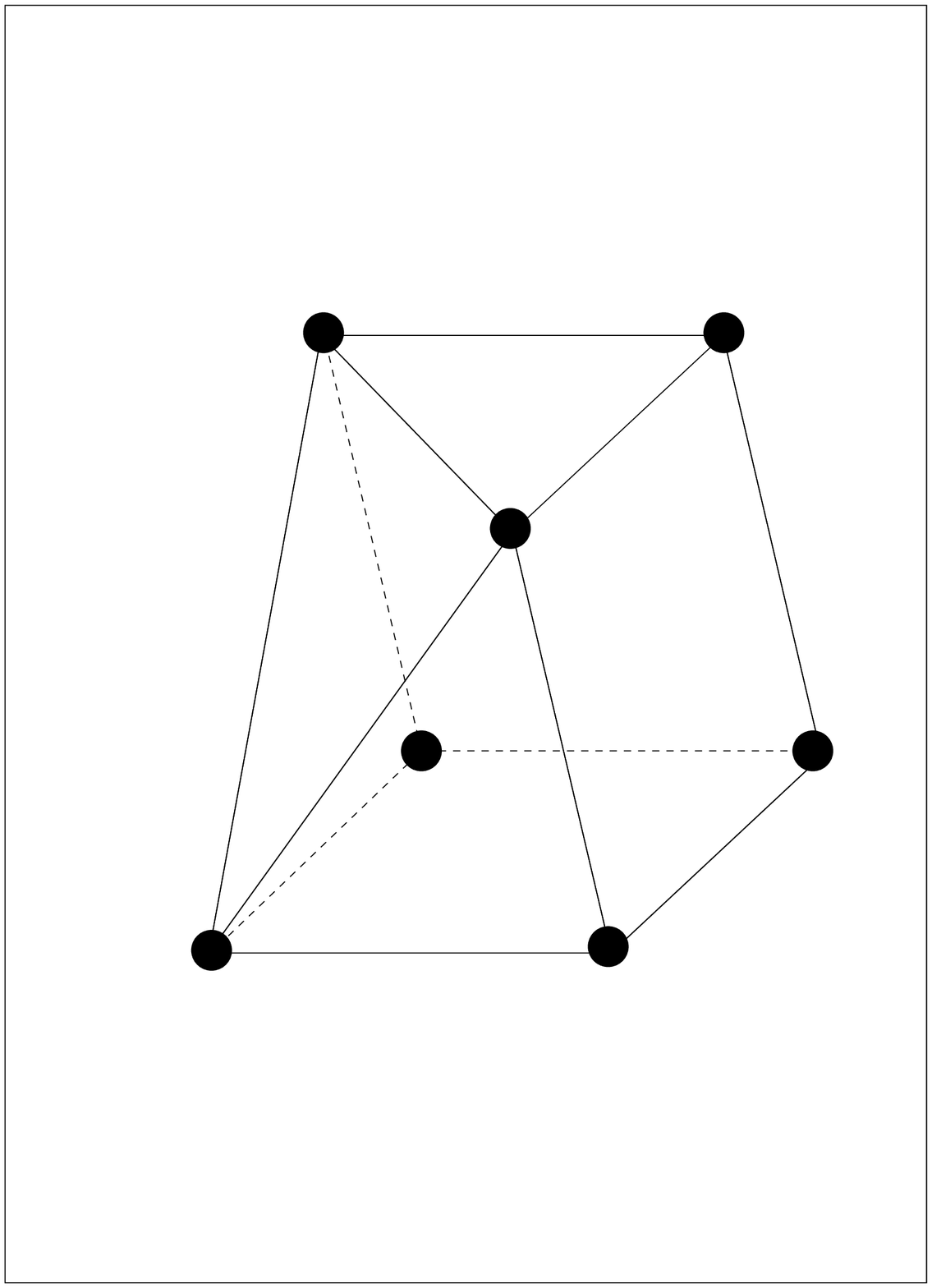, height=4cm}}} \\
  \begin{minipage}[h]{4cm}
    \begin{tabular}[h]{cl}
      \raisebox{.5ex}[-1.5ex]{$P_7$} &
      {{\epsfig{file=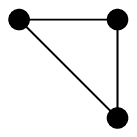, height=.5cm}} 
      \raisebox{1ex}{$\ast$}
    {\epsfig{file=side7.eps, height=.5cm}}  } \\ \hline
    &$f =(1,6,12,8,1) $ \rule{0em}{1em}\\
    & $ P_{7}^{\ast} = P_{12} $\\
    & $\nv(P_7) = 4$
  \end{tabular}
  \end{minipage}\hspace*{2cm}
   &
  \begin{tabular}[h]{cl}
    \raisebox{.5ex}[-1.5ex]{$P_8$} &
    {{\epsfig{file=side9.eps, height=.5cm}} 
      \raisebox{1ex}{$\ast$}
    {\epsfig{file=side7.eps, height=.5cm}} }\\ \hline
    &$f = (1,7,12,7,1)$ \rule{0em}{1em}\\
    & $ P_{8}^{\ast} =  P_{8} $\\
    & $\nv(P_8) = 5$
  \end{tabular}\\  \vspace*{1em}
\end{tabular}

\begin{tabular}[h]{cc}
  \hspace*{0.1cm} \rotatebox{90}{{\epsfig{file=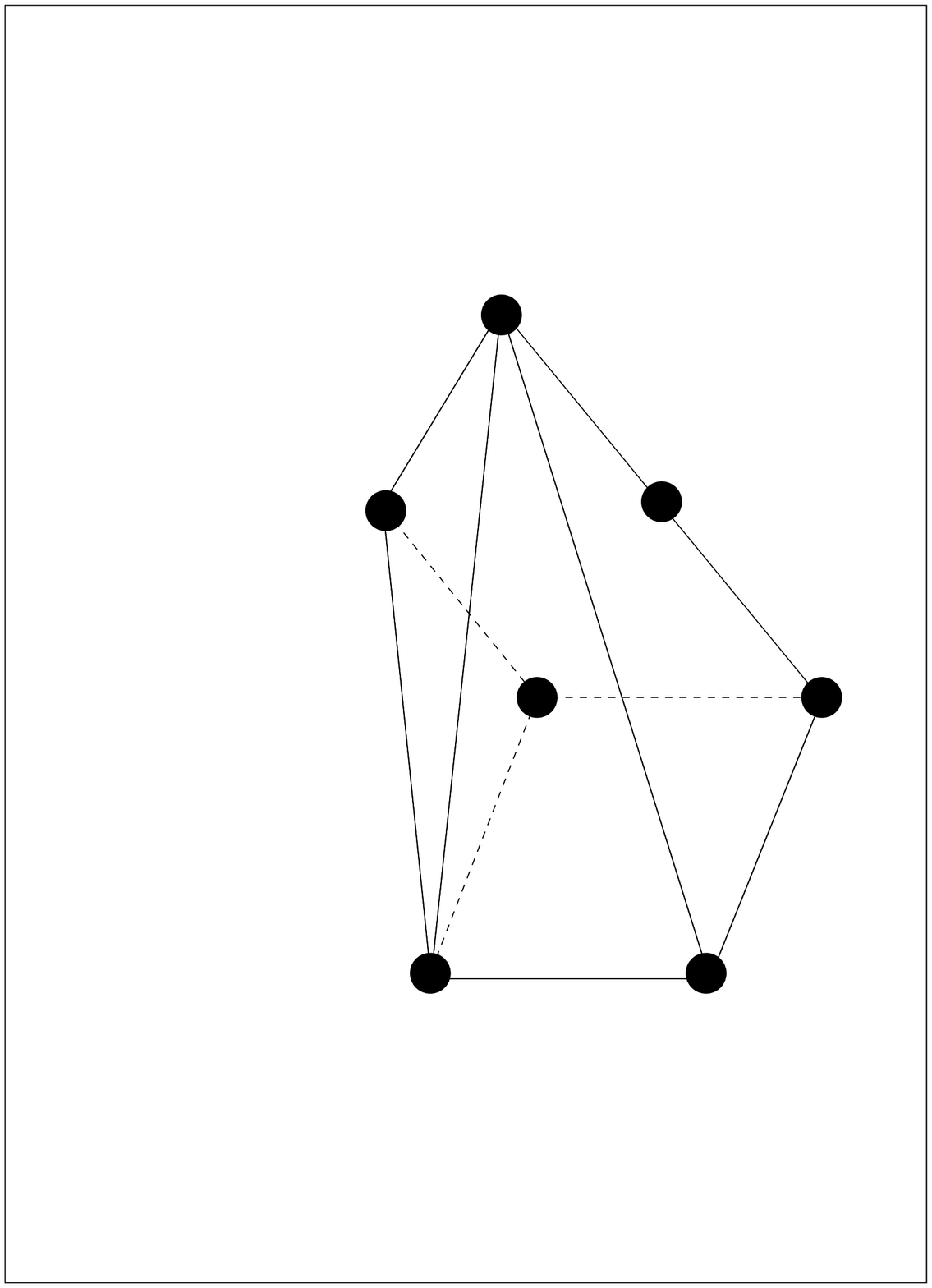, height=4cm}}}
  \hspace*{1.9cm}&
  \rotatebox{90}{{\epsfig{file=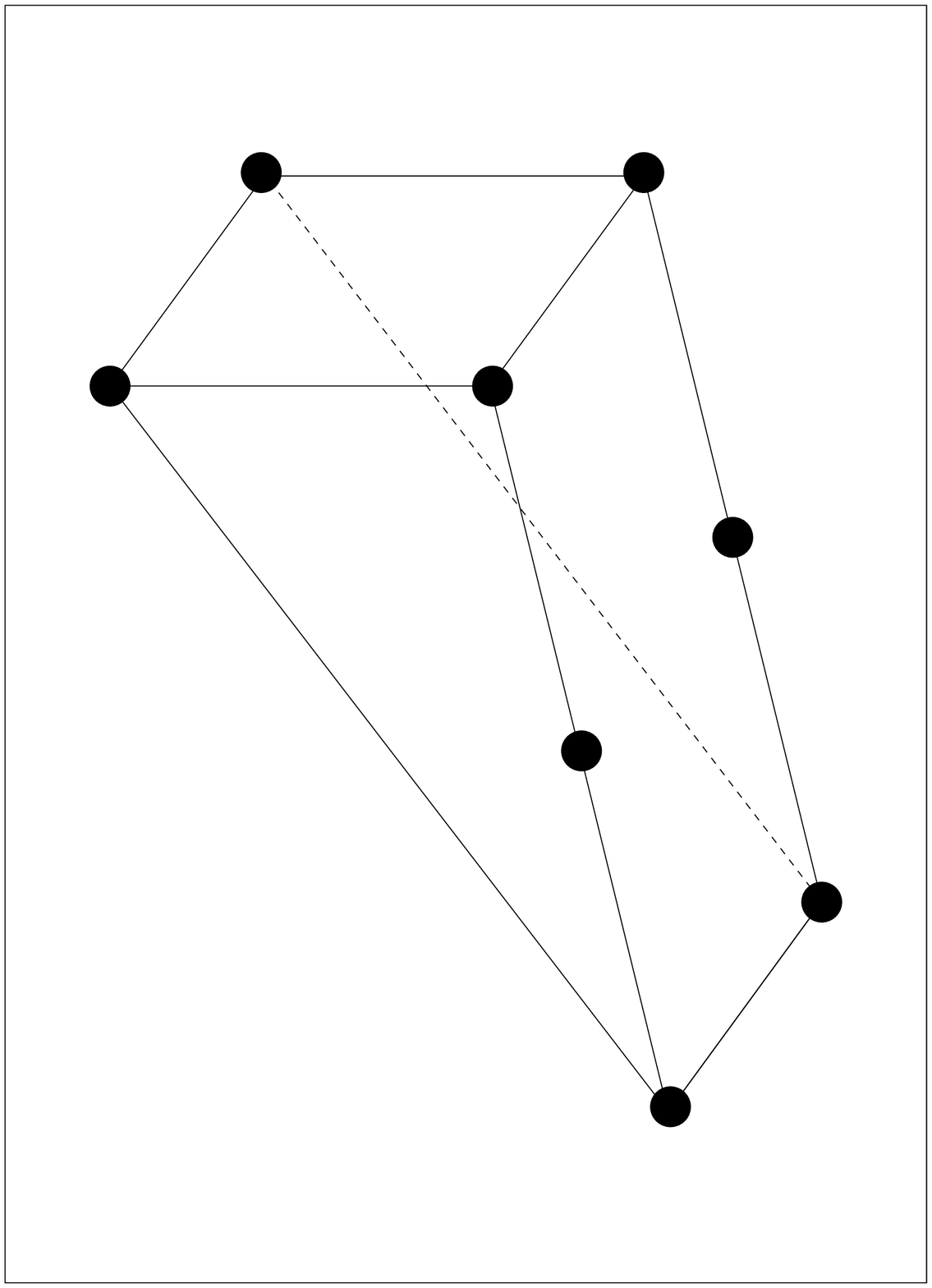, height=4cm}}} \\
  \begin{minipage}[h]{4cm}
    \begin{tabular}[h]{cl}
      \raisebox{.5ex}[-1.5ex]{$P_9$} &
      { {\epsfig{file=side10.eps, height=.5cm}} 
      \raisebox{1ex}{$\ast$}
    {\epsfig{file=side7.eps, height=.5cm}} } \\ \hline
    &$f =(1,6,10,6,1) $ \rule{0em}{1em}\\
    & $ P_{9}^{\ast} =  P_{9}  $\\
    & $\nv(P_9) = 5$    
  \end{tabular}
  \end{minipage}\hspace*{2cm}
   &
  \begin{tabular}[h]{cl}
     &
    {{\epsfig{file=side10.eps, height=.5cm}} 
      \raisebox{1ex}{$\ast$ }
      {\epsfig{file=side10.eps, height=.5cm}}}\\  
     \raisebox{3ex}[-1.5ex]{$P_{10}$} 
     & {\epsfig{file=side15.eps, height=.5cm}} \raisebox{1ex}{$\ast$}
    {\epsfig{file=side2.eps, height=.5cm}}\\ \hline
    &$f =(1,6,9,5,1) $ \rule{0em}{1em}\\
    & $ P_{10}^{\ast} =  P_{5} $\\
    & $\nv(P_{10}) = 6$
  \end{tabular}\\  \vspace*{1em}
\end{tabular}

\begin{tabular}[h]{cc}
  \hspace*{0.1cm} \rotatebox{90}{{\epsfig{file=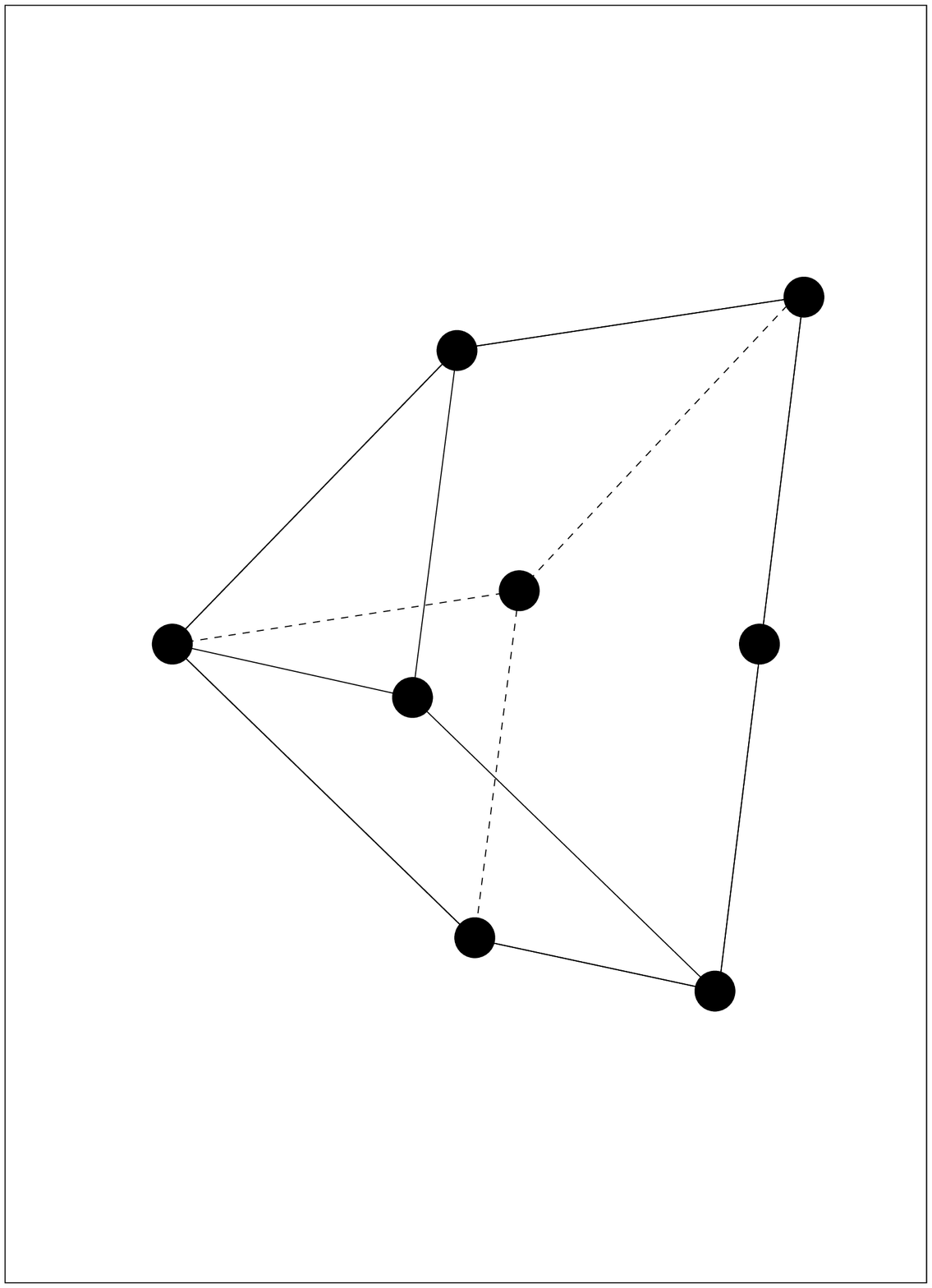, height=4cm}}}
  \hspace*{1.9cm}&
  \rotatebox{90}{{\epsfig{file=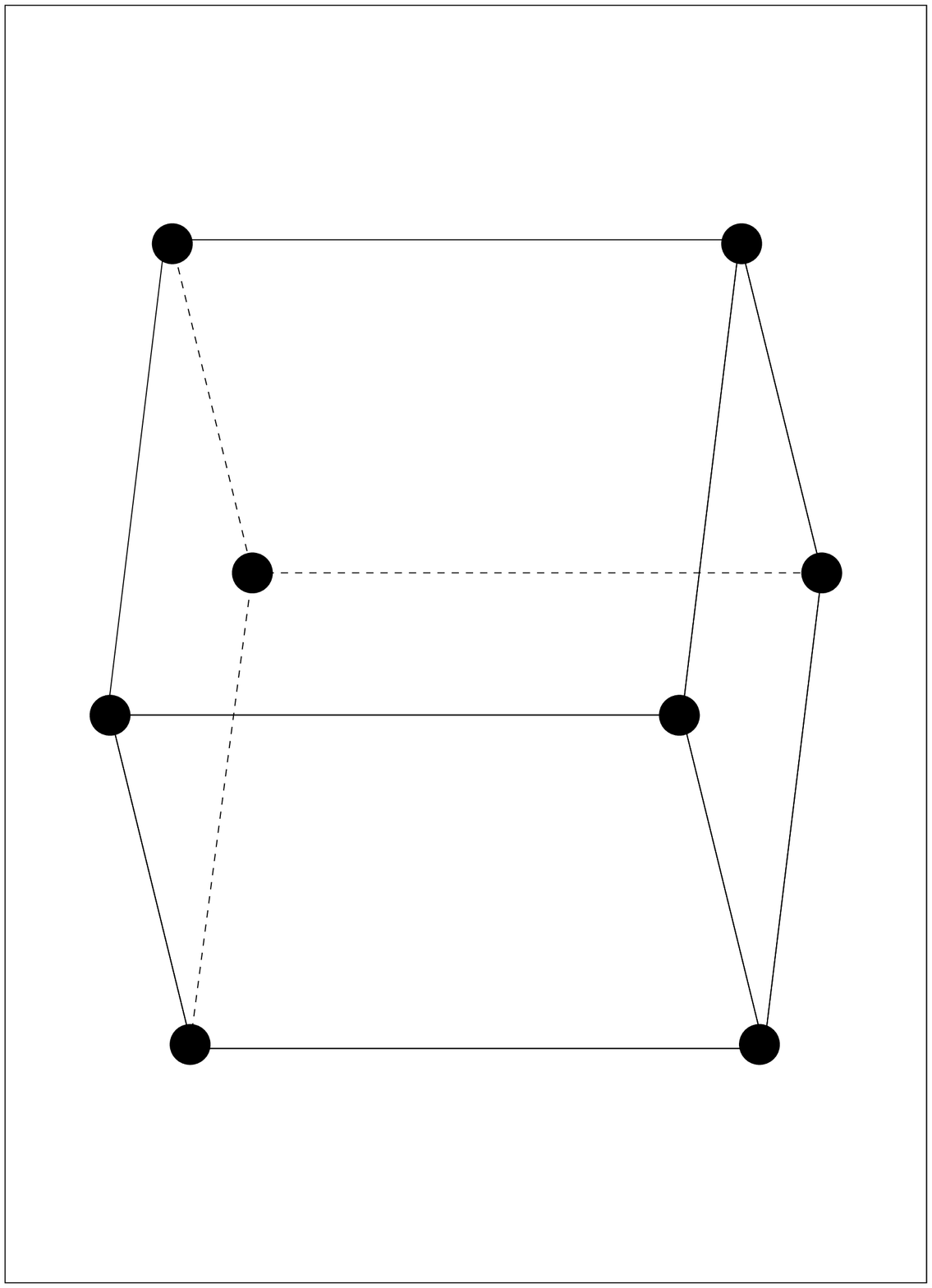, height=4cm}}} \\
  \begin{minipage}[h]{4cm}
     \begin{tabular}[h]{cl}
       \raisebox{.5ex}[-1.5ex]{$P_{11}$} &
       {{\epsfig{file=side7.eps, height=.5cm}} 
      \raisebox{1ex}{$\ast$}
    {\epsfig{file=side12.eps, height=.5cm}}  } \\ \hline
    &$f = (1,7,11,6,1)$ \rule{0em}{1em}\\
    & $ P_{11}^{\ast} = P_{6}  $ \\
    & $\nv(P_{11}) = 6$
  \end{tabular}
  \end{minipage}\hspace*{2cm}
  &
  \begin{tabular}[h]{cl}
    \raisebox{.5ex}[-1.5ex]{$P_{12}$} &
    {{\epsfig{file=side9.eps, height=.5cm}} 
      \raisebox{1ex}{$\ast$}
    {\epsfig{file=side9.eps, height=.5cm}} }\\ \hline
    &$f =(1,8,12,6,1) $ \rule{0em}{1em}\\
    & $P_{12}^{\ast}=  P_{7} $\\
    & $\nv(P_{12}) = 6$
  \end{tabular}\\  \vspace*{1em}
\end{tabular}

\begin{tabular}[h]{cc}
  \hspace*{0.1cm} \rotatebox{90}{{\epsfig{file=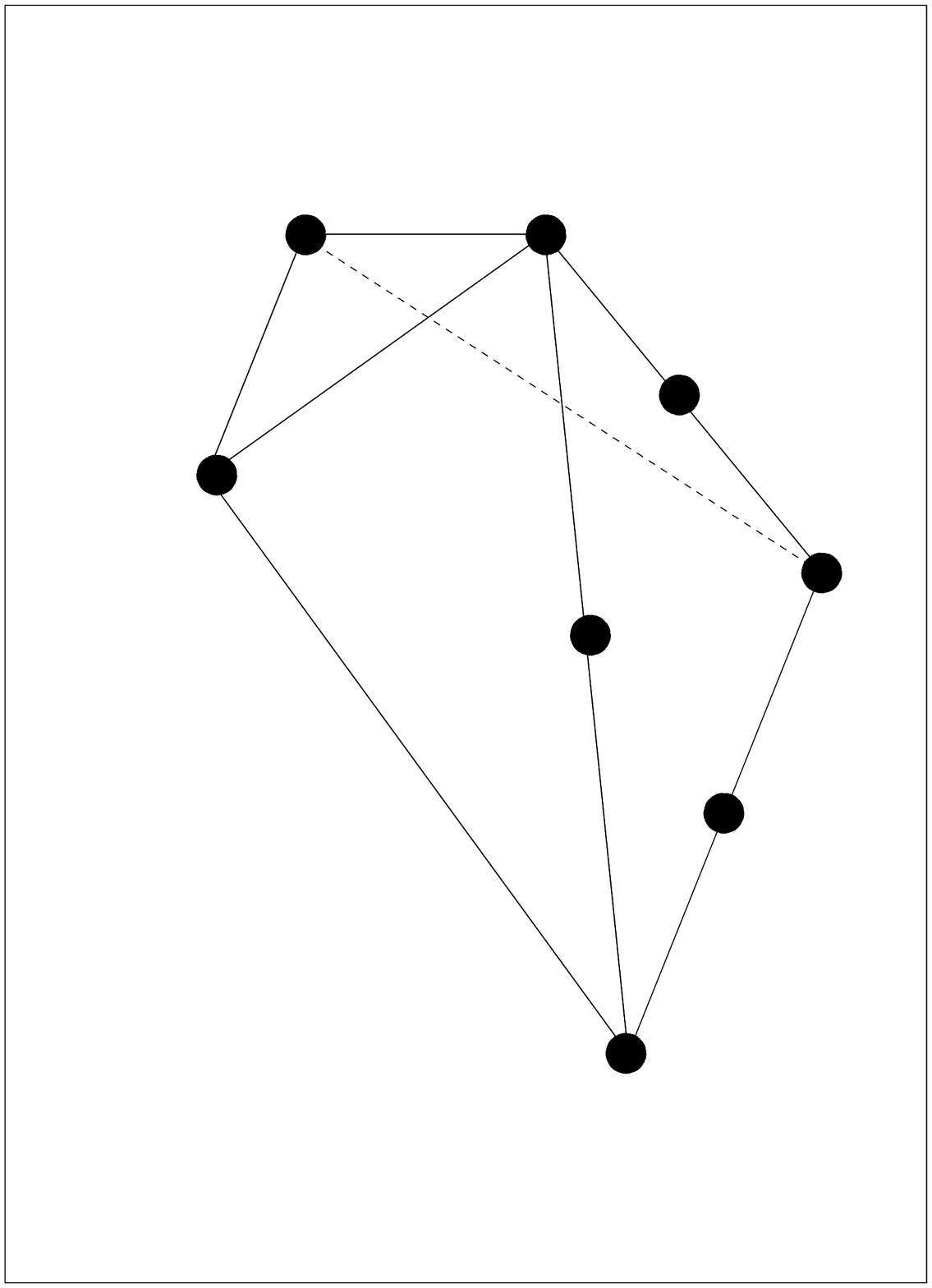, height=4cm}}}
  \hspace*{1.9cm} &
  \rotatebox{90}{{\epsfig{file=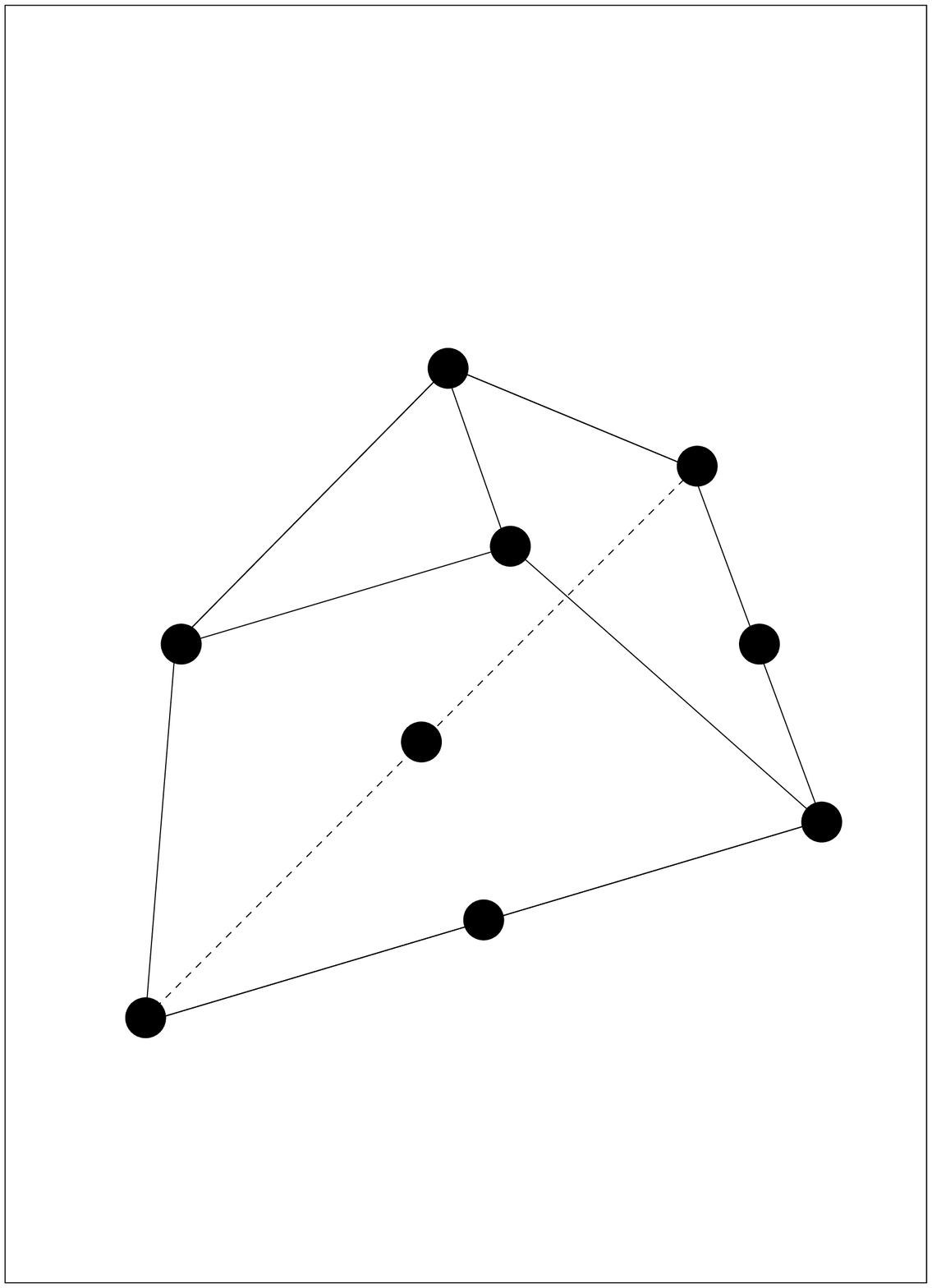, height=4cm}}} \\
  \begin{minipage}[h]{4cm}
    \begin{tabular}[h]{cl}
      \raisebox{.5ex}[-1.5ex]{$P_{13}$}&
      {{\epsfig{file=side1.eps, width=.5cm}} 
      \raisebox{2.5ex}{$\ast$}
    {\epsfig{file=side13.eps, height=1cm}} } \\ \hline
    &$f = (1,5,8,5,1)$ \rule{0em}{1em}\\
    & $  P_{13}^{\ast} =  P_{4} $\\
    & $\nv (P_{13}) = 6$
  \end{tabular}
  \end{minipage}\hspace*{2cm}
   &
  \begin{tabular}[h]{cl}
    \raisebox{.5ex}[-1.5ex]{$P_{14}$} &
    {{\epsfig{file=side13.eps, height=1cm}} 
      \raisebox{2.5ex}{$\ast$}
    {\epsfig{file=side7.eps, width=.5cm}} }\\ \hline
    &$f =(1,6,9,5,1) $ \rule{0em}{1em}\\
    & $P_{14}^{\ast} =  P_{2}  $\\
    & $\nv (P_{14}) = 7$
  \end{tabular}\\  \vspace*{1em}
\end{tabular}

\end{center}

\begin{center}

\hspace*{5.25cm}
\begin{tabular}[h]{lc}
  \hspace*{0.1cm}\rotatebox{90}{{\epsfig{file=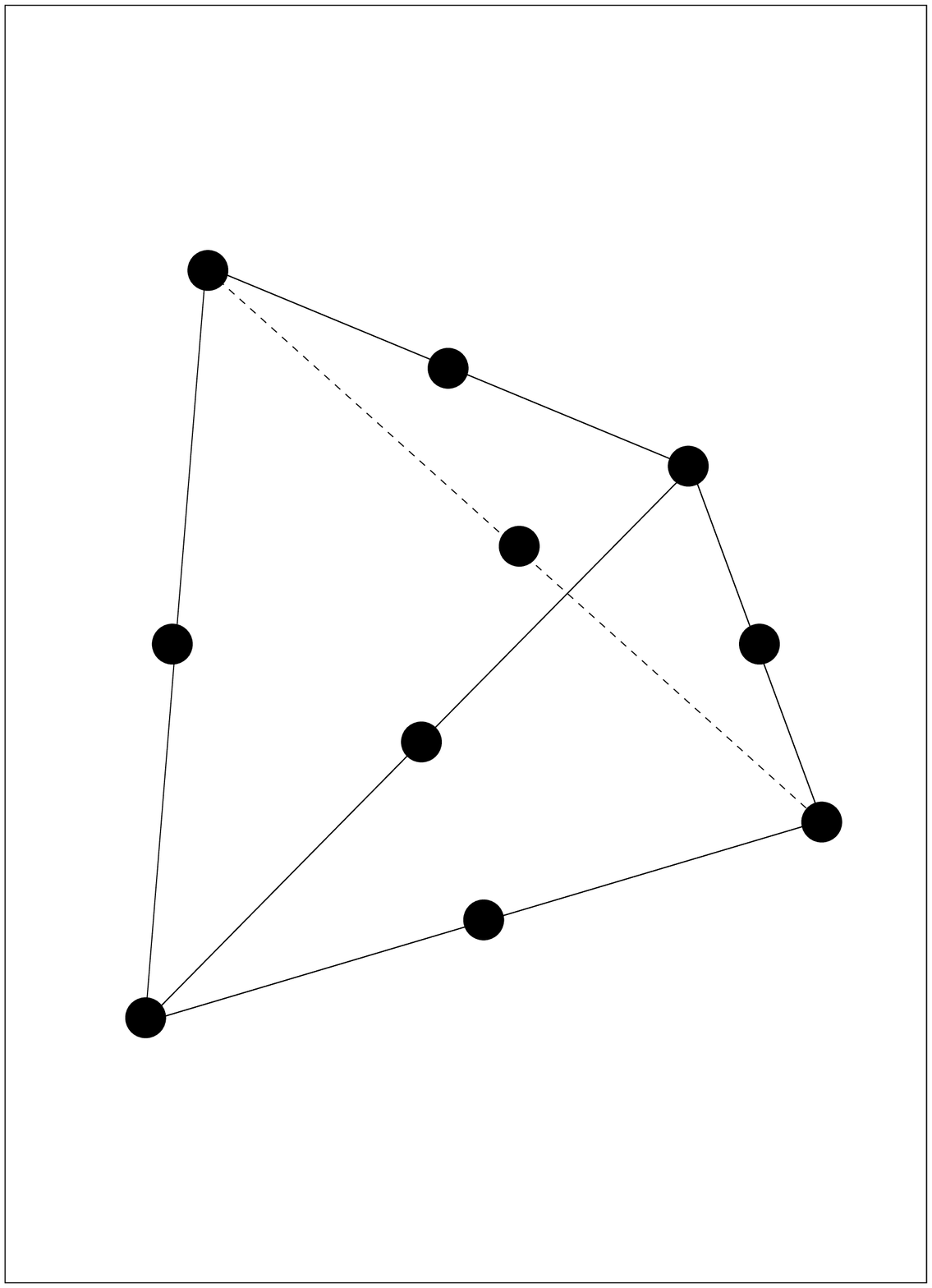, height=4cm}}}\\
 \begin{tabular}[h]{cl}    
    \raisebox{.5ex}[-1.5ex]{$P_{15}$} &  none \\ \hline
    &$f = (1,4,6,4,1)$ \rule{0em}{1em}\\
    & $ P_{15}^{\ast} =  P_{1}  $\\
    & $\nv(P_{15}) = 8$
  \end{tabular} \hspace*{2cm}  &
   \rule{4cm}{0pt}
\end{tabular}

\end{center}

\newpage


\begin{center} 
The list of all $4$-dimensional Gorenstein non-pyramids $Q_i$
of index $2$:

\end{center}
\bigskip

\bigskip

\begin{center}  
\begin{tabular}[h]{cc}
  \begin{tabular}[h]{cl}
    \raisebox{.5ex}[-1.5ex]{$Q_{1}$} &
    {\includegraphics[{height=0.5cm}]{side3.eps}
      \raisebox{1ex}{$\ast$ }
      \includegraphics[{width=.5cm}]{side1.eps}  
      \raisebox{1ex}{$\ast$ }
      \includegraphics[{height=.5cm}]{side2.eps}}\\  \hline
    &$f = (1,6,15,18,9,1)$ \rule{0em}{1em}\\
    & $  Q_{1}^{\ast} = Q_4  $\\
    & $\nv (Q_1) = 3$
  \end{tabular} \hspace*{1cm} &
  \begin{tabular}[h]{cl}
    \raisebox{.5ex}[-1.5ex]{$Q_{2}$}&
    {\includegraphics[{width=1cm}]{side5.eps} 
      \raisebox{1ex}{$\ast$}
      \includegraphics[{height=.5cm}]{side2.eps} 
      \raisebox{1ex}{$\ast$}
      {\includegraphics[{height=.5cm}]{side2.eps}}}\\ \hline
    &$f = (1,7,17,18,8,1)$ \rule{0em}{1em}\\
    & $Q_{2}^{\ast} =  Q_{5}  $\\
    & $\nv (Q_2) = 4$
  \end{tabular}\\  \vspace*{2ex}
\end{tabular}

\begin{tabular}[h]{cc}
  \begin{tabular}[h]{cl}
    \raisebox{.5ex}[-1.5ex]{$Q_{3}$} &
    {\includegraphics[{width=1cm}]{side5.eps} 
      \raisebox{1ex}{$\ast$ }
      \includegraphics[{height=.5cm}]{side2.eps} 
      \raisebox{1ex}{$\ast$ }
       \includegraphics[{height=.5cm}]{side2.eps}}\\  \hline
    &$f = (1,6,13,13,6,1)$ \rule{0em}{1em}\\
    & $  Q_{3}^{\ast} = Q_3  $\\
    & $\nv (Q_3) = 4$
  \end{tabular} \hspace*{1cm}&
  \begin{tabular}[h]{cl}
    \raisebox{.5ex}[-1.5ex]{$Q_{4}$} &
    {\includegraphics[{height=0.5cm}]{side7.eps}
      \raisebox{1ex}{$\ast$ }
      \includegraphics[{width=.5cm}]{side7.eps}
      \raisebox{1ex}{$\ast$ }
      \includegraphics[{height=.5cm}]{side7.eps} }\\  \hline
    &$f = (1,9,18,15,6,1)$ \rule{0em}{1em}\\
    & $  Q_{4}^{\ast} = Q_1  $\\
    & $\nv (Q_4) = 6$
  \end{tabular}\\  \vspace*{2ex}
\end{tabular}

\begin{center}
\hspace*{5.25cm}
\begin{tabular}[h]{cc}
  \begin{tabular}[h]{cl}     
    &{\includegraphics[{height=0.5cm}]{side2.eps} 
      \raisebox{1ex}{$\ast$}
      \includegraphics[{width=.5cm}]{side7.eps} 
      \raisebox{1ex}{$\ast$}
      \includegraphics[{width=.5cm}]{side7.eps}} \\
    \raisebox{3ex}[-1.5ex]{$Q_{5}$} &
    {\includegraphics[{height=.5cm}]{side2.eps}
      \raisebox{1ex}{$\ast$}
      \includegraphics[{width=.5cm}]{side1.eps} 
      \raisebox{1ex}{$\ast$}
      \includegraphics[{height=.5cm}]{side9.eps}}\\ \hline
    &$f = (1,8,18,17,7,1)$ \rule{0em}{1em}\\
    & $Q_{5}^{\ast} =  Q_{2}  $\\
    & $\nv (Q_5) = 5$
  \end{tabular}   & 
  \rule{5.7cm}{0pt}
\end{tabular}

\end{center}

\end{center}

\bigskip

\bigskip

\begin{center}
There exists a single $5$-dimensional Gorenstein (selfdual) 
non-pyramid $R_1$ of index 
$3$:
\end{center}

\begin{center}


\vspace*{1cm}
\begin{tabular}[h]{cc} \hspace*{.6cm}
  \begin{tabular}[h]{cl}
    \raisebox{.5ex}[-1.5ex]{$R_{1}$} &
    {\includegraphics[{width=.5cm}]{side1.eps}
      \raisebox{1ex}{$\ast$ }
      \includegraphics[{width=.5cm}]{side1.eps}
      \raisebox{1ex}{$\ast$ }
      \includegraphics[{height=.5cm}]{side2.eps}
      \raisebox{1ex}{$\ast$}
      \includegraphics[{height=.5cm}]{side2.eps}}\\  \hline
    &$f =(1,8,24,34,24,8,1) $ \rule{0em}{1em}\\
    & $  R_{1}^{\ast} = R_1  $\\
    & $\nv (R_1) = 4$
  \end{tabular} 
\end{tabular}

\end{center}

\section{Toric almost Del Pezzo manifolds}
\label{almostDP}

A smooth $n$-dimensional toric variety $X$ is called {\sl almost Del Pezzo 
manifold} if the anticanonical class $-K_X$ is a semiample Cartier 
divisor which defines a $K_X$-trivial birational morphism 
$\varphi\; : \; X \to X'$, where $X'$ is a Gorenstein toric 
Del Pezzo variety. 
Our classification of Gorenstein polytopes $P$ of degree $2$ 
in \ref{class-GorDP} is equivalent to a biregular classification of 
all $n$-dimensional  Gorenstein toric  Del Pezzo varieties, i.e., of all 
$X'$. 
Let $P$ be a $n$-dimensional Gorenstein polytope of degree $2$. We denote 
by $X_P$ the corresponding toric variety.  
First we need  to understand when a given  
 $n$-dimensional  Gorenstein toric  Del Pezzo variety $X_P$ admits 
a smooth crepant resolution.

\begin{theo} \label{crep}
Let  $P$ a  $n$-dimensional Gorenstein polytope of degree 
$2$ which is not a pyramid over a $(n-1)$-dimensional Gorenstein polytope. 
Then the Gorenstein toric Del Pezzo variety $X_P$ always admits a crepant 
desingularization $\widehat{X_P}$. 
\end{theo} 

\begin{proof}
Since a projective toric variety is smooth if and only if all its 
$0$-dimensional torus orbits are nonsingular points, it is enough 
to analyse the singularities of $0$-dimensional torus orbits in $X_P$ 
which correspond to vertices of the Gorenstein polytope $P$. 
If $ v_i \in P$ such a vertex, then, by \ref{Ci}, the cone $C(v_i)$
generated by $v - v_i$ $(v \in P)$ is dual to the cone $C_{\Gamma_i}$ 
over the dual facet $\Gamma_i \subset P^*$. By \ref{pyr3}, all facets
of $P^*$ are lattice polynomials of degree $\leq 1$. In \cite[Prop. 5.1]{BN07}
it was shown that every lattice polytope of degree $\leq 1$ admits 
a unimodular triangulation. This implies 
existence of smooth crepant resolutions of all $0$-dimensional 
orbits in $X_P$.        
\end{proof}

\begin{exam} 
{\rm The statement in \ref{crep} is not true if $P$ is a pyramid. 
Let $X_{\Pi(P_{15})}$ be $4$-dimensional Gorenstein toric Del Pezzo 
variety which is a cone over the $3$-dimensional 
Gorenstein toric Del Pezzo variety $X_{P_{15}}$. Then the vertex of the   
cone  is a terminal  quotient singularity $\A^4 /\pm id$, which does not 
admit a crepant resolution. } 
\end{exam}

\begin{rem} 
{\rm A smooth crepant desingularization of a Gorenstein 
toric Del Pezzo variety is not uniquely determined, because 
there might be many different unimodular triangualations of lattice polytopes
of degree $1$. For example, $3$-dimensional Gorenstein 
toric Del Pezzo variety $X_{P_{13}}$ admits exactly two different 
crepant desingularizations related by a flop, because all except one 
 facets of the dual polytope $P_4 = P_{13}^*$ admit a unique unimodular
triangualtion and the exceptional face defines an isolated conifold 
singulatity in  $X_{P_{13}}$, which has exactly two different crepant 
resolutions. Another example of a $3$-dimensional  
Gorenstein toric Del Pezzo variety with many different crepant resolutions 
is a $3$-dimensional 
singular cubic $X_{P_2} =\{ x_1x_2 x_3 - x_4 x_5^2 =0\} \subset \PP_4$.} 
\end{rem}

\begin{theo} 
Let $P$ be a $n$-dimensional Gorenstein polytope of degree $2$ and 
let $(1, a_{P^*},1)$ be the $h^*$-vector of the dual Gorenstein polytope 
$P^*$. Then the Picard number of the crepant resolution $\widehat{X_P}$ 
is equal to $a_{P^*} +1$. In particular, if $X_P$ corresponds to a 
maximal Gorenstein polytope $P$, then ${\rm rk}\, {\rm Pic} (\widehat{X_P}) 
\leq 3$. 
\label{picard}
\end{theo} 

\begin{proof} Let $V$ be a smooth projective 
$n$-dimensional toric variety defined by a fan $\Sigma$. Denote 
by $\Sigma^{(1)}$ the set of all $1$-dimensional cones in $\Sigma$. 
The Picard number of $V$ equals $|\Sigma^{(1)}| - n$ \cite{Oda88}. 
Consider now the case $V=\widehat{X_P}$. Then the fan $\Sigma$ defining 
$V$ is a simplicial subdivison of the fan defining the Gorenstein 
toric Del Pezzo variety $X_P$. By \cite{Bat94}, the latter  is the fan 
of cones over proper faces of the dual reflexive polytope 
$((n-1)P)^*$, i.e., $|\Sigma^{(1)}| = \partial ((n-1)P)^* \cap N$. 
By \ref{duality-Gor}, the number of $N$-lattice 
points in $\partial ((n-1)P)^*$ equals 
the number of $N'$-lattice points in the dual Gorenstein polytope 
$P^*$. The latter is equal to $a_{P^*} + n +1$ (see Section 3). So we obtain   
${\rm rk}\, {\rm Pic} (\widehat{X_P}) = a_{P^*} +1$. 
The last statement follows from the duality of Gorenstein polytopes and 
from the inequality $\nv(P^*) = 2 + a_{P^*} \leq 4$ (see \ref{theo-vol4}). 
\end{proof}

\begin{theo} \label{birat}
Let $P$ and $Q$ be two $n$-dimensional Gorenstein polytopes of degree 
$2$ which are not pyramids 
such  that $Q \preccurlyeq P$. Then, up to finitely many  
flops, the almost  Del Pezzo
manifold $\widehat{X_{Q}}$ is obtained from the almost Del Pezzo 
manifold $\widehat{X_{P}}$ by blow ups of 
$\nv(P) - \nv(Q)$ smooth points. 
\end{theo}

\begin{proof} 
Consider $Q$ as a lattice subpolytope of $P$. 
Let $v_i \in P$ be a vertex of $P$ which is not contained in $Q$. 
In the proof of \ref{minim} it was shown that $P_i:= \conv (\{ M \cap P\} 
\setminus \{v_i\})$ is another Gorenstein polytope of degree $2$ containing 
$Q$. Moreover, $P_i$ is obtained from $P$ by cutting out a $n$-dimensional 
basic simplex $S_i$ $(\nv (S_i) =1)$ containg $v_i$ as one of 
its vertices, i.e., 
we have $\nv(P) - \nv(P_i) =1$ and the vertex $v$ corresponds to 
a smooth fixed point $p(v_i) \in X_{P}$. 
Let $\Gamma_i$ be the facet of $S$ which does 
not contain $v_i$. Then $\Gamma_i$ is also 
a facet of $P_i$ and it defines a torus invariant divisor in 
$X_{P_i}$ which is isomorpic to $\PP^{n-1}$, because 
$\Gamma_i \subset S_i$ is 
a $(n-1)$-dimensional basic simplex. By \ref{sum-vol}, 
we have 
\[ 1 =  \nv(P) - \nv(P_i) = \nv(P_i^*)  - \nv(P^*) = a_{P_i^*} - a_{P^*}. \]
Using \ref{picard}, we obtain 
$${\rm rk}\, {\rm Pic}  \widehat{X_{P_i}} - {\rm rk}\, {\rm Pic}  
\widehat{ X_{P}} =1.$$  
Choose an unimodular triangulation  of facets of $P^*$ defining 
a crepant birational morphism $\widehat{X_{P}} \to X_{P}$. We 
can extend this triangulation to an unimodular  triangulation 
of faces of $P_i^* \supset P^*$ which defines a 
crepant birational morphism $\widehat{X_{P_i}} \to X_{P_i}$. 
The compatibility of triangulations defines  a birational 
morphism $\varphi\, : \, \widehat{X_{P_i}} \to  \widehat{X_{P}}$ 
which contracts a divisor $D \cong \PP^{n-1} 
\subset \widehat{X_{P_i}}$ to a smooth 
point $p(v_i)$. Thus, $\widehat{X_{P_i}}$ is a blow up of a smooth 
point in  $\widehat{X_{P}}$. We can repeat the same procedure 
$\nv(P) - \nv (Q)$ times and obtain the statement.      
\end{proof}

Theorems \ref{crep}, \ref{picard}, \ref{birat} show that our combinatorial 
method in the biregular classification of all Gorenstein toric 
Del Pezzo varieties is parallel to the birational method in the 
classification of almost Del Pezzo manifolds due to Jahnke and 
Peternell \cite{JP08}.

\end{document}

%% file: P1Mod.pstex_t
\begin{picture}(0,0)%
\includegraphics{P1Mod.pstex}%
\end{picture}%
\setlength{\unitlength}{3947sp}%
\begingroup\makeatletter\ifx\SetFigFont\undefined%
\gdef\SetFigFont#1#2#3#4#5{%
  \reset@font\fontsize{#1}{#2pt}%
  \fontfamily{#3}\fontseries{#4}\fontshape{#5}%
  \selectfont}%
\fi\endgroup%
\begin{picture}(1155,1716)(1486,-2380)
\put(1501,-2311){\makebox(0,0)[lb]{\smash{{\SetFigFont{12}{14.4}{\rmdefault}{\mddefault}{\updefault}{\color[rgb]{0,0,0}$x_1$}%
}}}}
\put(2626,-1711){\makebox(0,0)[lb]{\smash{{\SetFigFont{11}{13.2}{\rmdefault}{\mddefault}{\updefault}{\color[rgb]{0,0,0}$x_2$}%
}}}}
\put(1576,-811){\makebox(0,0)[lb]{\smash{{\SetFigFont{11}{13.2}{\rmdefault}{\mddefault}{\updefault}{\color[rgb]{0,0,0}$x_3$}%
}}}}
\put(2326,-1261){\makebox(0,0)[lb]{\smash{{\SetFigFont{11}{13.2}{\rmdefault}{\mddefault}{\updefault}{\color[rgb]{0,0,0}$x_4$}%
}}}}
\end{picture}%

%% file: P5Mod.pstex_t
\begin{picture}(0,0)%
\includegraphics{P5Mod.pstex}%
\end{picture}%
\setlength{\unitlength}{3947sp}%
\begingroup\makeatletter\ifx\SetFigFont\undefined%
\gdef\SetFigFont#1#2#3#4#5{%
  \reset@font\fontsize{#1}{#2pt}%
  \fontfamily{#3}\fontseries{#4}\fontshape{#5}%
  \selectfont}%
\fi\endgroup%
\begin{picture}(1455,1711)(1186,-1700)
\put(2626,-586){\makebox(0,0)[lb]{\smash{{\SetFigFont{11}{13.2}{\rmdefault}{\mddefault}{\updefault}{\color[rgb]{0,0,0}$x_5$}%
}}}}
\put(2326,-136){\makebox(0,0)[lb]{\smash{{\SetFigFont{11}{13.2}{\rmdefault}{\mddefault}{\updefault}{\color[rgb]{0,0,0}$x_1$}%
}}}}
\put(2401,-1636){\makebox(0,0)[lb]{\smash{{\SetFigFont{11}{13.2}{\rmdefault}{\mddefault}{\updefault}{\color[rgb]{0,0,0}$x_2$}%
}}}}
\put(1351,-586){\makebox(0,0)[lb]{\smash{{\SetFigFont{11}{13.2}{\rmdefault}{\mddefault}{\updefault}{\color[rgb]{0,0,0}$x_3$}%
}}}}
\put(1201,-1336){\makebox(0,0)[lb]{\smash{{\SetFigFont{11}{13.2}{\rmdefault}{\mddefault}{\updefault}{\color[rgb]{0,0,0}$x_4$}%
}}}}
\end{picture}%

%% file: P2Mod.pstex_t
\begin{picture}(0,0)%
\includegraphics{P2Mod.pstex}%
\end{picture}%
\setlength{\unitlength}{3947sp}%
\begingroup\makeatletter\ifx\SetFigFont\undefined%
\gdef\SetFigFont#1#2#3#4#5{%
  \reset@font\fontsize{#1}{#2pt}%
  \fontfamily{#3}\fontseries{#4}\fontshape{#5}%
  \selectfont}%
\fi\endgroup%
\begin{picture}(1305,1561)(1186,-1175)
\put(1501,-1111){\makebox(0,0)[lb]{\smash{{\SetFigFont{11}{13.2}{\rmdefault}{\mddefault}{\updefault}{\color[rgb]{0,0,0}$x_1$}%
}}}}
\put(2476,-736){\makebox(0,0)[lb]{\smash{{\SetFigFont{11}{13.2}{\rmdefault}{\mddefault}{\updefault}{\color[rgb]{0,0,0}$x_2$}%
}}}}
\put(1201,239){\makebox(0,0)[lb]{\smash{{\SetFigFont{11}{13.2}{\rmdefault}{\mddefault}{\updefault}{\color[rgb]{0,0,0}$x_4$}%
}}}}
\put(2476,-61){\makebox(0,0)[lb]{\smash{{\SetFigFont{11}{13.2}{\rmdefault}{\mddefault}{\updefault}{\color[rgb]{0,0,0}$x_5$}%
}}}}
\put(2026,-961){\makebox(0,0)[lb]{\smash{{\SetFigFont{11}{13.2}{\rmdefault}{\mddefault}{\updefault}{\color[rgb]{0,0,0}$x_3$}%
}}}}
\put(1801,239){\makebox(0,0)[lb]{\smash{{\SetFigFont{11}{13.2}{\rmdefault}{\mddefault}{\updefault}{\color[rgb]{0,0,0}$x_6$}%
}}}}
\end{picture}%

%% file: P4Mod.pstex_t
\begin{picture}(0,0)%
\includegraphics{P4Mod.pstex}%
\end{picture}%
\setlength{\unitlength}{3947sp}%
\begingroup\makeatletter\ifx\SetFigFont\undefined%
\gdef\SetFigFont#1#2#3#4#5{%
  \reset@font\fontsize{#1}{#2pt}%
  \fontfamily{#3}\fontseries{#4}\fontshape{#5}%
  \selectfont}%
\fi\endgroup%
\begin{picture}(1530,1486)(511,-1025)
\put(2026,-586){\makebox(0,0)[lb]{\smash{{\SetFigFont{11}{13.2}{\rmdefault}{\mddefault}{\updefault}{\color[rgb]{0,0,0}$x_4$}%
}}}}
\put(2026,-136){\makebox(0,0)[lb]{\smash{{\SetFigFont{11}{13.2}{\rmdefault}{\mddefault}{\updefault}{\color[rgb]{0,0,0}$x_5$}%
}}}}
\put(1426,-286){\makebox(0,0)[lb]{\smash{{\SetFigFont{11}{13.2}{\rmdefault}{\mddefault}{\updefault}{\color[rgb]{0,0,0}$x_1$}%
}}}}
\put(526,314){\makebox(0,0)[lb]{\smash{{\SetFigFont{11}{13.2}{\rmdefault}{\mddefault}{\updefault}{\color[rgb]{0,0,0}$x_2$}%
}}}}
\put(826,-286){\makebox(0,0)[lb]{\smash{{\SetFigFont{11}{13.2}{\rmdefault}{\mddefault}{\updefault}{\color[rgb]{0,0,0}$x_3$}%
}}}}
\put(1351,-961){\makebox(0,0)[lb]{\smash{{\SetFigFont{11}{13.2}{\rmdefault}{\mddefault}{\updefault}{\color[rgb]{0,0,0}$x_6$}%
}}}}
\end{picture}%

%% file: P6.pstex_t
\begin{picture}(0,0)%
\includegraphics{P6.pstex}%
\end{picture}%
\setlength{\unitlength}{3947sp}%
\begingroup\makeatletter\ifx\SetFigFont\undefined%
\gdef\SetFigFont#1#2#3#4#5{%
  \reset@font\fontsize{#1}{#2pt}%
  \fontfamily{#3}\fontseries{#4}\fontshape{#5}%
  \selectfont}%
\fi\endgroup%
\begin{picture}(1830,1261)(286,-1625)
\put(976,-811){\makebox(0,0)[lb]{\smash{{\SetFigFont{11}{13.2}{\rmdefault}{\mddefault}{\updefault}{\color[rgb]{0,0,0}$x_1$}%
}}}}
\put(1501,-961){\makebox(0,0)[lb]{\smash{{\SetFigFont{11}{13.2}{\rmdefault}{\mddefault}{\updefault}{\color[rgb]{0,0,0}$x_3$}%
}}}}
\put(2101,-1186){\makebox(0,0)[lb]{\smash{{\SetFigFont{11}{13.2}{\rmdefault}{\mddefault}{\updefault}{\color[rgb]{0,0,0}$x_2$}%
}}}}
\put(901,-1561){\makebox(0,0)[lb]{\smash{{\SetFigFont{11}{13.2}{\rmdefault}{\mddefault}{\updefault}{\color[rgb]{0,0,0}$x_5$}%
}}}}
\put(301,-1036){\makebox(0,0)[lb]{\smash{{\SetFigFont{11}{13.2}{\rmdefault}{\mddefault}{\updefault}{\color[rgb]{0,0,0}$x_4$}%
}}}}
\put(976,-511){\makebox(0,0)[lb]{\smash{{\SetFigFont{11}{13.2}{\rmdefault}{\mddefault}{\updefault}{\color[rgb]{0,0,0}$x_6$}%
}}}}
\end{picture}%

%% file: P10Mod.pstex_t
\begin{picture}(0,0)%
\includegraphics{P10Mod.pstex}%
\end{picture}%
\setlength{\unitlength}{3947sp}%
\begingroup\makeatletter\ifx\SetFigFont\undefined%
\gdef\SetFigFont#1#2#3#4#5{%
  \reset@font\fontsize{#1}{#2pt}%
  \fontfamily{#3}\fontseries{#4}\fontshape{#5}%
  \selectfont}%
\fi\endgroup%
\begin{picture}(1305,1411)(511,-1025)
\put(526,-961){\makebox(0,0)[lb]{\smash{{\SetFigFont{11}{13.2}{\rmdefault}{\mddefault}{\updefault}{\color[rgb]{0,0,0}$x_6$}%
}}}}
\put(1801,-961){\makebox(0,0)[lb]{\smash{{\SetFigFont{11}{13.2}{\rmdefault}{\mddefault}{\updefault}{\color[rgb]{0,0,0}$x_5$}%
}}}}
\put(1801,-136){\makebox(0,0)[lb]{\smash{{\SetFigFont{11}{13.2}{\rmdefault}{\mddefault}{\updefault}{\color[rgb]{0,0,0}$x_2$}%
}}}}
\put(1501,239){\makebox(0,0)[lb]{\smash{{\SetFigFont{11}{13.2}{\rmdefault}{\mddefault}{\updefault}{\color[rgb]{0,0,0}$x_4$}%
}}}}
\put(976,-136){\makebox(0,0)[lb]{\smash{{\SetFigFont{11}{13.2}{\rmdefault}{\mddefault}{\updefault}{\color[rgb]{0,0,0}$x_3$}%
}}}}
\put(526,239){\makebox(0,0)[lb]{\smash{{\SetFigFont{11}{13.2}{\rmdefault}{\mddefault}{\updefault}{\color[rgb]{0,0,0}$x_1$}%
}}}}
\end{picture}%

%% file: P11Mod.pstex_t
\begin{picture}(0,0)%
\includegraphics{P11Mod.pstex}%
\end{picture}%
\setlength{\unitlength}{3947sp}%
\begingroup\makeatletter\ifx\SetFigFont\undefined%
\gdef\SetFigFont#1#2#3#4#5{%
  \reset@font\fontsize{#1}{#2pt}%
  \fontfamily{#3}\fontseries{#4}\fontshape{#5}%
  \selectfont}%
\fi\endgroup%
\begin{picture}(1211,1711)(961,-1325)
\put(976,239){\makebox(0,0)[lb]{\smash{{\SetFigFont{11}{13.2}{\rmdefault}{\mddefault}{\updefault}{\color[rgb]{0,0,0}$x_6$}%
}}}}
\put(976,-1261){\makebox(0,0)[lb]{\smash{{\SetFigFont{11}{13.2}{\rmdefault}{\mddefault}{\updefault}{\color[rgb]{0,0,0}$x_4$}%
}}}}
\put(2101,-736){\makebox(0,0)[lb]{\smash{{\SetFigFont{11}{13.2}{\rmdefault}{\mddefault}{\updefault}{\color[rgb]{0,0,0}$x_2$}%
}}}}
\put(2101,164){\makebox(0,0)[lb]{\smash{{\SetFigFont{11}{13.2}{\rmdefault}{\mddefault}{\updefault}{\color[rgb]{0,0,0}$x_3$}%
}}}}
\put(2026,-1186){\makebox(0,0)[lb]{\smash{{\SetFigFont{11}{13.2}{\rmdefault}{\mddefault}{\updefault}{\color[rgb]{0,0,0}$x_5$}%
}}}}
\put(1051,-361){\makebox(0,0)[lb]{\smash{{\SetFigFont{11}{13.2}{\rmdefault}{\mddefault}{\updefault}{\color[rgb]{0,0,0}$x_1$}%
}}}}
\end{picture}%

%% file: ExDualPolytopesDim3.pstex_t
\begin{picture}(0,0)%
\includegraphics{ExDualPolytopesDim3.pstex}%
\end{picture}%
\setlength{\unitlength}{4144sp}%
\begingroup\makeatletter\ifx\SetFigFont\undefined%
\gdef\SetFigFont#1#2#3#4#5{%
  \reset@font\fontsize{#1}{#2pt}%
  \fontfamily{#3}\fontseries{#4}\fontshape{#5}%
  \selectfont}%
\fi\endgroup%
\begin{picture}(2617,1362)(1024,-1300)
\put(1126,-1231){\makebox(0,0)[lb]{\smash{{\SetFigFont{12}{14.4}{\rmdefault}{\mddefault}{\updefault}{\color[rgb]{0,0,0}$P$}%
}}}}
\put(2926,-1231){\makebox(0,0)[lb]{\smash{{\SetFigFont{12}{14.4}{\rmdefault}{\mddefault}{\updefault}{\color[rgb]{0,0,0}$P^*$}%
}}}}
\end{picture}%